\documentclass[a4paper,english,sumlimits,reqno]{amsart}

\usepackage[T1]{fontenc}
\usepackage[utf8]{inputenc}
\usepackage{lmodern}

\usepackage{babel}

\usepackage{xifthen}
\usepackage{ifthenx}

\usepackage{xargs}

\usepackage{graphicx}
\usepackage[table,svgnames,x11names]{xcolor}

\usepackage[colorinlistoftodos,prependcaption,textsize=tiny]{todonotes}
\newcommandx{\todoin}[2][1=]{\todo[inline, caption={todo}, #1]{%
\begin{minipage}{\textwidth-20pt}#2\end{minipage}}}

\usepackage{amssymb}
\usepackage{amsmath}
\usepackage{amsthm}
\usepackage{mathtools}
\usepackage{exscale}
\usepackage{relsize}
\usepackage{bbm}
\usepackage{bm}
\usepackage{amsbsy}
\usepackage{mathdots}

\usepackage[mathcal]{eucal}
\usepackage{mathrsfs}

\usepackage{stmaryrd}

\usepackage[all]{xy}
\newcommand{\vcxymatrix}[1]{\vcenter{\xymatrix{#1}}}

\usepackage{fp}

\usepackage[section]{placeins}
\usepackage{float}

\usepackage{enumerate}

\newcounter{proof}
\newenvironment{myproof}%
{\stepcounter{proof}\begin{proof}}%
{\end{proof}}%
\newcounter{proofstep}[proof]
\newenvironment{proofstep}[1][]%
{\refstepcounter{proofstep}\bigskip\par\noindent%
  \ifthenelse{\isempty{#1}}
    {\textsc{Step \theproofstep. }}
    {\textsc{#1.}}
  \noindent}%
{\par}%
\newcounter{proofcase}[proof]
\newenvironment{proofcase}[1][]%
{\refstepcounter{proofcase}\bigskip\par\noindent%
  \ifthenelse{\isempty{#1}}
    {\textsc{Case \theproofcase. }}
    {\textsc{#1.}}
  \noindent}%
{\par}%

\usepackage{hyperref}

\theoremstyle{plain}
\newtheorem{thm}{Theorem}[section]
\newtheorem*{thm*}{Theorem}

\newtheorem{lem}[thm]{Lemma}
\newtheorem{question}[thm]{Question}
\theoremstyle{definition}

\theoremstyle{remark}

\newtheorem{rem}[thm]{Remark}

\numberwithin{equation}{section}

\newcommandx{\textref}[2][1=]{\hyperref[#2]{#1\ref*{#2}}}
\newcommandx{\textrefp}[2][1=]{(\hyperref[#2]{#1\ref*{#2}})}

\DeclareMathOperator{\union}{\cup}

\DeclareMathOperator{\isect}{\cap}

\newcommand{\dif}{\ensuremath{\, \mathrm d}}

\DeclareMathOperator{\sign}{sign}
\DeclareMathOperator{\Id}{Id}

\DeclareMathOperator{\spn}{span}

\DeclareMathOperator{\cond}{\mathbb{E}}
\DeclareMathOperator{\prob}{\mathbb{P}}

\newcommand{\bmo}{\ensuremath{\mathrm{BMO}}}

\begin{document}

\title[Dimension dependence of factorization problems: bi-parameter Hardy spaces]{Dimension
  dependence of factorization problems: bi-parameter Hardy spaces}

\author[R.~Lechner]{Richard Lechner}

\address{Richard Lechner, Institute of Analysis, Johannes Kepler University Linz, Altenberger
  Strasse 69, A-4040 Linz, Austria}

\email{richard.lechner@jku.at}

\date{\today}



\thanks{Supported by the Austrian Science Foundation (FWF) Pr.Nr.  P28352}

\begin{abstract}
  Given $1 \leq p,q < \infty$ and $n\in\mathbb{N}_0$, let $H_n^p(H_n^q)$ denote the canonical
  finite-dimensional bi-parameter dyadic Hardy space.  Let $(V_n : n\in\mathbb{N}_0)$ denote either
  $\bigl(H_n^p(H_n^q) : n\in\mathbb{N}_0\bigr)$ or
  $\bigl( (H_n^p(H_n^q))^* : n\in\mathbb{N}_0\bigr)$.  We show that the identity operator on $V_n$
  factors through any operator $T : V_N\to V_N$ which has large diagonal with respect to the Haar
  system, where $N$ depends \emph{linearly} on $n$.
\end{abstract}

\maketitle


\makeatletter
\providecommand\@dotsep{5}
\def\listtodoname{List of Todos}
\def\listoftodos{\@starttoc{tdo}\listtodoname}
\makeatother

\section{Introduction}\label{sec:intro}

\noindent
For each $n\in\mathbb{N}$, suppose that $V_n$ has a normalized $1$-unconditional basis $e_j$,
$1\leq j\leq n$, and let $e_j^*\in V_n^*$, $1\leq j \leq n$ denote the associated coordinate
functionals.  This work is concerned with the following question:
\begin{question}\label{que:factor}
  Given $n\in\mathbb{N}$ and $\delta,\Gamma,\eta > 0$, what is the smallest integer
  $N=N(n,\delta,\Gamma,\eta)$, such that for any operator $T : V_N\to V_N$ satisfying
  \begin{equation}\label{eq:que:factor:hypo}
    \|T\|\leq \Gamma
    \qquad\text{and}\qquad
    |\langle e_j^*, T e_j\rangle|
    \geq \delta,
    \quad 1\leq j\leq N,
  \end{equation}
  there are there operators $E : V_n\to V_N$ and $F : V_N\to V_n$, such that the diagram
  \begin{equation}\label{eq:que:factor:diagram}
    \vcxymatrix{V_n \ar[rr]^{\Id_{V_n}} \ar[d]_E && V_n\\
      V_N \ar[rr]_T && V_N \ar[u]_F}
    \qquad \|E\| \|F\| \leq \frac{1+\eta}{\delta}
  \end{equation}
  is commutative?
\end{question}
In numerous Banach spaces, there exist quantitative estimates for $N$ (see
e.g.~\cite{bourgain:1983,bourgain:tzafriri:1987,mueller:1988,blower:1990,wark:2007:direct,wark:2007:class,mueller:2012,lechner:mueller:2015,lechner:2016-factor-mixed,lechner:2017-local-factor-SL,lechner:2018:1-d}).
To illustrate: the estimate for the relationship between $N$ and $n$ is
\begin{itemize}
\item \emph{linear} for $V_n = \ell_n^p$, $1\leq p \leq \infty$ (see
  e.g.~\cite{bourgain:tzafriri:1987});
\item \emph{polynomial} for the one-parameter dyadic Hardy spaces $H_n^p$, $1\leq p < \infty$
  (see~Section~\ref{sec:notation} for the definition of $H_n^p$) and their duals
  (see~\cite{lechner:2018:1-d}).
\end{itemize}
However, in many other Banach spaces the best known estimates for $N$ are often
\emph{super-exponential}.  To illustrate, put $d_n = 2^{n+1}-1$, $n\in\mathbb{N}_0$, and let
$H_n^p(H_n^q)$, $1 \leq p,q < \infty$ denote the bi-parameter mixed norm dyadic Hardy space with
dimension $d_n^2$ (see~Section~\ref{sec:notation} for the definition of $H_n^p(H_n^q)$).  The best known
estimate for $V_{d_n^2} = H_n^p(H_n^q)$ and $V_{d_n^2} = (H_n^p(H_n^q))^*$, $1 \leq p,q < \infty$ is
a \emph{nested exponential} (see~\cite{lechner:2016-factor-mixed}), e.g. of the form
\begin{equation}\label{eq:nest-exp}
  N\leq 2^{8^n 2^{8^{n-1}2^{8^{n-2}2^{8^{n-3}2^{\iddots}}}}}.
\end{equation}
In this work, we use the new \emph{probabilistic method} introduced in~\cite{lechner:2018:1-d}, to
improve the \emph{super-exponential} estimate~\eqref{eq:nest-exp} to the \emph{linear} estimate
\begin{equation}\label{eq:lin-dep}
  N \leq c n,
  \qquad\text{where $c=c(\delta,\Gamma,\eta) > 0$}.
\end{equation}


\section{Notation}\label{sec:notation}

\noindent
Let $\mathcal{D}$ denote the \emph{dyadic intervals} contained in the unit interval~$[0,1)$, i.e.
\begin{equation*}
  \mathcal{D} = \{[(k-1)2^{-n},k2^{-n}) : n\in \mathbb{N}_0, 1\leq k\leq 2^n\}.
\end{equation*}
Let $|\cdot|$ denote the Lebesgue measure.  For any $N\in\mathbb{N}_0$, we put
\begin{equation*}
  \mathcal{D}_N = \{I\in\mathcal{D} : |I|=2^{-N}\}
  \qquad\text{and}\qquad
  \mathcal{D}_{\leq N} = \bigcup_{n=0}^N\mathcal{D}_n.
\end{equation*}
Given $n\in\mathbb{N}_0$ and a dyadic interval $I\in\mathcal{D}_n$, we define
$I^-, I^+\in\mathcal{D}_{n+1}$ by
\begin{equation*}
  I^+ \cup I^- = I
  \qquad\text{and}\qquad
  \inf I^+ < \inf I^-.
\end{equation*}
For any two collections $\mathcal{A},\mathcal{B}\subset\mathcal{D}$, we introduce the following
notation:
\begin{equation*}
  \mathcal{A}\otimes \mathcal{B}
  = \{I\times J: I\in\mathcal{A},\ J\in\mathcal{B}\}.
\end{equation*}

The $L^\infty$-normalized \emph{Haar system} $h_I$, $I\in\mathcal{D}$ is given by
\begin{equation*}
  h_I = \chi_{I^+}-\chi_{I^-},
  \qquad I\in\mathcal{D},
\end{equation*}
where $\chi_A$ denotes the characteristic function of the set $A\subset[0,1)$.  Given
$1\leq p < \infty$, the \emph{one-parameter dyadic Hardy space} $H^p$ is the completion of
\begin{equation*}
  \spn\{ h_I : I \in \mathcal{D} \}
\end{equation*}
under the square function norm
\begin{equation*}
  \|f\|_{H^p}
  = \Big(
  \int_0^1 \big(
  \sum_I |a_I|^2 h_I^2(x)
  \big)^{p/2}
  \dif x
  \Big)^{1/p}
  ,
\end{equation*}
where $f = \sum_I a_I h_I$.  For all $n\in\mathbb{N}_0$, we define the finite-dimensional subspaces
$H_n^p$ of $H^p$ by
\begin{equation*}
  H_n^p = \spn\{h_I : I\in\mathcal{D}_{\leq n}\}.
\end{equation*}

The \emph{bi-parameter} $L^\infty$-normalized \emph{Haar system} $h_{I\times J}$,
$I, J\in\mathcal{D}$ is given by
\begin{equation*}
  h_{I\times J} = h_I\otimes h_J,
  \qquad I, J\in\mathcal{D},
\end{equation*}
where the tensor product of two functions $f,g :[0,1)\to \mathbb{R}$ is defined by
\begin{equation*}
  \bigl(f\otimes g\bigr)(x,y)
  = f(x)g(x),
  \qquad x,y\in [0,1).
\end{equation*}
For $1\leq p,q < \infty$, the \emph{bi-parameter dyadic Hardy space} $H^p(H^q)$ is the completion of
\begin{equation*}
  \spn\{ h_{R} : R \in \mathcal{D}\otimes\mathcal{D} \}
\end{equation*}
under the square function norm
\begin{equation*}
  \|f\|_{H^p(H^q)}
  = \bigg(
  \int_0^1 \Big(
  \int_0^1 \big(
  \sum_{R\in\mathcal{D}\otimes\mathcal{D}} |a_{R}|^2 h_{R}^2(x,y)
  \big)^{q/2}
  \dif y
  \Big)^{p/q}
  \dif x
  \bigg)^{1/p}
  ,
\end{equation*}
where $f = \sum_{R\in\mathcal{D}\otimes\mathcal{D}} a_{R} h_{R}$.  For each $n\in\mathbb{N}_0$, we
define the finite-dimensional subspace $H_n^p(H_n^q)$ of $H^p(H^q)$ by
\begin{equation*}
  H_n^p(H_n^q)
  = \spn\{ h_R : R\in \mathcal{D}_{\leq n}\otimes\mathcal{D}_{\leq n}\}.
\end{equation*}


\section{Main result}\label{sec:results}

\noindent
Recall that we put $d_n = 2^{n+1}-1$, $n\in\mathbb{N}_0$, and let $1\leq p,q < \infty$.  We give a
quantitative estimate for the $N$ appearing in Question~\ref{que:factor} for the spaces
$V_{d_n} = H_n^p(H_n^q)$ and $V_{d_n} = (H_n^p(H_n^q))^*$.  In particular, the relation between $N$
and $n$ is \emph{linear}.
\begin{thm}\label{thm:results:factor}
  Let $1 \leq p,q < \infty$, and let $(V_k : k\in\mathbb{N}_0)$ denote either
  \begin{equation}\label{eq:thm:results:factor:spaces}
    (H_k^p(H_k^q) : k\in\mathbb{N}_0)
    \qquad\text{or}\qquad
    \bigl((H_k^p(H_k^q))^* : k\in\mathbb{N}_0\bigr).
  \end{equation}
  Let $n \in \mathbb{N}_0$ and $\delta,\Gamma,\eta > 0$.  Define the integer
  $N=N(n,\delta,\Gamma,\eta)$ by the formula
  \begin{equation}\label{eq:thm:results:factor:dim}
    N
    = 41(n+3) + \bigl\lfloor
    4\log_2(\Gamma/\delta) + 4\log_2\bigl(1+\eta^{-1}\bigr)
    \bigr\rfloor
    .
  \end{equation}
  Then for any operator $T : V_N\rightarrow V_N$ satisfying
  \begin{equation}\label{eq:thm:results:factor:large}
    \|T\|
    \leq \Gamma
    \qquad\text{and}\qquad
    |\langle T h_Q, h_Q \rangle|
    \geq \delta |Q|,
    \quad Q\in \mathcal{D}_{\leq N}\otimes \mathcal{D}_{\leq N},
  \end{equation}
  there exist bounded linear operators $E : V_n\to V_N$ and $F : V_N\to V_n$, such that the diagram
  \begin{equation}\label{eq:thm:results:factor:diag}
    \vcxymatrix{V_n \ar[rr]^{\Id_{V_n}} \ar[d]_E && V_n\\
      V_N \ar[rr]_T && V_N \ar[u]_F}
    \qquad \|E\|\|F\| \leq \frac{1+\eta}{\delta}
  \end{equation}
  is commutative.
\end{thm}
Note that the linear relation between $N$ and $n$ amounts to a polynomial relation between
the dimensions of the respective spaces; i.e.~$\dim V_N$ is a polynomial in $\dim V_n$.

Formula~\eqref{eq:thm:results:factor:dim} is the main focus of this work.  Specifically, we improve
the previously best known estimate for the relation between $N$ and $n$ in $H_N^p(H_N^q)$ and
$(H_N^p(H_N^q))^*$, $1\leq p,q <\infty$ (see~\cite{lechner:2016-factor-mixed}), \emph{from
  super-exponential to linear} (which means from~\eqref{eq:nest-exp} to~\eqref{eq:lin-dep}).  The
super-exponential growth in~\cite{lechner:2016-factor-mixed} is caused by the use of
\emph{combinatorics}.  The same is true even in one-parameter spaces
(see~e.g.~\cite{mueller:1988,mueller:2012,lechner:mueller:2015,lechner:2017-local-factor-SL}).

Recently, using a \emph{probabilistic} approach (see~\cite{lechner:2018:1-d}), \emph{linear}
estimates for $N$ in $n$ were obtained in the context of one-parameter spaces.  In this work, we
extend this probabilistic method to the bi-parameter spaces $H_N^p(H_N^q)$ and $(H_N^p(H_N^q))^*$,
$1\leq p,q <\infty$, and thereby obtain the formula~\eqref{eq:thm:results:factor:spaces}.


\section{Tensor products, embeddings and projections in mixed norm
  spaces}\label{sec:tensor-products}

\noindent
This section consists of two major parts: The first part connects Jones' compatibility
condition~(\hyperref[enu:j1]{J}) to Capon's local product
condition~\textrefp[P]{enu:p1}--\textrefp[P]{enu:p4}.  In the second part, we show that every
operator on a bi-parameter Hardy space is almost-diagonalized by a properly constructed randomized
block basis.  Both parts are vital components in the proof of our main result
Theorem~\ref{thm:results:factor}.

\subsection{Jones' compatibility condition and Capon's local product condition}\label{sec:jones}

Given $\mathcal{Z}_I\subset\mathcal{D}$, $I\in\mathcal{D}$, we put $Z_I = \bigcup \mathcal{Z}_I$.
We say that the collections $\mathcal{Z}_I$, $I\in\mathcal{D}$ satisfy \emph{Jones' compatibility
  condition~(\hyperref[enu:j1]{J})} (see~\cite{jones:1985}; see also~\cite{mueller:2005}) with
constant $\kappa\geq 1$, if the following four conditions are satisfied:
\begin{enumerate}[(J1)]
\item\label{enu:j1} For each $I\in\mathcal{D}$, the collection $\mathcal{Z}_I$ consists of finitely
  many pairwise disjoint dyadic intervals; moreover,
  $\mathcal{Z}_I\cap \mathcal{Z}_{I'} = \emptyset$, whenever $I,I'\in\mathcal{D}$, $I\neq I'$.

\item\label{enu:j2} For every $I\in\mathcal{D}$, we have that $Z_{I^-}\cup Z_{I^+}\subset Z_I$ and
  $Z_{I^-}\cap Z_{I^+} = \emptyset$.

\item\label{enu:j3} $\kappa^{-1} |I| \leq |Z_I| \leq \kappa |I|$, for all $I\in\mathcal{D}$.

\item\label{enu:j4} For all $I_0,I\in \mathcal{D}$ with $I_0\subset I$ and $K\in \mathcal{Z}_I$, we
  have $\frac{|K\cap Z_{I_0}|}{|K|} \geq \kappa^{-1} \frac{|Z_{I_0}|}{|Z_I|}$.
\end{enumerate}

Jones' compatibility condition~(\hyperref[enu:j1]{J}) is crucial to construct block bases of the
Haar system onto which the natural projection is bounded in $H^1$; especially \textrefp[J]{enu:j4}.
Lemma~\ref{lem:tensor-product} below asserts that the tensor product of collections satisfying Jones'
compatibility condition~(\hyperref[enu:j1]{J}) satisfies Capon's local product
condition~\textrefp[P]{enu:p1}--\textrefp[P]{enu:p4} (see~\cite{laustsen:lechner:mueller:2015}).
Capon's local product condition is used to construct block bases of the bi-parameter Haar system
onto which the natural projection onto that block basis is bounded in $H^p(H^q)$,
$1\leq p,q < \infty$; \textrefp[P]{enu:p4} is crucial for the endpoint spaces $p=1$ or $q=1$.

\begin{lem}\label{lem:tensor-product}
  Let $\mathcal{X}_I\subset\mathcal{D}$, $I\in\mathcal{D}$ and $\mathcal{Y}_J\subset\mathcal{D}$,
  $J\in\mathcal{D}$ both satisfy condition~(\hyperref[enu:j1]{J}) with constant $\kappa\geq 1$.
  Define
  \begin{subequations}\label{eq:lem:tensor-product}
    \begin{equation}\label{eq:lem:tensor-product:a}
      \mathcal{B}_{I\times J}
      = \mathcal{X}_I\otimes\mathcal{Y}_J
      = \{K\times L\, :\, K\in \mathcal{X}_I,\, L\in \mathcal{Y}_J\},
      \qquad I,J\in \mathcal{D},
    \end{equation}
    and put
    \begin{equation}\label{eq:lem:tensor-product:b}
      X_I = \bigcup \mathcal{X}_I,
      \quad I\in\mathcal{D}
      \qquad\text{as well as}\qquad
      Y_J = \bigcup \mathcal{Y}_J,
      \quad J\in \mathcal{D}.
    \end{equation}
  \end{subequations}
  Then $\mathcal{B}_R$, $R\in\mathcal{D}\otimes\mathcal{D}$ satisfies Capon's local product
  condition~\textrefp[P]{enu:p1}--\textrefp[P]{enu:p4} with constants $C_X = C_Y = \kappa$, i.e. the
  following four properties~\textrefp[P]{enu:p1}, \textrefp[P]{enu:p2}, \textrefp[P]{enu:p3}
  and~\textrefp[P]{enu:p4} hold true:
  \begin{enumerate}[(P1)]
  \item\label{enu:p1}%
    For all $R\in \mathcal{D}\otimes\mathcal{D}$ the collection $\mathcal{B}_R$ consists of pairwise
    disjoint dyadic rectangles, and for all $R_0,R_1\in\mathcal{D}\otimes\mathcal{D}$ with
    $R_0\neq R_1$ we have $\mathcal{B}_{R_0} \isect \mathcal{B}_{R_1} = \emptyset$.
  \item\label{enu:p2}%
    For all $I, J, I_0, J_0, I_1, J_1\in\mathcal{D}$ with $I_0 \cap I_1 = \emptyset$,
    $I_0\union I_1\subset I$ and $J_0 \cap J_1 = \emptyset$, $J_0\union J_1\subset J$ we have
    \begin{align*}
      X_{I_0}\isect X_{I_1} &= \emptyset, &X_{I_0}\union X_{I_1} & \subset X_I,\\
      Y_{J_0}\isect Y_{J_1} &= \emptyset, &Y_{J_0}\union Y_{J_1} & \subset Y_J.
    \end{align*}
  \item\label{enu:p3}%
    For every $I, J\in\mathcal{D}$ we have
    \begin{equation*}
      \kappa^{-1} |I|
      \leq |X_I|
      \leq \kappa |I|
      \qquad\text{and}\qquad
      \kappa^{-1} |J|
      \leq |Y_J|
      \leq \kappa |J|.
    \end{equation*}
  \item\label{enu:p4}%
    For all $I_0, J_0, I, J\in \mathcal{D}$ with $I_0\subset I$, $J_0\subset J$ and for every
    $K\in \mathcal{X}_I$, $L\in \mathcal{Y}_{J}$, we have
    \begin{equation*}
      \frac{|K\isect X_{I_0}|}{|K|} \geq  \kappa^{-1}\frac{|X_{I_0}|}{|X_I|}
      \qquad\text{and}\qquad
      \frac{|L\isect Y_{J_0}|}{|L|} \geq  \kappa^{-1}\frac{|Y_{J_0}|}{|Y_J|}.
    \end{equation*}
  \end{enumerate}
\end{lem}

\begin{proof}
  \textrefp[P]{enu:p1}--\textrefp[P]{enu:p4} follow directly
  from~\textrefp[J]{enu:j1}--\textrefp[J]{enu:j4}.
\end{proof}

\begin{rem}
  The conditions \textrefp[P]{enu:p1}--\textrefp[P]{enu:p4} were introduced
  in~\cite{laustsen:lechner:mueller:2015} in a more general form: the collections
  $\mathcal{B}_{I\times J}$, $I,J\in\mathcal{D}$ in~\cite{laustsen:lechner:mueller:2015} have
  \emph{local product structure}, i.e. there exist collections $\mathcal{X}_{I\times J}$,
  $\mathcal{Y}_{I\times J}$, $I,J\in\mathcal{D}$ such that
  \begin{equation}\label{eq:local-product}
    \mathcal{B}_{I\times J}
    = \{K\times L\, :\, K\in \mathcal{X}_{I\times J},\, L\in \mathcal{Y}_{I\times J}\},
    \qquad I,J\in \mathcal{D}.
  \end{equation}
  
  In Lemma~\ref{lem:tensor-product}, we have a special case of~\eqref{eq:local-product}: \emph{true
    product structure} (see~\eqref{eq:lem:tensor-product}).  To highlight the distinction
  explicitly, in Lemma~\ref{lem:tensor-product} we have that $\mathcal{X}_{I\times J}$ does not depend on
  $J$ and that $\mathcal{Y}_{I\times J}$ does not depend on $I$.
\end{rem}

\begin{lem}\label{lem:block:basic-estimate}
  Let $\mathcal{X}$ and $\mathcal{Y}$ each a denote non-empty, finite collection of pairwise
  disjoint dyadic intervals, and define $X = \bigcup\mathcal{X}$ as well as
  $Y = \bigcup\mathcal{Y}$.  Given $\theta,\varepsilon \in \{\pm 1\}^{\mathcal{D}}$, put
  \begin{equation}\label{lem:block:basic-estimate:func}
    b^{(\theta,\varepsilon)}
    = \sum_{\substack{K\in\mathcal{X}\\L\in\mathcal{Y}}} \theta_K\varepsilon_L h_{K\times L}.
  \end{equation}
  Then:
  \begin{equation}
    \label{eq:lem:block:basic-estimate}
    \|b^{(\theta,\varepsilon)}\|_{H^p(H^q)} = |X|^{1/p}|Y|^{1/q}
    \qquad\text{and}\qquad
    \|b^{(\theta,\varepsilon)}\|_{(H^p(H^q))^*} = |X|^{1/{p'}}|Y|^{1/{q'}},
  \end{equation}
  where $1\leq p,q < \infty$, $1< p',q' \leq \infty$ with
  $\frac{1}{p}+\frac{1}{{p'}} = \frac{1}{q}+\frac{1}{{q'}} = 1$.
\end{lem}

Lemma~\ref{lem:block:basic-estimate} follows immediately
from~\cite[Lemma~4.1]{laustsen:lechner:mueller:2015}.  Since the proof is short, we include it here
for the sake of completeness.
\begin{proof}
  \eqref{lem:block:basic-estimate:func} and the disjointness of the collections
  $\mathcal{X},\mathcal{Y}$ yields
  \begin{align*}
    \|b^{(\theta,\varepsilon)}\|_{H^p(H^q)}
    &= \biggl(
      \int_0^1 \Bigl(
      \int_0^1 \sum_{\substack{K\in\mathcal{X}\\L\in\mathcal{Y}}} h_K^2(x) h_L^2(y) \dif y
    \Bigr)^{p/q} \dif x
    \biggr)^{1/p}\\
      &= |Y|^{1/q} \biggl(
        \int_0^1  \sum_{K\in\mathcal{X}} h_K^2(x) \dif x
        \biggr)^{1/p}
        = |X|^{1/p} |Y|^{1/q}.
  \end{align*}
  
  We will now compute $\|b^{(\theta,\varepsilon)}\|_{(H^p(H^q))^*}$.  To this end, let
  $h\in H^p(H^q)$ be given by $h = \sum_{K,L} a_{K\times L} h_{K\times L}\in H^p(H^q)$, and observe
  that by Hölder's inequality we obtain
  \begin{align*}
    \langle b^{(\theta,\varepsilon)}, h \rangle
    &\leq \sum_{K\in\mathcal{X}} |K| \sum_{L\in\mathcal{Y}} |a_{K\times L}| |L|
      \leq |Y|^{1/{q'}} \sum_{K\in\mathcal{X}} |K| \bigl(
      \sum_{L\in\mathcal{Y}} |a_{K\times L}|^q |L|
      \bigr)^{1/q} \\
    &\leq |X|^{1/{p'}}|Y|^{1/{q'}} \Bigl(\sum_{K\in\mathcal{X}} |K| \bigl(
      \sum_{L\in\mathcal{Y}} |a_{K\times L}|^q |L|
      \bigr)^{p/q}\Bigr)^{1/p}\\
    &= |X|^{1/{p'}}|Y|^{1/{q'}} \|h\|_{H^p(H^q)}.
  \end{align*}
  Thus, we have $\|b^{(\theta,\varepsilon)}\|_{(H^p(H^q))^*}\leq |X|^{1/{p'}}|Y|^{1/{q'}}$.  Since
  $\langle b^{(\theta,\varepsilon)}, b^{(\theta,\varepsilon)} \rangle = |X| |Y|$ and
  $\|b^{(\theta,\varepsilon)}\|_{H^p(H^q)} = |X|^{1/{p}}|Y|^{1/{q}}$ by the first part of the proof,
  we obtain
  \begin{equation*}
    \|b^{(\theta,\varepsilon)}\|_{(H^p(H^q))^*}
    = |X|^{1/{p'}} |Y|^{1/{q'}}.
    \qedhere
  \end{equation*}
\end{proof}

The following Theorem~\ref{thm:projection} is one of the two main ingredients in the proof of
Theorem~\ref{thm:results:factor}; the other one is the almost-diagonalization of operators using random
block bases (see~Theorem~\ref{thm:var}).
\begin{thm}\label{thm:projection}
  Let $\mathcal{X}_I\subset\mathcal{D}$, $I\in\mathcal{D}$ and $\mathcal{Y}_J\subset\mathcal{D}$,
  $J\in\mathcal{D}$ both satisfy condition~(\hyperref[enu:j1]{J}) with constant $\kappa = 1$, and
  define the product collections
  \begin{equation}\label{eq:thm:projection:coll-B}
    \mathcal{B}_{I\times J}
    = 
    \mathcal{X}_I\otimes\mathcal{Y}_J
    = \{K\times L\, :\, K\in \mathcal{X}_I,\, L\in \mathcal{Y}_J\},
    \qquad I,J\in \mathcal{D}.
  \end{equation}
  Given $\theta,\varepsilon \in \{\pm 1\}^{\mathcal{D}}$, we define the tensor product system
  \begin{equation}\label{eq:thm:projection:tensor:1}
    b_{I\times J}^{(\theta,\varepsilon)}
    = f_I^{(\theta)}\otimes g_J^{(\varepsilon)},
    \qquad I,J\in \mathcal{D},
  \end{equation}
  where
  \begin{equation}\label{eq:thm:projection:tensor:2}
    f_I^{(\theta)}
    = \sum_{K\in\mathcal{X}_I} \theta_K h_K,
    \quad I\in \mathcal{D}
    \qquad\text{and}\qquad
    g_J^{(\varepsilon)}
    = \sum_{L\in\mathcal{Y}_J} \varepsilon_L h_L,
    \quad J\in \mathcal{D}.
  \end{equation}
  Given $1 \leq p,q < \infty$, let $V$ denote either $H^p(H^q)$ or $(H^p(H^q))^*$.  Then the
  operators $B^{(\theta,\varepsilon)}, A^{(\theta,\varepsilon)} : V\to V$ given by
  \begin{equation}\label{eq:thm:projection:operators}
    B^{(\theta,\varepsilon)} f = \sum_{R\in \mathcal{D}\otimes\mathcal{D}} \frac{\langle f, h_R\rangle}{\|h_R\|_2^2} b_R^{(\theta,\varepsilon)}
    \qquad\text{and}\qquad
    A^{(\theta,\varepsilon)} f = \sum_{R\in \mathcal{D}\otimes\mathcal{D}} \frac{\langle f, b_R^{(\theta,\varepsilon)}\rangle}{\|b_R^{(\theta,\varepsilon)}\|_2^2} h_R
  \end{equation}
  satisfy the estimates
  \begin{equation}\label{eq:thm:projection:estimates}
    \begin{aligned}
      \|B^{(\theta,\varepsilon)} f \|_{V} & \leq \|f\|_{V},
      &f&\in V,\\
      \|A^{(\theta,\varepsilon)} f \|_{V} &\leq \|f\|_{V}, &f&\in V.
    \end{aligned}
  \end{equation}
  Moreover, the diagram
  \begin{equation}\label{eq:thm:projection:diagram}
    \vcxymatrix{V \ar[rr]^{I_{V}} \ar[rd]_{B^{(\theta,\varepsilon)}} & & V\\
      &  V  \ar[ru]_{A^{(\theta,\varepsilon)}} &
    }
  \end{equation}
  is commutative and the composition
  $P^{(\theta,\varepsilon)} = B^{(\theta,\varepsilon)} A^{(\theta,\varepsilon)}$ is the norm $1$
  projection $P^{(\theta,\varepsilon)} : V\to V$ given by
  \begin{equation}\label{eq:thm:projection:proj-formula}
    P^{(\theta,\varepsilon)}(f)
    = \sum_{R\in\mathcal{D}\otimes\mathcal{D}} \frac{\langle f, b_R^{(\theta,\varepsilon)}\rangle}{\|b_R\|_2^2}
    b_R^{(\theta,\varepsilon)}.
  \end{equation}
  Consequently, the range of $B^{(\theta,\varepsilon)}$ is complemented (by
  $P^{(\theta,\varepsilon)}$), and $B^{(\theta,\varepsilon)}$ is an isometric isomorphism onto its
  range.
\end{thm}

\begin{proof}
  The case $V = H^p(H^q)$ follows immediately from Lemma~\ref{lem:tensor-product}
  and~\cite{laustsen:lechner:mueller:2015}.

  If $V = (H^p(H^q))^*$ the theorem follows from the case $V = H^p(H^q)$ and the observation that
  $(B^{(\theta,\varepsilon)})^* = A^{(\theta,\varepsilon)}$ and
  $(A^{(\theta,\varepsilon)})^* = B^{(\theta,\varepsilon)}$.
\end{proof}

\subsection{Random block bases with tensor product structure}\label{sec:rand-tens-prod}

\noindent
Let $\prob_\theta$ denote the uniform measure on $\Omega_\theta = \{\pm 1\}^{\mathcal{D}}$, and let
$(\Omega_\varepsilon,\prob_\varepsilon)$ denote an independent copy of
$(\Omega_\theta,\prob_\theta)$.  $\prob_{\theta,\varepsilon}$ is the product measure on
$\Omega_\theta\times\Omega_\varepsilon$.  Moreover, $\cond_\theta$, $\cond_\varepsilon$ and
$\cond_{\theta,\varepsilon}$ are the expectations with respect to the probability measures
$\prob_\theta$, $\prob_\varepsilon$ and $\prob_{\theta,\varepsilon}$, respectively.

Given $n,N\in\mathbb{N}$, $I,J\in\mathcal{D}_{\leq n}$ and
$\mathcal{X}_I, \mathcal{Y}_J\subset\mathcal{D}_{\leq N}$, define the functions
\begin{equation}\label{eq:fg:dfn}
  f_I^{(\theta)} = \sum_{K\in\mathcal{X}_I} \theta_K h_K,
  \quad \theta\in\Omega_\theta
  \qquad\text{and}\qquad
  g_J^{(\varepsilon)} = \sum_{L\in\mathcal{Y}_J}\varepsilon_L h_L,
  \quad \varepsilon\in\Omega_\varepsilon.
\end{equation}
Hence, their tensor product $b_{I\times J}^{(\theta,\varepsilon)}$ is given by
\begin{equation}\label{eq:b:dfn}
  b_{I\times J}^{(\theta,\varepsilon)}
  = f_I^{(\theta)}\otimes g_J^{(\varepsilon)}
  = \sum_{\substack{K\in\mathcal{X}_I\\L\in\mathcal{Y}_J}} \theta_K \varepsilon_L h_{K\times L},
  \qquad (\theta,\varepsilon)\in\Omega_\theta\times\Omega_\varepsilon.
\end{equation}

Let $1 \leq p,q < \infty$ and let $V_N$ denote either $H_N^p(H_N^q)$ or
$(H_N^p(H_N^q))^*$.  Given a bounded linear operator $T : V_N\to V_N$, we put
\begin{subequations}\label{eq:dfn:rv}
  \begin{align}
    W_{I,I',J,J'}(\theta,\varepsilon)
    & = \langle T b_{I\times J}^{(\theta,\varepsilon)},
      b_{I'\times J'}^{(\theta,\varepsilon)}\rangle,
      \qquad I, J, I', J'\in\mathcal{D}_{\leq n},\ I\neq I',\ J\neq J',
      \label{eq:dfn:rv:W}\\
    X_{I,I',J}(\theta,\varepsilon)
    & = \langle T b_{I\times J}^{(\theta,\varepsilon)},
      b_{I'\times J'}^{(\theta,\varepsilon)}\rangle,
      \qquad I, J, I'\in\mathcal{D}_{\leq n},\ I\neq I',
      \label{eq:dfn:rv:X}\\
    Y_{I,J,J'}(\theta,\varepsilon)
    & = \langle T b_{I\times J}^{(\theta,\varepsilon)},
      b_{I'\times J'}^{(\theta,\varepsilon)}\rangle,
      \qquad I, J, J'\in\mathcal{D}_{\leq n},\ J\neq J',
      \label{eq:dfn:rv:Y}\\
    Z_{I,J}(\theta,\varepsilon)
    & = \langle T b_{I\times J}^{(\theta,\varepsilon)},
      b_{I\times J}^{(\theta,\varepsilon)}\rangle -
      \sum_{\substack{K\in\mathcal{X}_I\\L\in\mathcal{Y}_J}}
    \langle Th_{K\times L}, h_{K\times L}\rangle,
    \quad I, J\in\mathcal{D}_{\leq n},
    \label{eq:dfn:rv:Z}
  \end{align}
  for all $(\theta,\varepsilon)\in\Omega_\theta\times\Omega_\varepsilon$.
\end{subequations}
From here on, we will regularly omit the subindices of the above random variables,
i.e.~$W=W_{I,I',J,J'}$, $X=X_{I,I',J}$, $Y=Y_{I,J,J'}$ and $Z = Z_{I,J}$.

\begin{thm}\label{thm:var}
  Given $n,N\in\mathbb{N}$, let $\mathcal{X}_I\subset\mathcal{D}_{\leq N}$, $I\in\mathcal{D}_{\leq n}$ and
  $\mathcal{Y}_J\subset\mathcal{D}_{\leq N}$, $J\in\mathcal{D}_{\leq n}$ both denote non-empty collections which
  satisfy~\textrefp[J]{enu:j1}.  Define $\alpha > 0$ by putting
  \begin{equation}\label{eq:alpha-small}
    \alpha = \max \{ |K|, |L| : K\in\mathcal{X}_I,\ L\in\mathcal{Y}_J,\ I,J\in\mathcal{D}_{\leq n}\}.
  \end{equation}
  Given $1 \leq p,q < \infty$, let $V_N$ denote either $H_N^p(H_N^q)$ or $(H_N^p(H_N^q))^*$.  Then
  for any bounded operator $T : V_N\to V_N$ we have
  \begin{equation}\label{eq:lem:var:exp}
    \cond_{\theta,\varepsilon} W
    = \cond_{\theta,\varepsilon} X
    = \cond_{\theta,\varepsilon} Y
    = \cond_{\theta,\varepsilon} Z
    = 0,
  \end{equation}
  as well as the estimates
  \begin{subequations}\label{eq:lem:var:var}
    \begin{align}
      \cond_{\theta,\varepsilon} W^2 & \leq \|T\|^2 \alpha^{1/2},
      & \cond_{\theta,\varepsilon} X^2 & \leq 4 \|T\|^2 \alpha^{1/2},\\
      \cond_{\theta,\varepsilon} Y^2 & \leq 4 \|T\|^2 \alpha^{1/2},
      & \cond_{\theta,\varepsilon} Z^2 & \leq 12 \|T\|^2 \alpha^{1/2},
    \end{align}
  \end{subequations}
  where the random variables $W,X,Y,Z$ are defined in~\eqref{eq:dfn:rv}.
\end{thm}
The proof is given in Section~\ref{sec:lem-proof}.


\section{Proof of the main result Theorem~\ref{thm:results:factor}}\label{sec:factor}

\noindent
Here we prove our main result Theorem~\ref{thm:results:factor}, by extending the probabilistic method
introduced in~\cite[Theorem~3.1]{lechner:2018:1-d} for one-parameter Hardy spaces $H^p$, to the
bi-parameter Hardy spaces $H_N^p(H_N^q)$.  The proof heavily relies on the results of
Section~\ref{sec:tensor-products}.

For convenience of the reader we repeat Theorem~\ref{thm:results:factor} here.
\begin{thm*}[Main result Theorem~\ref{thm:results:factor}]\label{thm:factor}
  Let $1 \leq p,q < \infty$, and let $(V_k : k\in\mathbb{N}_0)$ denote either
  \begin{equation}\label{eq:thm:factor:spaces}
    (H_k^p(H_k^q) : k\in\mathbb{N}_0)
    \qquad\text{or}\qquad
    \bigl((H_k^p(H_k^q))^* : k\in\mathbb{N}_0\bigr).
  \end{equation}
  Let $n \in \mathbb{N}_0$ and $\delta,\Gamma,\eta > 0$.  Define the integer
  $N=N(n,\delta,\Gamma,\eta)$ by the formula
  \begin{equation}\label{eq:thm:factor:dim}
    N
    = 41(n+3) + \bigl\lfloor
    4\log_2(\Gamma/\delta) + 4\log_2\bigl(1+\eta^{-1}\bigr)
    \bigr\rfloor
    .
  \end{equation}
  Then for any operator $T : V_N\rightarrow V_N$ satisfying
  \begin{equation}\label{eq:thm:factor:large}
    \|T\|
    \leq \Gamma
    \qquad\text{and}\qquad
    |\langle T h_Q, h_Q \rangle|
    \geq \delta |Q|,
    \quad Q\in \mathcal{D}_{\leq N}\otimes \mathcal{D}_{\leq N},
  \end{equation}
  there exist bounded linear operators $E : V_n\to V_N$ and $F : V_N\to V_n$, such that the diagram
  \begin{equation}\label{eq:thm:factor:diag}
    \vcxymatrix{V_n \ar[rr]^{\Id_{V_n}} \ar[d]_E && V_n\\
      V_N \ar[rr]_T && V_N \ar[u]_F}
    \qquad \|E\|\|F\| \leq \frac{1+\eta}{\delta}
  \end{equation}
  is commutative.
\end{thm*}

\begin{myproof}
  Let $M : V_N\to V_N$ denote the norm $1$ multiplication operator given by the linear extension of
  \begin{equation*}
    h_Q\mapsto \sign(\langle T h_Q, h_Q\rangle) h_Q,
    \qquad Q\in\mathcal{D}_{\leq N}\otimes\mathcal{D}_{\leq N}.
  \end{equation*}
  By~\eqref{eq:thm:factor:large}, we obtain
  \begin{equation*}
    \langle T M h_Q, h_Q \rangle
    = |\langle T h_Q, h_Q \rangle|
    \geq \delta |Q|,
    \qquad Q\in\mathcal{D}_{\leq N}\otimes\mathcal{D}_{\leq N},
  \end{equation*}
  and therefore we can assume
  \begin{equation}
    \label{eq:proof:thm:factor:large}
    \langle T h_Q, h_Q \rangle\geq \delta|Q|,
    \qquad Q\in\mathcal{D}_{\leq N}\otimes\mathcal{D}_{\leq N}.
  \end{equation}

  Before we proceed to \textref[Step~]{step:proof:thm:factor:1} of the proof, we define the
  constants $m_0$ and $\eta_0$: Let $m_0\in\mathbb{N}_0$ denote the smallest integer such that
  \begin{equation}\label{proof:thm:factor:const:1}
    2^{m_0} > \frac{2^{8(n+3)}\Gamma^4}{\eta_0^4},
    \qquad\text{where}\quad
    \eta_0 = \frac{\eta\delta}{(1+\eta)2^{8(n+2)}}.
  \end{equation}
  
  \begin{proofstep}[Step~\theproofstep: constructing the block basis $b_R^{(\theta,\varepsilon)}$,
    $R\in\mathcal{D}_{\leq n}\otimes\mathcal{D}_{\leq n}$]\label{step:proof:thm:factor:1}
    In this step, we will define a random block basis
    $(\theta,\varepsilon)\mapsto b_R^{(\theta,\varepsilon)}$,
    $R\in\mathcal{D}_{\leq n}\otimes\mathcal{D}_{\leq n}$ of the Haar system $h_Q$,
    $Q\in\mathcal{D}_{\leq N}\otimes\mathcal{D}_{\leq N}$ given by
    \begin{equation}\label{proof:thm:factor:overview:0}
      b_{I\times J}^{(\theta,\varepsilon)}
      = f_I^{(\theta)}\otimes g_J^{(\varepsilon)}
      = \sum_{K\in\mathcal{X}_I} \theta_K h_K
      \otimes \sum_{L\in\mathcal{Y}_J} \varepsilon_L h_L,
      \qquad \theta\in\Omega_\theta,\ \varepsilon\in\Omega_\varepsilon,
    \end{equation}
    where $\mathcal{X}_I\subset\mathcal{D}_{\leq N}$, $I\in\mathcal{D}_{\leq n}$ and
    $\mathcal{Y}_J\subset\mathcal{D}_{\leq N}$, $J\in\mathcal{D}_{\leq n}$ both satisfy
    condition~(\hyperref[enu:j1]{J}) with constant $\kappa=1$.  The collections will be selected by
    a minimalist Gamlen-Gaudet construction.  Then, using Theorem~\ref{thm:var}, we will find signs
    $(\theta,\varepsilon) \in \Omega_\theta\times\Omega_\varepsilon$ such that
    \begin{subequations}\label{proof:thm:factor:overview:1}
      \begin{align}
        |\langle T b_R^{(\theta,\varepsilon)},
        b_{R'}^{(\theta,\varepsilon)}\rangle|
        & \leq \eta_0,
        & R,R'\in\mathcal{D}_{\leq n}\otimes\mathcal{D}_{\leq n},\ R\neq R',&
                                                                              \label{proof:thm:factor:overview:1:a}\\
        \langle T b_R^{(\theta,\varepsilon)},
        b_R^{(\theta,\varepsilon)}\rangle
        & \geq (\delta - \eta_0 2^{2n}) \|b_R^{(\theta,\varepsilon)}\|_2^2,
        & R\in\mathcal{D}_{\leq n}\otimes\mathcal{D}_{\leq n}.&
                                                                \label{proof:thm:factor:overview:1:b}
      \end{align}
    \end{subequations}

    We will now inductively define the collections $\mathcal{X}_I$, $I\in\mathcal{D}_{\leq n}$ and
    $\mathcal{Y}_J$, $J\in\mathcal{D}_{\leq n}$.  We begin by putting,
    \begin{equation}\label{eq:proof:thm:factor:coll:1}
      \mathcal{X}_{[0,1)}
      = \mathcal{Y}_{[0,1)}
      = \mathcal{D}_{m_0}.
    \end{equation}
    Now, let $0\leq k \leq n-1$, assume that we have already constructed the collections
    $\mathcal{X}_I$, $I\in\mathcal{D}_{\leq k}$ and $\mathcal{Y}_J$, $J\in\mathcal{D}_{\leq k}$.
    Then we define
    \begin{subequations}\label{eq:proof:thm:factor:coll:2}
      \begin{align}
        \mathcal{X}_{I^+}
        &= \{ K^+ : K\in\mathcal{X}_I\},
        &\mathcal{X}_{I^-}
        &= \{ K^- : K\in\mathcal{X}_I\},
        &&I\in\mathcal{D}_k,
           \label{eq:proof:thm:factor:coll:2:a}\\
        \mathcal{Y}_{J^+}
        &= \{ K^+ : K\in\mathcal{Y}_J\},
        &\mathcal{Y}_{J^-}
        &= \{ K^- : K\in\mathcal{Y}_J\},
        &&J\in\mathcal{D}_k.
           \label{eq:proof:thm:factor:coll:2:a}
      \end{align}
    \end{subequations}
    Clearly, $\mathcal{X}_I$, $I\in\mathcal{D}_{\leq n}$ and $\mathcal{Y}_J$,
    $J\in\mathcal{D}_{\leq n}$ both satisfy condition~(\hyperref[enu:j1]{J}) with constant
    $\kappa = 1$.

    Next, we will use the probabilistic Theorem~\ref{thm:var} to find signs
    $(\theta,\varepsilon)\in\Omega_\theta\times\Omega_\varepsilon$ such
    that~\eqref{proof:thm:factor:overview:1} is satisfied.  To this end, we define the off-diagonal
    events
    \begin{equation*}
      O_{R,R'} = \bigl\{ (\theta,\varepsilon) :
      |\langle T b_R^{(\theta,\varepsilon)}, b_{R'}^{(\theta,\varepsilon)}\rangle|
      > \eta_0
      \bigr\},
      \qquad R,R'\in\mathcal{D}_{\leq n}\otimes\mathcal{D}_{\leq n},\ R\neq R'
    \end{equation*}
    and the diagonal events
    \begin{equation*}
      D_{I,J} = \biggl\{ (\theta,\varepsilon) :
      \Bigl|\langle T b_{I\times J}^{(\theta,\varepsilon)},
      b_{I\times J}^{(\theta,\varepsilon)}\rangle -
      \sum_{\substack{K\in\mathcal{X}_I\\L\in\mathcal{Y}_J}}
      \langle Th_{K\times L}, h_{K\times L}\rangle
      \Bigr|
      > \eta_0
      \biggr\},
      \qquad I,J\in\mathcal{D}_{\leq n}.
    \end{equation*}
    By Theorem~\ref{thm:var} and the definition of the random variables $W,X,Y,Z$ (see~\eqref{eq:dfn:rv}),
    we obtain
    \begin{subequations}\label{eq:dfn:rv:tail:3}
      \begin{align}
        \prob_{\theta,\varepsilon}(O_{R,R'})
        &\leq \frac{4\Gamma^2}{2^{m_0/2}\eta_0^2},
        & R,R'\in\mathcal{D}_{\leq n}\otimes\mathcal{D}_{\leq n},\ R\neq R',&
                                                                              \label{eq:dfn:rv:tail:3:a}\\
        \prob_{\theta,\varepsilon}(D_{I,J})
        &\leq \frac{12\Gamma^2}{2^{m_0/2}\eta_0^2},
        &I,J\in\mathcal{D}_{\leq n}.&
                                      \label{eq:dfn:rv:tail:3:b}
      \end{align}
    \end{subequations}
    Combining~\eqref{eq:dfn:rv:tail:3} with~\eqref{proof:thm:factor:const:1} yields
    \begin{equation}\label{eq:dfn:rv:tail:4}
      \prob_{\theta,\varepsilon} \Biggr(
      \bigcup_{\substack{R,R'\in\mathcal{D}_{\leq n}\otimes\mathcal{D}_{\leq n}\\R\neq R'}} O_{R,R'} \cup
      \bigcup_{I,J\in\mathcal{D}_{\leq n}} D_{I,J}
      \Biggl)
      \leq 2^{4(n+3)}\frac{\Gamma^2}{2^{m_0/2}\eta_0^2}
      < 1.
    \end{equation}
    Hence, we can find at least one $(\theta,\varepsilon)\in\Omega_\theta\times\Omega_\varepsilon$
    such that
    \begin{subequations}\label{proof:thm:factor:almost-diag}
      \begin{align}
        |\langle T b_R^{(\theta,\varepsilon)}, b_{R'}^{(\theta,\varepsilon)}\rangle|
        \leq \eta_0,
        \qquad R,R'\in\mathcal{D}_{\leq n}\otimes\mathcal{D}_{\leq n},\ R\neq R',
        \label{proof:thm:factor:almost-diag:a}\\
        \Bigl|\langle T b_{I\times J}^{(\theta,\varepsilon)},
        b_{I\times J}^{(\theta,\varepsilon)}\rangle -
        \sum_{\substack{K\in\mathcal{X}_I\\L\in\mathcal{Y}_J}}
        \langle Th_{K\times L}, h_{K\times L}\rangle\Bigr|
        \leq \eta_0,
        \qquad I,J\in\mathcal{D}_{\leq n}.
        \label{proof:thm:factor:almost-diag:b}
      \end{align}
    \end{subequations}
    Recall that $\kappa=1$ by construction of $\mathcal{X}_I$ and $\mathcal{Y}_J$,
    $I,J\in\mathcal{D}_{\leq n}$ (see~\eqref{eq:proof:thm:factor:coll:2}).  Hence,
    by~\eqref{eq:thm:factor:large}, \textrefp[J]{enu:j1} and~\textrefp[J]{enu:j3} we obtain
    \begin{equation*}
      \sum_{\substack{K\in\mathcal{X}_I\\L\in\mathcal{Y}_J}}
      \langle Th_{K\times L}, h_{K\times L}\rangle
      \geq \sum_{\substack{K\in\mathcal{X}_I\\L\in\mathcal{Y}_J}}
      \delta |K\times L|
      = \delta |X_I\times Y_J|
      = \delta |I\times J|,
      \qquad I,J\in\mathcal{D}_{\leq n}.
    \end{equation*}
    The latter estimate and~\eqref{proof:thm:factor:almost-diag:b} give us
    \begin{equation}\label{proof:thm:factor:almost-diag:b:final:1}
      \langle T b_R^{(\theta,\varepsilon)}, b_R^{(\theta,\varepsilon)}\rangle
      \geq \delta |R| - \eta_0,
      \qquad R\in\mathcal{D}_{\leq n}\otimes\mathcal{D}_{\leq n}.
    \end{equation}
    Note that by Lemma~\ref{lem:block:basic-estimate} we have $\|b_R^{(\theta,\varepsilon)}\|_2^2 = R$,
    thus we obtain from~\eqref{proof:thm:factor:almost-diag:b:final:1}
    \begin{equation}\label{eq:proof:thm:factor:diag-est}
      \langle T b_R^{(\theta,\varepsilon)}, b_R^{(\theta,\varepsilon)}\rangle
      \geq (\delta - \eta_0 2^{2n}) \|b_R^{(\theta,\varepsilon)}\|_2^2,
      \qquad R\in\mathcal{D}_{\leq n}\otimes\mathcal{D}_{\leq n}.
    \end{equation}
    Combining~\eqref{proof:thm:factor:almost-diag:a} with~\eqref{eq:proof:thm:factor:diag-est}
    yields
    \begin{subequations}\label{proof:thm:factor:almost-diag:final}
      \begin{align}
        |\langle T b_R^{(\theta,\varepsilon)}, b_{R'}^{(\theta,\varepsilon)}\rangle|
        &\leq \eta_0,
        &R,R'\in\mathcal{D}_{\leq n}\otimes\mathcal{D}_{\leq n},\ R\neq R',&
                                                                             \label{proof:thm:factor:almost-diag:final:a}\\
        \langle T b_R^{(\theta,\varepsilon)}, b_R^{(\theta,\varepsilon)}\rangle
        &\geq (\delta - \eta_0 2^{2n}) \|b_R^{(\theta,\varepsilon)}\|_2^2,
        &R\in\mathcal{D}_{\leq n}\otimes\mathcal{D}_{\leq n}.&
                                                               \label{proof:thm:factor:almost-diag:final:b}
      \end{align}
    \end{subequations}
  \end{proofstep}

  \begin{proofstep}[Step~\theproofstep: Constructing the operators]
    Here, we will use the basic operators $B^{(\theta,\varepsilon)} : V_n\to V_N$ and
    $A^{(\theta,\varepsilon)} : V_N\to V_n$ given by
    \begin{subequations}\label{eq:proof:thm:factor:operators}
      \begin{align}
        B^{(\theta,\varepsilon)} f
        &= \sum_{R\in \mathcal{D}_{\leq n}\otimes\mathcal{D}_{\leq n}} \frac{\langle f, h_R\rangle}{\|h_R\|_2^2}
          b_R^{(\theta,\varepsilon)},
        && f\in V_n,
           \label{eq:proof:thm:factor:operators:a}\\
        A^{(\theta,\varepsilon)} f
        &= \sum_{R\in \mathcal{D}_{\leq n}\otimes\mathcal{D}_{\leq n}}
          \frac{\langle f, b_R^{(\theta,\varepsilon)}\rangle}{\|b_R^{(\theta,\varepsilon)}\|_2^2} h_R,
        && f\in V_N,
           \label{eq:proof:thm:factor:operators:b}
      \end{align}
    \end{subequations}
    as building blocks for the operators $E$ and $F$ in diagram~\eqref{eq:thm:factor:diag}.  Let us
    recall that by Theorem~\ref{thm:projection}, the operators $B^{(\theta,\varepsilon)}$ and
    $A^{(\theta,\varepsilon)}$ satisfy the estimates
    \begin{equation}\label{eq:proof:thm:factor:estimates}
      \|B^{(\theta,\varepsilon)}\|
      \leq 1
      \qquad\text{and}\qquad
      \|A^{(\theta,\varepsilon)}\|
      \leq 1,
    \end{equation}
    and $P^{(\theta,\varepsilon)} : V_N\to V_N$ defined as
    $P^{(\theta,\varepsilon)} = B^{(\theta,\varepsilon)}A^{(\theta,\varepsilon)}$ is a norm $1$
    projection given by
    \begin{equation}\label{eq:proof:thm:factor:P:definition}
      P^{(\theta,\varepsilon)} f
      = \sum_{R\in\mathcal{D}_{\leq n}\otimes\mathcal{D}_{\leq n}}
      \frac{\langle f, b_R^{(\theta,\varepsilon)}\rangle}{\|b_R^{(\theta,\varepsilon)}\|_2^2}
      b_R^{(\theta,\varepsilon)},
      \qquad f\in V_N.
    \end{equation}

    Now put $Y = P^{(\theta,\varepsilon)}(V_N)$ and note that the following diagram is commutative:
    \begin{equation}\label{eq:proof:thm:factor:commutative-diagram:preimage}
      \vcxymatrix{V_n \ar[rr]^{\Id_{V_n}} \ar[d]_{B^{(\theta,\varepsilon)}}
        && V_n\\
        Y \ar[rr]_{\Id_Y} && Y \ar[u]_{A^{(\theta,\varepsilon)}_{|Y}}}
      \qquad \|B^{(\theta,\varepsilon)}\|,\|A^{(\theta,\varepsilon)}_{|Y}\| \leq 1.
    \end{equation}
    Observe that $T$ almost acts as a multiplication operator on $Y$
    (see~\eqref{proof:thm:factor:almost-diag:final}).  Next, we define
    $U^{(\theta,\varepsilon)} : V_N\to Y$ by putting
    \begin{equation}\label{eq:proof:thm:factor:almost-inverse}
      U^{(\theta,\varepsilon)} f = \sum_{R\in\mathcal{D}_{\leq n}\otimes \mathcal{D}_{\leq n}}
      \frac{\langle f, b_R^{(\theta,\varepsilon)}\rangle}
      {\langle Tb_R^{(\theta,\varepsilon)}, b_R^{(\theta,\varepsilon)}\rangle}
      b_R^{(\theta,\varepsilon)},
      \qquad f\in V_N.
    \end{equation}
    By the $1$-unconditionality of the bi-parameter Haar system in $V_N$ and the definition of the
    norm $1$ projection $P^{(\theta,\varepsilon)}$ (see~\eqref{eq:proof:thm:factor:P:definition}
    and~\eqref{eq:proof:thm:factor:estimates}), we obtain
    \begin{equation}\label{eq:proof:thm:factor:U-bound}
      \|U^{(\theta,\varepsilon)}\|
      \leq \frac{\|P^{(\theta,\varepsilon)}\|}{\delta - \eta_0 2^{2n}}
      \leq \frac{1}{\delta - \eta_0 2^{2n}}.
    \end{equation}

    We will now show that $U^{(\theta,\varepsilon)} : V_N\to Y$ almost acts as an inverse of $T$
    restricted to $Y$.  Firstly, for all
    $g = \sum_{R\in\mathcal{D}_{\leq n}\otimes\mathcal{D}_{\leq n}} a_R b_R^{(\theta,\varepsilon)}
    \in Y$, we have the following identity:
    \begin{equation}\label{eq:proof:thm:factor:crucial-identity}
      U^{(\theta,\varepsilon)}Tg - g
      = \sum_{\substack{R,R'\in\mathcal{D}_{\leq n}\otimes\mathcal{D}_{\leq n}\\R' \neq R}} a_{R'}
      \frac{\langle T b_{R'}^{(\theta,\varepsilon)}, b_R^{(\theta,\varepsilon)}\rangle}
      {\langle Tb_R^{(\theta,\varepsilon)}, b_R^{(\theta,\varepsilon)}\rangle} b_R^{(\theta,\varepsilon)}.
    \end{equation}
    Secondly, by Lemma~\ref{lem:block:basic-estimate}, we have the estimate
    \begin{equation*}
      |a_{R'}|
      \leq \frac{\|g\|_{V_N}}{\|b_{R'}^{(\theta,\varepsilon)}\|_{V_N}}
      \leq 2^{2(n+1)}\|g\|_{V_N},
      \qquad R'\in\mathcal{D}_{\leq n}\otimes\mathcal{D}_{\leq n},
    \end{equation*}
    thus, by~\eqref{eq:proof:thm:factor:crucial-identity}
    and~\eqref{proof:thm:factor:almost-diag:final}, we obtain
    \begin{equation}\label{eq:proof:thm:factor:crucial-estimate:2}
      \|U^{(\theta,\varepsilon)}Tg - g\|_{V_N}
      \leq \frac{\eta_0 2^{8(n+1)}}{\delta-\eta_0 2^{2n}} \|g\|_{V_N}.
    \end{equation}

    Next, let $I : Y\to V_N$ denote the operator given by $Iy = y$.  Observe that
    by~\eqref{proof:thm:factor:const:1} we have
    $\frac{\eta_0 2^{8(n+1)}}{\delta-\eta_0 2^{2n}} < 1$, hence we obtain
    from~\eqref{eq:proof:thm:factor:crucial-estimate:2} that
    \begin{equation}\label{eq:proof:thm:factor:crucial-estimate:3}
      \|(U^{(\theta,\varepsilon)}TI)^{-1}g\|_{V_N}
      \leq \frac{1}{1-\frac{\eta_0 2^{8(n+1)}}{\delta-\eta_0 2^{2n}}} \|g\|_{V_N}.
    \end{equation}
    Now, define $S^{(\theta,\varepsilon)} : V_N\to Y$ by putting
    $S^{(\theta,\varepsilon)}=(U^{(\theta,\varepsilon)}TI)^{-1}U^{(\theta,\varepsilon)}$, and note
    that~\eqref{eq:proof:thm:factor:U-bound}, \eqref{eq:proof:thm:factor:crucial-estimate:3}
    and~\eqref{proof:thm:factor:const:1} gives us
    \begin{equation*}
      \|S^{(\theta,\varepsilon)}\|
      \leq \frac{1}{\delta - \eta_0 (2^{2n} + 2^{8(n+1)})}
      \leq \frac{1+\eta}{\delta}.
    \end{equation*}
    Moreover, the following diagram is commutative:
    \begin{equation}\label{eq:proof:thm:factor:commutative-diagram:image}
      \vcxymatrix{
        Y \ar@/^/[rrrr]^{\Id_Y} \ar@/_/[dd]_I \ar[rrd]_{U^{(\theta,\varepsilon)}TI} &&&& Y\\
        && Y \ar[rru]^{(U^{(\theta,\varepsilon)}TI)^{-1}} &&\\
        V_N \ar[rrrr]_T &&&& V_N \ar[llu]_{U^{(\theta,\varepsilon)}}
        \ar[uu]_{S^{(\theta,\varepsilon)}}
      }
      \qquad \|I\|\|S^{(\theta,\varepsilon)}\| \leq \frac{1+\eta}{\delta}.
    \end{equation}
    Merging the diagrams~\eqref{eq:proof:thm:factor:commutative-diagram:preimage}
    and~\eqref{eq:proof:thm:factor:commutative-diagram:image} yields
    \begin{equation}\label{eq:proof:thm:factor:commutative-diagram:merged}
      \vcxymatrix{%
        V_n \ar@/_{20pt}/[ddd]_{E} \ar[rrrr]^{I_{V_n}} \ar[d]^{B^{(\theta,\varepsilon)}} &&&& V_n\\
        Y \ar@/^/[rrrr]^{\Id_Y} \ar@/_/[dd]_I \ar[rrd]_{U^{(\theta,\varepsilon)}TI} &&&& Y \ar[u]^{A^{(\theta,\varepsilon)}_{|Y}}\\
        && Y \ar[rru]^{(U^{(\theta,\varepsilon)}TI)^{-1}} &&\\
        V_N \ar[rrrr]_T &&&& V_N \ar[llu]_{U^{(\theta,\varepsilon)}}
        \ar[uu]_{S^{(\theta,\varepsilon)}} \ar@/_{30pt}/[uuu]_{F}
      }
      \qquad \|E\| \|F\| \leq \frac{1+\eta}{\delta}.
    \end{equation}
  \end{proofstep}

  Finally, by reviewing the construction of our block basis $b_R^{(\theta)}$,
  $R\in\mathcal{D}_{\leq n}\otimes\mathcal{D}_{\leq n}$ (see~\eqref{eq:proof:thm:factor:coll:1}
  and~\eqref{eq:proof:thm:factor:coll:2}), the definition of our basic operators
  $B^{(\theta,\varepsilon)}$ and $A^{(\theta,\varepsilon)}$ and the constants defined
  in~\eqref{proof:thm:factor:const:1}, we conclude that~\eqref{eq:thm:factor:dim} is an appropriate
  choice for $N$.
\end{myproof}


\section{Proof of Theorem~\ref{thm:var}}\label{sec:lem-proof}

\noindent
We only present the proof for $V_N = H_N^p(H_N^q)$.  For $V_N = (H_N^p(H_N^q))^*$, the roles of
$p,q$ and $p',q'$ are reversed, where $\frac{1}{p} + \frac{1}{p'} = 1$ and
$\frac{1}{q} + \frac{1}{q'} = 1$.

The proof is divided into four parts:
\begin{itemize}
\item Estimates for $W$,
\item Estimates for $X$,
\item Estimates for $Y$,
\item Estimates for $Z$.
\end{itemize}
For each of the four random variables $W,X,Y,Z$, there is a unique ensemble of summation parameters,
which is recorded at the beginning of each section.  The summation parameters are split into
separate cases.  Every case possess a left variant and a right variant, which refers to whether we
place the sum in the left argument or in the right argument of the bilinear form.  The estimates
obtained for the left and the right variant of a case are combined to a single estimate at the end
of each case.

Before we begin with the proof, we make the following crucial observations: Firstly,
$\cond_\zeta \zeta_{S_0}\zeta_{S_1}\zeta_{S_0'}\zeta_{S_1'}\in\{0,1\}$, for all
$S_0,S_1,S_0',S_1'\in\mathcal{D}$.  Secondly, given dyadic intervals
$S_0,S_1,S_0',S_1'\in\mathcal{D}$, we have that
$\cond_\zeta \zeta_{S_0}\zeta_{S_1}\zeta_{S_0'}\zeta_{S_1'} = 1$ if and only if one of the following
conditions \textrefp[R]{enu:proof:lem:var:r1}--\textrefp[R]{enu:proof:lem:var:r4} is satisfied:
\begin{enumerate}[(R1)]
\item\label{enu:proof:lem:var:r1} $S_0 = S_1 = S_0' = S_1'$;
\item\label{enu:proof:lem:var:r2} $S_0 = S_1\neq S_0' = S_1'$;
\item\label{enu:proof:lem:var:r3} $S_0 = S_0'\neq S_1 = S_1'$;
\item\label{enu:proof:lem:var:r4} $S_0 = S_1'\neq S_1 = S_0'$.
\end{enumerate}

\subsection{Estimates for $W$}

In this case, the following variables will always be summed over the following sets:
\begin{itemize}
\item
  $K_0$, $K_1$ over $\mathcal{X}_I$;
\item
  $K_0'$, $K_1'$ over $\mathcal{X}_{I'}$;
\item
  $L_0$, $L_1$ over $\mathcal{Y}_J$;
\item
  $L_0'$, $L_1'$ over $\mathcal{Y}_{J'}$.
\end{itemize}

\begin{myproof}
  First, note that by~\eqref{eq:dfn:rv:W} and~\eqref{eq:b:dfn} we obtain $W^2(\theta,\varepsilon)$
  is given by
  \begin{equation}\label{eq:proof:lem:var:W}
    \sum_{\substack{K_0,K_1,K_0',K_1'\\L_0,L_1,L_0',L_1'}}
    \theta_{K_0}\theta_{K_1}\theta_{K_0'}\theta_{K_1'}
    \varepsilon_{L_0}\varepsilon_{L_1}\varepsilon_{L_0'}\varepsilon_{L_1'}
    \langle T h_{K_0\times L_0}, h_{K_0'\times L_0'}\rangle
    \langle T h_{K_1\times L_1}, h_{K_1'\times L_1'}\rangle
  \end{equation}
  In view of~\textrefp[J]{enu:j1}
  and~\textrefp[R]{enu:proof:lem:var:r1}--\textrefp[R]{enu:proof:lem:var:r4}, we obtain that
  \begin{equation}\label{eq:proof:lem:var:W:cond}
    \cond_{\theta,\varepsilon} W^2
    = \sum_{K_0,K_0',L_0,L_0'} \langle T h_{K_0\times L_0}, h_{K_0'\times L_0'}\rangle^2.
  \end{equation}
  
  \begin{proofcase}[Case~\theproofcase: $K_0 = K_1 \neq K_0' = K_1'$, $L_0 = L_1 \neq L_0' = L_1'$
    (((1) (1)) ((0) (0)) (nil nil)) -- left variant]
    In this case, we have to estimate
    \begin{equation}\label{eq:w-est(((0 1 0 1) (0 1 0 1)) (((1) (1)) ((0) (0)) (nil nil))):0}
      \sum_{K_0', L_0', K_0, L_0} \langle Th_{K_0\times L_0}, h_{K_0'\times L_0'} \rangle \langle Th_{K_0\times L_0}, h_{K_0'\times L_0'} \rangle.
    \end{equation}
    We put $a_{K_0, K_0', L_0, L_0'} = \langle T h_{K_0\times L_0}, h_{K_0'\times L_0'} \rangle$ and
    note the estimate
    \begin{equation}\label{eq:w-est(((0 1 0 1) (0 1 0 1)) (((1) (1)) ((0) (0)) (nil nil))):1}
      |a_{K_0, K_0', L_0, L_0'}|
      \leq \|T\| |K_0|^{1/{p}} |L_0|^{1/{q}} |K_0'|^{1/{p'}} |L_0'|^{1/{q'}}.
    \end{equation}
    Now, we write~\eqref{eq:w-est(((0 1 0 1) (0 1 0 1)) (((1) (1)) ((0) (0)) (nil nil))):0} as
    follows:
    \begin{equation}\label{eq:w-est(((0 1 0 1) (0 1 0 1)) (((1) (1)) ((0) (0)) (nil nil))):2}
      \sum_{K_0', L_0'} \Bigl\langle T \sum_{K_0, L_0} a_{K_0, K_0', L_0, L_0'} h_{K_0\times L_0},  h_{K_0'\times L_0'} \Bigr\rangle.
    \end{equation}
    By duality, we obtain the subsequent upper estimate for~\eqref{eq:w-est(((0 1 0 1) (0 1 0 1))
      (((1) (1)) ((0) (0)) (nil nil))):2}:
    \begin{equation}\label{eq:w-est(((0 1 0 1) (0 1 0 1)) (((1) (1)) ((0) (0)) (nil nil))):3}
      \sum_{K_0', L_0'}\left\|T \sum_{K_0, L_0} a_{K_0, K_0', L_0, L_0'} h_{K_0\times L_0}\right\|_{H^p(H^q)} \left\| h_{K_0'\times L_0'}\right\|_{(H^p(H^q))^*}.
    \end{equation}
    Estimate~\eqref{eq:w-est(((0 1 0 1) (0 1 0 1)) (((1) (1)) ((0) (0)) (nil nil))):1} and the
    disjointness of the dyadic intervals (see~\textrefp[J]{enu:j1}) yield
    \begin{equation}\label{eq:w-est(((0 1 0 1) (0 1 0 1)) (((1) (1)) ((0) (0)) (nil nil))):4}
      \sum_{K_0', L_0'}\|T\| \left\| \sum_{K_0, L_0}\max_{K_0, L_0}\left(\|T\| |K_0|^{1/{p}} |L_0|^{1/{q}} |K_0'|^{1/{p'}} |L_0'|^{1/{q'}}\right) h_{K_0\times L_0}\right\|_{H^p(H^q)}\left\| h_{K_0'\times L_0'}\right\|_{(H^p(H^q))^*}.
    \end{equation}
    Consequently, we obtain
    \begin{equation}\label{eq:w-est(((0 1 0 1) (0 1 0 1)) (((1) (1)) ((0) (0)) (nil nil))):5}
      \sum_{K_0', L_0'}\|T\| \max_{K_0, L_0}\left(\|T\| |K_0|^{1/{p}} |L_0|^{1/{q}} |K_0'|^{1/{p'}} |L_0'|^{1/{q'}}\right) \left\| \sum_{K_0, L_0}h_{K_0\times L_0}\right\|_{H^p(H^q)}\left\| h_{K_0'\times L_0'}\right\|_{(H^p(H^q))^*}.
    \end{equation}
    Thus,~\eqref{eq:w-est(((0 1 0 1) (0 1 0 1)) (((1) (1)) ((0) (0)) (nil nil))):5} is bounded from
    above by
    \begin{equation*}
      \|T\|^2 \max_{K_0, L_0}  |K_0|^{1/{p}} |L_0|^{1/{q}}
      \sum_{K_0', L_0'} |K_0'|^{2/{p'}} |L_0'|^{2/{q'}},
    \end{equation*}
    which by Hölder's inequality is dominated by
    \begin{equation}\label{eq:w-est(((0 1 0 1) (0 1 0 1)) (((1) (1)) ((0) (0)) (nil nil))):6}
      \|T\|^2 \max_{K_0, L_0, K_0', L_0'} |K_0|^{1/{p}} |L_0|^{1/{q}} |K_0'|^{2/{p'}-1} |L_0'|^{2/{q'}-1}.
    \end{equation}
    Inserting $|K_0|, |L_0|, |K_0'|, |L_0'| \leq \alpha$ (see~\eqref{eq:alpha-small})
    into~\eqref{eq:w-est(((0 1 0 1) (0 1 0 1)) (((1) (1)) ((0) (0)) (nil nil))):6}, we obtain the
    estimate
    \begin{equation}\label{eq:w-est(((0 1 0 1) (0 1 0 1)) (((1) (1)) ((0) (0)) (nil nil))):7}
      \|T\|^2 \alpha^{1/{p'}+1/{q'}}.
    \end{equation}
  \end{proofcase}

  \begin{proofcase}[Case~\theproofcase: $K_0 = K_1 \neq K_0' = K_1'$, $L_0 = L_1 \neq L_0' = L_1'$
    (((0) (0)) (nil nil) ((1) (1))) -- right variant]
    In this case, we have to estimate
    \begin{equation}\label{eq:w-est(((0 1 0 1) (0 1 0 1)) (((0) (0)) (nil nil) ((1) (1)))):0}
      \sum_{K_0, L_0, K_0', L_0'} \langle Th_{K_0\times L_0}, h_{K_0'\times L_0'} \rangle \langle Th_{K_0\times L_0}, h_{K_0'\times L_0'} \rangle.
    \end{equation}
    We put $a_{K_0, K_0', L_0, L_0'} = \langle T h_{K_0\times L_0}, h_{K_0'\times L_0'} \rangle$ and
    note the estimate
    \begin{equation}\label{eq:w-est(((0 1 0 1) (0 1 0 1)) (((0) (0)) (nil nil) ((1) (1)))):1}
      |a_{K_0, K_0', L_0, L_0'}|
      \leq \|T\| |K_0|^{1/{p}} |L_0|^{1/{q}} |K_0'|^{1/{p'}} |L_0'|^{1/{q'}}.
    \end{equation}
    Now, we write~\eqref{eq:w-est(((0 1 0 1) (0 1 0 1)) (((0) (0)) (nil nil) ((1) (1)))):0} as
    follows:
    \begin{equation}\label{eq:w-est(((0 1 0 1) (0 1 0 1)) (((0) (0)) (nil nil) ((1) (1)))):2}
      \sum_{K_0, L_0} \Bigl\langle T  h_{K_0\times L_0}, \sum_{K_0', L_0'} a_{K_0, K_0', L_0, L_0'} h_{K_0'\times L_0'} \Bigr\rangle.
    \end{equation}
    By duality, we obtain the subsequent upper estimate for~\eqref{eq:w-est(((0 1 0 1) (0 1 0 1))
      (((0) (0)) (nil nil) ((1) (1)))):2}:
    \begin{equation}\label{eq:w-est(((0 1 0 1) (0 1 0 1)) (((0) (0)) (nil nil) ((1) (1)))):3}
      \sum_{K_0, L_0}\left\|T  h_{K_0\times L_0}\right\|_{H^p(H^q)} \left\|\sum_{K_0', L_0'} a_{K_0, K_0', L_0, L_0'} h_{K_0'\times L_0'}\right\|_{(H^p(H^q))^*}.
    \end{equation}
    Estimate~\eqref{eq:w-est(((0 1 0 1) (0 1 0 1)) (((0) (0)) (nil nil) ((1) (1)))):1} and the
    disjointness of the dyadic intervals (see~\textrefp[J]{enu:j1}) yield
    \begin{equation}\label{eq:w-est(((0 1 0 1) (0 1 0 1)) (((0) (0)) (nil nil) ((1) (1)))):4}
      \sum_{K_0, L_0}\|T\| \left\| h_{K_0\times L_0}\right\|_{H^p(H^q)}\left\| \sum_{K_0', L_0'}\max_{K_0', L_0'}\left(\|T\| |K_0|^{1/{p}} |L_0|^{1/{q}} |K_0'|^{1/{p'}} |L_0'|^{1/{q'}}\right) h_{K_0'\times L_0'}\right\|_{(H^p(H^q))^*}.
    \end{equation}
    Consequently, we obtain
    \begin{equation}\label{eq:w-est(((0 1 0 1) (0 1 0 1)) (((0) (0)) (nil nil) ((1) (1)))):5}
      \sum_{K_0, L_0}\|T\| \left\| h_{K_0\times L_0}\right\|_{H^p(H^q)}\max_{K_0', L_0'}\left(\|T\| |K_0|^{1/{p}} |L_0|^{1/{q}} |K_0'|^{1/{p'}} |L_0'|^{1/{q'}}\right) \left\| \sum_{K_0', L_0'}h_{K_0'\times L_0'}\right\|_{(H^p(H^q))^*}.
    \end{equation}
    Thus,~\eqref{eq:w-est(((0 1 0 1) (0 1 0 1)) (((0) (0)) (nil nil) ((1) (1)))):5} is bounded from
    above by
    \begin{equation*}
      \|T\|^2 \max_{K_0', L_0'}  |K_0'|^{1/{p'}} |L_0'|^{1/{q'}}
      \sum_{K_0, L_0} |K_0|^{2/{p}} |L_0|^{2/{q}},
    \end{equation*}
    which by Hölder's inequality is dominated by
    \begin{equation}\label{eq:w-est(((0 1 0 1) (0 1 0 1)) (((0) (0)) (nil nil) ((1) (1)))):6}
      \|T\|^2 \max_{K_0, L_0, K_0', L_0'}  |K_0'|^{1/{p'}} |L_0'|^{1/{q'}} |K_0|^{2/{p}-1} |L_0|^{2/{q}-1}.
    \end{equation}
    Inserting $|K_0|, |L_0|, |K_0'|, |L_0'| \leq \alpha$ (see~\eqref{eq:alpha-small})
    into~\eqref{eq:w-est(((0 1 0 1) (0 1 0 1)) (((0) (0)) (nil nil) ((1) (1)))):6}, we obtain the
    estimate
    \begin{equation}\label{eq:w-est(((0 1 0 1) (0 1 0 1)) (((0) (0)) (nil nil) ((1) (1)))):7}
      \|T\|^2 \alpha^{1/{p}+1/{q}}.
    \end{equation}
  \end{proofcase}

  \subsubsection*{Summary for $W$}
  Combining~\eqref{eq:w-est(((0 1 0 1) (0 1 0 1)) (((1) (1)) ((0) (0)) (nil nil))):7}
  with~\eqref{eq:w-est(((0 1 0 1) (0 1 0 1)) (((0) (0)) (nil nil) ((1) (1)))):7} yields
  \begin{equation}\label{eq:w-est:final}
    \cond_{\theta,\varepsilon} W^2
    \leq \|T\|^2 \alpha.
  \end{equation}
\end{myproof}

\subsection{Estimates for $X$}

In this case, the following variables will always be summed over the following sets:
\begin{itemize}
\item
  $K_0$, $K_1$ over $\mathcal{X}_I$;
\item
  $K_0'$, $K_1'$ over $\mathcal{X}_{I'}$;
\item
  $L_0$, $L_1$, $L'$, $L_0'$, $L_1'$ over $\mathcal{Y}_J$.
\end{itemize}

\begin{myproof}
  Note that by~\eqref{eq:dfn:rv:X} and~\eqref{eq:b:dfn} we obtain $X^2(\theta,\varepsilon)$ is given
  by
  \begin{equation}\label{proof:eq:lem:var:X}
    \sum_{\substack{K_0,K_1,K_0',K_1'\\L_0,L_1,L_0',L_1'}}
    \theta_{K_0}\theta_{K_1}\theta_{K_0'}\theta_{K_1'}
    \varepsilon_{L_0}\varepsilon_{L_1}\varepsilon_{L_0'}\varepsilon_{L_1'}
    \langle T h_{K_0\times L_0}, h_{K_0'\times L_0'}\rangle
    \langle T h_{K_1\times L_1}, h_{K_1'\times L_1'}\rangle
  \end{equation}
  Note that since in this case $I\neq I'$, we have that
  $\mathcal{X}_I\cap \mathcal{X}_{I'} = \emptyset$, by~\textrefp[J]{enu:j1}.  Thus,
  $\cond_{\theta} \theta_{K_0}\theta_{K_0'}\theta_{K_1}\theta_{K_1'} \neq 0$, only if
  $K_0 = K_1 \neq K_0' = K_1'$.  Hence, in view
  of~\textrefp[R]{enu:proof:lem:var:r1}--\textrefp[R]{enu:proof:lem:var:r4}, we decompose the index
  set in~\eqref{proof:eq:lem:var:X} into the following four groups:
  \begin{enumerate}[({a}1)]
  \item\label{enu:proof:lem:var:X:a:1} $K_0 = K_1 \neq K_0' = K_1'$ and $L_0 = L_1 = L_0' = L_1'$;
  \item\label{enu:proof:lem:var:X:a:2} $K_0 = K_1 \neq K_0' = K_1'$ and
    $L_0 = L_1 \neq L_0' = L_1'$;
  \item\label{enu:proof:lem:var:X:a:3} $K_0 = K_1 \neq K_0' = K_1'$ and
    $L_0 = L_1' \neq L_0' = L_1$;
  \item\label{enu:proof:lem:var:X:a:4} $K_0 = K_1 \neq K_0' = K_1'$ and
    $L_0 = L_0' \neq L_1 = L_1'$.
  \end{enumerate}
  
  \begin{proofcase}[Case~\theproofcase, group~{\textrefp[a]{enu:proof:lem:var:X:a:1}}:
    $K_0 = K_1 \neq K_0' = K_1'$, $L_0 = L_0' = L_1 = L_1'$ (((1) (0)) ((0) nil) (nil nil)) -- left
    variant]\label{case:x-est(((0 1 0 1) (0 0 0 0)) (((1) (0)) ((0) nil) (nil nil))):left}
    In this case, we have to estimate
    \begin{equation}\label{eq:x-est(((0 1 0 1) (0 0 0 0)) (((1) (0)) ((0) nil) (nil nil))):0}
      \sum_{K_0', L_0, K_0} \langle Th_{K_0\times L_0}, h_{K_0'\times L_0} \rangle \langle Th_{K_0\times L_0}, h_{K_0'\times L_0} \rangle.
    \end{equation}
    We put $a_{K_0, K_0', L_0} = \langle T h_{K_0\times L_0}, h_{K_0'\times L_0} \rangle$ and note
    the estimate
    \begin{equation}\label{eq:x-est(((0 1 0 1) (0 0 0 0)) (((1) (0)) ((0) nil) (nil nil))):1}
      |a_{K_0, K_0', L_0}|
      \leq \|T\| |K_0|^{1/{p}} |L_0|^{1/{q}} |K_0'|^{1/{p'}} |L_0|^{1/{q'}}.
    \end{equation}
    Now, we write~\eqref{eq:x-est(((0 1 0 1) (0 0 0 0)) (((1) (0)) ((0) nil) (nil nil))):0} as
    follows:
    \begin{equation}\label{eq:x-est(((0 1 0 1) (0 0 0 0)) (((1) (0)) ((0) nil) (nil nil))):2}
      \sum_{K_0', L_0} \Bigl\langle T \sum_{K_0} a_{K_0, K_0', L_0} h_{K_0\times L_0},  h_{K_0'\times L_0} \Bigr\rangle.
    \end{equation}
    By duality, we obtain the subsequent upper estimate for~\eqref{eq:x-est(((0 1 0 1) (0 0 0 0))
      (((1) (0)) ((0) nil) (nil nil))):2}:
    \begin{equation}\label{eq:x-est(((0 1 0 1) (0 0 0 0)) (((1) (0)) ((0) nil) (nil nil))):3}
      \sum_{K_0', L_0}\left\|T \sum_{K_0} a_{K_0, K_0', L_0} h_{K_0\times L_0}\right\|_{H^p(H^q)} \left\| h_{K_0'\times L_0}\right\|_{(H^p(H^q))^*}.
    \end{equation}
    Estimate~\eqref{eq:x-est(((0 1 0 1) (0 0 0 0)) (((1) (0)) ((0) nil) (nil nil))):1} and the
    disjointness of the dyadic intervals (see~\textrefp[J]{enu:j1}) yield
    \begin{equation}\label{eq:x-est(((0 1 0 1) (0 0 0 0)) (((1) (0)) ((0) nil) (nil nil))):4}
      \sum_{K_0', L_0}\|T\| \left\| \sum_{K_0}\max_{K_0}\left(\|T\| |K_0|^{1/{p}} |L_0|^{1/{q}} |K_0'|^{1/{p'}} |L_0|^{1/{q'}}\right) h_{K_0\times L_0}\right\|_{H^p(H^q)}\left\| h_{K_0'\times L_0}\right\|_{(H^p(H^q))^*}.
    \end{equation}
    Consequently, we obtain
    \begin{equation}\label{eq:x-est(((0 1 0 1) (0 0 0 0)) (((1) (0)) ((0) nil) (nil nil))):5}
      \sum_{K_0', L_0}\|T\| \max_{K_0}\left(\|T\| |K_0|^{1/{p}} |L_0|^{1/{q}} |K_0'|^{1/{p'}} |L_0|^{1/{q'}}\right) \left\| \sum_{K_0}h_{K_0\times L_0}\right\|_{H^p(H^q)}\left\| h_{K_0'\times L_0}\right\|_{(H^p(H^q))^*}.
    \end{equation}
    Thus,~\eqref{eq:x-est(((0 1 0 1) (0 0 0 0)) (((1) (0)) ((0) nil) (nil nil))):5} is bounded from
    above by
    \begin{equation*}
      \|T\|^2 \max_{K_0} |K_0|^{1/{p}} \sum_{K_0', L_0} |L_0|^2 |K_0'|^{2/{p'}}.
    \end{equation*}
    Using Hölder's inequality yields
    \begin{equation}\label{eq:x-est(((0 1 0 1) (0 0 0 0)) (((1) (0)) ((0) nil) (nil nil))):6}
      \|T\|^2 \max_{K_0, K_0', L_0} |K_0|^{1/{p}} |L_0| |K_0'|^{2/{p'}-1}.
    \end{equation}
    Inserting $|K_0|, |L_0|, |K_0'|, |L_0'| \leq \alpha$ (see~\eqref{eq:alpha-small})
    into~\eqref{eq:x-est(((0 1 0 1) (0 0 0 0)) (((1) (0)) ((0) nil) (nil nil))):6}, we obtain the
    estimate
    \begin{equation}\label{eq:x-est(((0 1 0 1) (0 0 0 0)) (((1) (0)) ((0) nil) (nil nil))):7}
      \|T\|^2 \alpha^{1+1/{p'}}.
    \end{equation}
  \end{proofcase}

  \begin{proofcase}[Case~\theproofcase, group~{\textrefp[a]{enu:proof:lem:var:X:a:2}}: $K_0 = K_1 \neq K_0' = K_1'$, $L_0 = L_1 \neq L_0' = L_1'$
    (((1) (1)) ((0) (0)) (nil nil)) -- left variant]\label{case:x-est(((0 1 0 1) (0 1 0 1)) (((1) (1)) ((0) (0)) (nil nil))):left}
    In this case, we have to estimate
    \begin{equation}\label{eq:x-est(((0 1 0 1) (0 1 0 1)) (((1) (1)) ((0) (0)) (nil nil))):0}
      \sum_{K_0', L_0', K_0, L_0} \langle Th_{K_0\times L_0}, h_{K_0'\times L_0'} \rangle \langle Th_{K_0\times L_0}, h_{K_0'\times L_0'} \rangle.
    \end{equation}
    We put $a_{K_0, K_0', L_0, L_0'} = \langle T h_{K_0\times L_0}, h_{K_0'\times L_0'} \rangle$ and
    note the estimate
    \begin{equation}\label{eq:x-est(((0 1 0 1) (0 1 0 1)) (((1) (1)) ((0) (0)) (nil nil))):1}
      |a_{K_0, K_0', L_0, L_0'}|
      \leq \|T\| |K_0|^{1/{p}} |L_0|^{1/{q}} |K_0'|^{1/{p'}} |L_0'|^{1/{q'}}.
    \end{equation}
    Now, we write~\eqref{eq:x-est(((0 1 0 1) (0 1 0 1)) (((1) (1)) ((0) (0)) (nil nil))):0} as
    follows:
    \begin{equation}\label{eq:x-est(((0 1 0 1) (0 1 0 1)) (((1) (1)) ((0) (0)) (nil nil))):2}
      \sum_{K_0', L_0'} \Bigl\langle T \sum_{K_0, L_0} a_{K_0, K_0', L_0, L_0'} h_{K_0\times L_0},  h_{K_0'\times L_0'} \Bigr\rangle.
    \end{equation}
    By duality, we obtain the subsequent upper estimate for~\eqref{eq:x-est(((0 1 0 1) (0 1 0 1))
      (((1) (1)) ((0) (0)) (nil nil))):2}:
    \begin{equation}\label{eq:x-est(((0 1 0 1) (0 1 0 1)) (((1) (1)) ((0) (0)) (nil nil))):3}
      \sum_{K_0', L_0'}\left\|T \sum_{K_0, L_0} a_{K_0, K_0', L_0, L_0'} h_{K_0\times L_0}\right\|_{H^p(H^q)} \left\| h_{K_0'\times L_0'}\right\|_{(H^p(H^q))^*}.
    \end{equation}
    Estimate~\eqref{eq:x-est(((0 1 0 1) (0 1 0 1)) (((1) (1)) ((0) (0)) (nil nil))):1} and the
    disjointness of the dyadic intervals (see~\textrefp[J]{enu:j1}) yield
    \begin{equation}\label{eq:x-est(((0 1 0 1) (0 1 0 1)) (((1) (1)) ((0) (0)) (nil nil))):4}
      \sum_{K_0', L_0'}\|T\| \left\| \sum_{K_0, L_0}\max_{K_0, L_0}\left(\|T\| |K_0|^{1/{p}} |L_0|^{1/{q}} |K_0'|^{1/{p'}} |L_0'|^{1/{q'}}\right) h_{K_0\times L_0}\right\|_{H^p(H^q)}\left\| h_{K_0'\times L_0'}\right\|_{(H^p(H^q))^*}.
    \end{equation}
    Consequently, we obtain
    \begin{equation}\label{eq:x-est(((0 1 0 1) (0 1 0 1)) (((1) (1)) ((0) (0)) (nil nil))):5}
      \sum_{K_0', L_0'}\|T\| \max_{K_0, L_0}\left(\|T\| |K_0|^{1/{p}} |L_0|^{1/{q}} |K_0'|^{1/{p'}} |L_0'|^{1/{q'}}\right) \left\| \sum_{K_0, L_0}h_{K_0\times L_0}\right\|_{H^p(H^q)}\left\| h_{K_0'\times L_0'}\right\|_{(H^p(H^q))^*}.
    \end{equation}
    Thus,~\eqref{eq:x-est(((0 1 0 1) (0 1 0 1)) (((1) (1)) ((0) (0)) (nil nil))):5} is bounded from
    above by
    \begin{equation*}
      \|T\|^2 \max_{K_0, L_0} |K_0|^{1/{p}} |L_0|^{1/{q}} \sum_{K_0', L_0'}  |K_0'|^{2/{p'}} |L_0'|^{2/{q'}}.
    \end{equation*}
    Using Hölder's inequality yields
    \begin{equation}\label{eq:x-est(((0 1 0 1) (0 1 0 1)) (((1) (1)) ((0) (0)) (nil nil))):6}
      \|T\|^2 \max_{K_0, L_0, K_0', L_0'} |K_0|^{1/{p}} |L_0|^{1/{q}} |K_0'|^{2/{p'}-1} |L_0'|^{2/{q'}-1}.
    \end{equation}
    Inserting $|K_0|, |L_0|, |K_0'|, |L_0'| \leq \alpha$ (see~\eqref{eq:alpha-small})
    into~\eqref{eq:x-est(((0 1 0 1) (0 1 0 1)) (((1) (1)) ((0) (0)) (nil nil))):6}, we obtain the
    estimate
    \begin{equation}\label{eq:x-est(((0 1 0 1) (0 1 0 1)) (((1) (1)) ((0) (0)) (nil nil))):7}
      \|T\|^2 \alpha^{1/{p'}+1/{q'}}.
    \end{equation}
  \end{proofcase}

  \begin{proofcase}[Case~\theproofcase, group~{\textrefp[a]{enu:proof:lem:var:X:a:2}}: $K_0 = K_1 \neq K_0' = K_1'$, $L_0 = L_1 \neq L_0' = L_1'$
    (((0) (0)) (nil nil) ((1) (1))) -- right variant]\label{case:x-est(((0 1 0 1) (0 1 0 1)) (((0) (0)) (nil nil) ((1) (1)))):right}
    In this case, we have to estimate
    \begin{equation}\label{eq:x-est(((0 1 0 1) (0 1 0 1)) (((0) (0)) (nil nil) ((1) (1)))):0}
      \sum_{K_0, L_0, K_0', L_0'} \langle Th_{K_0\times L_0}, h_{K_0'\times L_0'} \rangle \langle Th_{K_0\times L_0}, h_{K_0'\times L_0'} \rangle.
    \end{equation}
    We put $a_{K_0, K_0', L_0, L_0'} = \langle T h_{K_0\times L_0}, h_{K_0'\times L_0'} \rangle$ and
    note the estimate
    \begin{equation}\label{eq:x-est(((0 1 0 1) (0 1 0 1)) (((0) (0)) (nil nil) ((1) (1)))):1}
      |a_{K_0, K_0', L_0, L_0'}|
      \leq \|T\| |K_0|^{1/{p}} |L_0|^{1/{q}} |K_0'|^{1/{p'}} |L_0'|^{1/{q'}}.
    \end{equation}
    Now, we write~\eqref{eq:x-est(((0 1 0 1) (0 1 0 1)) (((0) (0)) (nil nil) ((1) (1)))):0} as
    follows:
    \begin{equation}\label{eq:x-est(((0 1 0 1) (0 1 0 1)) (((0) (0)) (nil nil) ((1) (1)))):2}
      \sum_{K_0, L_0} \Bigl\langle T  h_{K_0\times L_0}, \sum_{K_0', L_0'} a_{K_0, K_0', L_0, L_0'} h_{K_0'\times L_0'} \Bigr\rangle.
    \end{equation}
    By duality, we obtain the subsequent upper estimate for~\eqref{eq:x-est(((0 1 0 1) (0 1 0 1))
      (((0) (0)) (nil nil) ((1) (1)))):2}:
    \begin{equation}\label{eq:x-est(((0 1 0 1) (0 1 0 1)) (((0) (0)) (nil nil) ((1) (1)))):3}
      \sum_{K_0, L_0}\left\|T  h_{K_0\times L_0}\right\|_{H^p(H^q)} \left\|\sum_{K_0', L_0'} a_{K_0, K_0', L_0, L_0'} h_{K_0'\times L_0'}\right\|_{(H^p(H^q))^*}.
    \end{equation}
    Estimate~\eqref{eq:x-est(((0 1 0 1) (0 1 0 1)) (((0) (0)) (nil nil) ((1) (1)))):1} and the
    disjointness of the dyadic intervals (see~\textrefp[J]{enu:j1}) yield
    \begin{equation}\label{eq:x-est(((0 1 0 1) (0 1 0 1)) (((0) (0)) (nil nil) ((1) (1)))):4}
      \sum_{K_0, L_0}\|T\| \left\| h_{K_0\times L_0}\right\|_{H^p(H^q)}\left\| \sum_{K_0', L_0'}\max_{K_0', L_0'}\left(\|T\| |K_0|^{1/{p}} |L_0|^{1/{q}} |K_0'|^{1/{p'}} |L_0'|^{1/{q'}}\right) h_{K_0'\times L_0'}\right\|_{(H^p(H^q))^*}.
    \end{equation}
    Consequently, we obtain
    \begin{equation}\label{eq:x-est(((0 1 0 1) (0 1 0 1)) (((0) (0)) (nil nil) ((1) (1)))):5}
      \sum_{K_0, L_0}\|T\| \left\| h_{K_0\times L_0}\right\|_{H^p(H^q)}\max_{K_0', L_0'}\left(\|T\| |K_0|^{1/{p}} |L_0|^{1/{q}} |K_0'|^{1/{p'}} |L_0'|^{1/{q'}}\right) \left\| \sum_{K_0', L_0'}h_{K_0'\times L_0'}\right\|_{(H^p(H^q))^*}.
    \end{equation}
    Thus,~\eqref{eq:x-est(((0 1 0 1) (0 1 0 1)) (((0) (0)) (nil nil) ((1) (1)))):5} is bounded from
    above by
    \begin{equation*}
      \|T\|^2 \max_{K_0', L_0'} |K_0'|^{1/{p'}} |L_0'|^{1/{q'}} \sum_{K_0, L_0} |K_0|^{2/{p}} |L_0|^{2/{q}}.
    \end{equation*}
    Using Hölder's inequality yields
    \begin{equation}\label{eq:x-est(((0 1 0 1) (0 1 0 1)) (((0) (0)) (nil nil) ((1) (1)))):6}
      \|T\|^2 \max_{K_0, L_0, K_0', L_0'} |K_0|^{2/{p}-1} |L_0|^{2/{q}-1} |K_0'|^{1/{p'}} |L_0'|^{1/{q'}}.
    \end{equation}
    Inserting $|K_0|, |L_0|, |K_0'|, |L_0'| \leq \alpha$ (see~\eqref{eq:alpha-small})
    into~\eqref{eq:x-est(((0 1 0 1) (0 1 0 1)) (((0) (0)) (nil nil) ((1) (1)))):6}, we obtain the
    estimate
    \begin{equation}\label{eq:x-est(((0 1 0 1) (0 1 0 1)) (((0) (0)) (nil nil) ((1) (1)))):7}
      \|T\|^2 \alpha^{1/{p}+1/{q}}.
    \end{equation}
  \end{proofcase}

  \subsubsection*{Summary of \textref[Case~]{case:x-est(((0 1 0 1) (0 1 0 1)) (((1) (1)) ((0) (0))
      (nil nil))):left} and \textref[Case~]{case:x-est(((0 1 0 1) (0 1 0 1)) (((0) (0)) (nil nil)
      ((1) (1)))):right}}
  Combining~\eqref{eq:x-est(((0 1 0 1) (0 1 0 1)) (((1) (1)) ((0) (0)) (nil nil))):7}
  with~\eqref{eq:x-est(((0 1 0 1) (0 1 0 1)) (((0) (0)) (nil nil) ((1) (1)))):7} yields
  \begin{equation}\label{eq:x-est(((0 1 0 1) (0 1 0 1)) (((1) (1)) ((0) (0)) (nil nil)))(((0 1 0 1)
      (0 1 0 1)) (((0) (0)) (nil nil) ((1) (1))))-final}
    \cond_{\theta,\varepsilon} X^2
    \leq \|T\|^2 \alpha.
  \end{equation}

  \begin{proofcase}[Case~\theproofcase, group~{\textrefp[a]{enu:proof:lem:var:X:a:3}}: $K_0 = K_1 \neq K_0' = K_1'$, $L_0 = L_1' \neq L_0' = L_1$ (((1) (0)) ((0) (1)) (nil nil)) -- left variant]\label{case:x-est(((0 1 0 1) (0 1 1 0)) (((1) (0)) ((0) (1)) (nil nil))):left}
    In this case, we have to estimate
    \begin{equation}\label{eq:x-est(((0 1 0 1) (0 1 1 0)) (((1) (0)) ((0) (1)) (nil nil))):0}
      \sum_{K_0', L_0, K_0, L_0'} \langle Th_{K_0\times L_0}, h_{K_0'\times L_0'} \rangle \langle Th_{K_0\times L_0'}, h_{K_0'\times L_0} \rangle.
    \end{equation}
    We put $a_{K_0, K_0', L_0, L_0'} = \langle T h_{K_0\times L_0}, h_{K_0'\times L_0'} \rangle$ and
    note the estimate
    \begin{equation}\label{eq:x-est(((0 1 0 1) (0 1 1 0)) (((1) (0)) ((0) (1)) (nil nil))):1}
      |a_{K_0, K_0', L_0, L_0'}|
      \leq \|T\| |K_0|^{1/{p}} |L_0|^{1/{q}} |K_0'|^{1/{p'}} |L_0'|^{1/{q'}}.
    \end{equation}
    Now, we write~\eqref{eq:x-est(((0 1 0 1) (0 1 1 0)) (((1) (0)) ((0) (1)) (nil nil))):0} as
    follows:
    \begin{equation}\label{eq:x-est(((0 1 0 1) (0 1 1 0)) (((1) (0)) ((0) (1)) (nil nil))):2}
      \sum_{K_0', L_0} \Bigl\langle T \sum_{K_0, L_0'} a_{K_0, K_0', L_0, L_0'} h_{K_0\times L_0'},  h_{K_0'\times L_0} \Bigr\rangle.
    \end{equation}
    By duality, we obtain the subsequent upper estimate for~\eqref{eq:x-est(((0 1 0 1) (0 1 1 0))
      (((1) (0)) ((0) (1)) (nil nil))):2}:
    \begin{equation}\label{eq:x-est(((0 1 0 1) (0 1 1 0)) (((1) (0)) ((0) (1)) (nil nil))):3}
      \sum_{K_0', L_0}\left\|T \sum_{K_0, L_0'} a_{K_0, K_0', L_0, L_0'} h_{K_0\times L_0'}\right\|_{H^p(H^q)} \left\| h_{K_0'\times L_0}\right\|_{(H^p(H^q))^*}.
    \end{equation}
    Estimate~\eqref{eq:x-est(((0 1 0 1) (0 1 1 0)) (((1) (0)) ((0) (1)) (nil nil))):1} and the
    disjointness of the dyadic intervals (see~\textrefp[J]{enu:j1}) yield
    \begin{equation}\label{eq:x-est(((0 1 0 1) (0 1 1 0)) (((1) (0)) ((0) (1)) (nil nil))):4}
      \sum_{K_0', L_0}\|T\| \left\| \sum_{K_0, L_0'}\max_{K_0, L_0'}\left(\|T\| |K_0|^{1/{p}} |L_0|^{1/{q}} |K_0'|^{1/{p'}} |L_0'|^{1/{q'}}\right) h_{K_0\times L_0'}\right\|_{H^p(H^q)}\left\| h_{K_0'\times L_0}\right\|_{(H^p(H^q))^*}.
    \end{equation}
    Consequently, we obtain
    \begin{equation}\label{eq:x-est(((0 1 0 1) (0 1 1 0)) (((1) (0)) ((0) (1)) (nil nil))):5}
      \sum_{K_0', L_0}\|T\| \max_{K_0, L_0'}\left(\|T\| |K_0|^{1/{p}} |L_0|^{1/{q}} |K_0'|^{1/{p'}} |L_0'|^{1/{q'}}\right) \left\| \sum_{K_0, L_0'}h_{K_0\times L_0'}\right\|_{H^p(H^q)}\left\| h_{K_0'\times L_0}\right\|_{(H^p(H^q))^*}.
    \end{equation}
    Thus,~\eqref{eq:x-est(((0 1 0 1) (0 1 1 0)) (((1) (0)) ((0) (1)) (nil nil))):5} is bounded from
    above by
    \begin{equation*}
      \|T\|^2 \max_{K_0, L_0'} |K_0|^{1/{p}}|L_0'|^{1/{q'}} \sum_{K_0', L_0} |K_0'|^{2/{p'}} |L_0|.
    \end{equation*}
    Using Hölder's inequality yields
    \begin{equation}\label{eq:x-est(((0 1 0 1) (0 1 1 0)) (((1) (0)) ((0) (1)) (nil nil))):6}
      \|T\|^2 \max_{K_0, K_0', L_0, L_0'} |K_0|^{1/{p}} |K_0'|^{2/{p'}-1} |L_0'|^{1/{q'}}
    \end{equation}
    Inserting $|K_0|, |K_0'|, |L_0'| \leq \alpha$ (see~\eqref{eq:alpha-small})
    into~\eqref{eq:x-est(((0 1 0 1) (0 1 1 0)) (((1) (0)) ((0) (1)) (nil nil))):6}, we obtain the
    estimate
    \begin{equation}\label{eq:x-est(((0 1 0 1) (0 1 1 0)) (((1) (0)) ((0) (1)) (nil nil))):7}
      \|T\|^2 \alpha^{1/{p'} + 1/{q'}}.
    \end{equation}
  \end{proofcase}

  \begin{proofcase}[Case~\theproofcase, group~{\textrefp[a]{enu:proof:lem:var:X:a:3}}: $K_0 = K_1 \neq K_0' = K_1'$, $L_0 = L_1' \neq L_0' = L_1$ (((0) (1)) (nil nil) ((1) (0))) -- right variant]\label{case:x-est(((0 1 0 1) (0 1 1 0)) (((0) (1)) (nil nil) ((1) (0)))):right}
    In this case, we have to estimate
    \begin{equation}\label{eq:x-est(((0 1 0 1) (0 1 1 0)) (((0) (1)) (nil nil) ((1) (0)))):0}
      \sum_{K_0, L_0', K_0', L_0} \langle Th_{K_0\times L_0}, h_{K_0'\times L_0'} \rangle \langle Th_{K_0\times L_0'}, h_{K_0'\times L_0} \rangle.
    \end{equation}
    We put $a_{K_0, K_0', L_0, L_0'} = \langle T h_{K_0\times L_0}, h_{K_0'\times L_0'} \rangle$ and
    note the estimate
    \begin{equation}\label{eq:x-est(((0 1 0 1) (0 1 1 0)) (((0) (1)) (nil nil) ((1) (0)))):1}
      |a_{K_0, K_0', L_0, L_0'}|
      \leq \|T\| |K_0|^{1/{p}} |L_0|^{1/{q}} |K_0'|^{1/{p'}} |L_0'|^{1/{q'}}.
    \end{equation}
    Now, we write~\eqref{eq:x-est(((0 1 0 1) (0 1 1 0)) (((0) (1)) (nil nil) ((1) (0)))):0} as
    follows:
    \begin{equation}\label{eq:x-est(((0 1 0 1) (0 1 1 0)) (((0) (1)) (nil nil) ((1) (0)))):2}
      \sum_{K_0, L_0'} \Bigl\langle T  h_{K_0\times L_0'}, \sum_{K_0', L_0} a_{K_0, K_0', L_0, L_0'} h_{K_0'\times L_0} \Bigr\rangle.
    \end{equation}
    By duality, we obtain the subsequent upper estimate for~\eqref{eq:x-est(((0 1 0 1) (0 1 1 0))
      (((0) (1)) (nil nil) ((1) (0)))):2}:
    \begin{equation}\label{eq:x-est(((0 1 0 1) (0 1 1 0)) (((0) (1)) (nil nil) ((1) (0)))):3}
      \sum_{K_0, L_0'}\left\|T  h_{K_0\times L_0'}\right\|_{H^p(H^q)} \left\|\sum_{K_0', L_0} a_{K_0, K_0', L_0, L_0'} h_{K_0'\times L_0}\right\|_{(H^p(H^q))^*}.
    \end{equation}
    Estimate~\eqref{eq:x-est(((0 1 0 1) (0 1 1 0)) (((0) (1)) (nil nil) ((1) (0)))):1} and the
    disjointness of the dyadic intervals (see~\textrefp[J]{enu:j1}) yield
    \begin{equation}\label{eq:x-est(((0 1 0 1) (0 1 1 0)) (((0) (1)) (nil nil) ((1) (0)))):4}
      \sum_{K_0, L_0'}\|T\| \left\| h_{K_0\times L_0'}\right\|_{H^p(H^q)}\left\| \sum_{K_0', L_0}\max_{K_0', L_0}\left(\|T\| |K_0|^{1/{p}} |L_0|^{1/{q}} |K_0'|^{1/{p'}} |L_0'|^{1/{q'}}\right) h_{K_0'\times L_0}\right\|_{(H^p(H^q))^*}.
    \end{equation}
    Consequently, we obtain
    \begin{equation}\label{eq:x-est(((0 1 0 1) (0 1 1 0)) (((0) (1)) (nil nil) ((1) (0)))):5}
      \sum_{K_0, L_0'}\|T\| \left\| h_{K_0\times L_0'}\right\|_{H^p(H^q)}\max_{K_0', L_0}\left(\|T\| |K_0|^{1/{p}} |L_0|^{1/{q}} |K_0'|^{1/{p'}} |L_0'|^{1/{q'}}\right) \left\| \sum_{K_0', L_0}h_{K_0'\times L_0}\right\|_{(H^p(H^q))^*}.
    \end{equation}
    Thus,~\eqref{eq:x-est(((0 1 0 1) (0 1 1 0)) (((0) (1)) (nil nil) ((1) (0)))):5} is bounded from
    above by
    \begin{equation*}
      \|T\|^2 \max_{K_0', L_0} |L_0|^{1/{q}} |K_0'|^{1/{p'}} \sum_{K_0, L_0'} |K_0|^{2/{p}} |L_0'|.
    \end{equation*}
    Using Hölder's inequality yields
    \begin{equation}\label{eq:x-est(((0 1 0 1) (0 1 1 0)) (((0) (1)) (nil nil) ((1) (0)))):6}
      \|T\|^2 \max_{K_0, K_0', L_0} |K_0|^{2/{p}-1} |K_0'|^{1/{p'}} |L_0|^{1/{q}}.
    \end{equation}
    Inserting $|K_0|, |K_0'|, |L_0| \leq \alpha$ (see~\eqref{eq:alpha-small})
    into~\eqref{eq:x-est(((0 1 0 1) (0 1 1 0)) (((0) (1)) (nil nil) ((1) (0)))):6}, we obtain the
    estimate
    \begin{equation}\label{eq:x-est(((0 1 0 1) (0 1 1 0)) (((0) (1)) (nil nil) ((1) (0)))):7}
      \|T\|^2 \alpha^{1/{p} + 1/{q}}.
    \end{equation}
  \end{proofcase}

  \subsubsection*{Summary of \textref[Case~]{case:x-est(((0 1 0 1) (0 1 1 0)) (((1) (0)) ((0) (1))
      (nil nil))):left} and \textref[Case~]{case:x-est(((0 1 0 1) (0 1 1 0)) (((0) (1)) (nil nil)
      ((1) (0)))):right}}
  Combining~\eqref{eq:x-est(((0 1 0 1) (0 1 1 0)) (((1) (0)) ((0) (1)) (nil nil))):7}
  with~\eqref{eq:x-est(((0 1 0 1) (0 1 1 0)) (((0) (1)) (nil nil) ((1) (0)))):7} yields
  \begin{equation}\label{eq:x-est(((0 1 0 1) (0 1 1 0)) (((1) (0)) ((0) (1)) (nil nil)))(((0 1 0 1)
      (0 1 1 0)) (((0) (1)) (nil nil) ((1) (0))))-final}
    \cond_{\theta,\varepsilon} X^2
    \leq \|T\|^2 \alpha.
  \end{equation}

  \begin{proofcase}[Case~\theproofcase, group~{\textrefp[a]{enu:proof:lem:var:X:a:4}}: $K_0 = K_1 \neq K_0' = K_1'$, $L_0 = L_0' \neq L_1 = L_1'$ (((1) (2)) ((0) (0)) (nil nil)) -- left variant]\label{case:x-est(((0 1 0 1) (0 0 2 2)) (((1) (2)) ((0) (0)) (nil nil))):left}
    In this case, we have to estimate
    \begin{equation}\label{eq:x-est(((0 1 0 1) (0 0 2 2)) (((1) (2)) ((0) (0)) (nil nil))):0}
      \sum_{K_0', L_1, K_0, L_0} \langle Th_{K_0\times L_0}, h_{K_0'\times L_0} \rangle \langle Th_{K_0\times L_1}, h_{K_0'\times L_1} \rangle.
    \end{equation}
    We put $a_{K_0, K_0', L_0} = \langle T h_{K_0\times L_0}, h_{K_0'\times L_0} \rangle$ and note
    the estimate
    \begin{equation}\label{eq:x-est(((0 1 0 1) (0 0 2 2)) (((1) (2)) ((0) (0)) (nil nil))):1}
      |a_{K_0, K_0', L_0}|
      \leq \|T\| |K_0|^{1/{p}} |L_0|^{1/{q}} |K_0'|^{1/{p'}} |L_0|^{1/{q'}}.
    \end{equation}
    Now, we write~\eqref{eq:x-est(((0 1 0 1) (0 0 2 2)) (((1) (2)) ((0) (0)) (nil nil))):0} as
    follows:
    \begin{equation}\label{eq:x-est(((0 1 0 1) (0 0 2 2)) (((1) (2)) ((0) (0)) (nil nil))):2}
      \sum_{K_0', L_1} \Bigl\langle T \sum_{K_0, L_0} a_{K_0, K_0', L_0} h_{K_0\times L_1},  h_{K_0'\times L_1} \Bigr\rangle.
    \end{equation}
    By duality, we obtain the subsequent upper estimate for~\eqref{eq:x-est(((0 1 0 1) (0 0 2 2))
      (((1) (2)) ((0) (0)) (nil nil))):2}:
    \begin{equation}\label{eq:x-est(((0 1 0 1) (0 0 2 2)) (((1) (2)) ((0) (0)) (nil nil))):3}
      \sum_{K_0', L_1}\left\|T \sum_{K_0, L_0} a_{K_0, K_0', L_0} h_{K_0\times L_1}\right\|_{H^p(H^q)} \left\| h_{K_0'\times L_1}\right\|_{(H^p(H^q))^*}.
    \end{equation}
    Estimate~\eqref{eq:x-est(((0 1 0 1) (0 0 2 2)) (((1) (2)) ((0) (0)) (nil nil))):1} and the
    disjointness of the dyadic intervals (see~\textrefp[J]{enu:j1}) yield
    \begin{equation}\label{eq:x-est(((0 1 0 1) (0 0 2 2)) (((1) (2)) ((0) (0)) (nil nil))):4}
      \sum_{K_0', L_1}\|T\| \left\| \sum_{K_0, L_0}\max_{K_0}\left(\|T\| |K_0|^{1/{p}} |L_0|^{1/{q}} |K_0'|^{1/{p'}} |L_0|^{1/{q'}}\right) h_{K_0\times L_1}\right\|_{H^p(H^q)}\left\| h_{K_0'\times L_1}\right\|_{(H^p(H^q))^*}.
    \end{equation}
    Consequently, we obtain
    \begin{equation}\label{eq:x-est(((0 1 0 1) (0 0 2 2)) (((1) (2)) ((0) (0)) (nil nil))):5}
      \sum_{K_0', L_1}\|T\| \sum_{L_0}\max_{K_0}\left(\|T\| |K_0|^{1/{p}} |L_0|^{1/{q}} |K_0'|^{1/{p'}} |L_0|^{1/{q'}}\right) \left\| \sum_{K_0}h_{K_0\times L_1}\right\|_{H^p(H^q)}\left\| h_{K_0'\times L_1}\right\|_{(H^p(H^q))^*}.
    \end{equation}
    Thus,~\eqref{eq:x-est(((0 1 0 1) (0 0 2 2)) (((1) (2)) ((0) (0)) (nil nil))):5} is bounded from
    above by
    \begin{equation*}
      \|T\|^2 \max_{K_0}  |K_0|^{1/{p}} \sum_{K_0', L_1} \sum_{L_0} |L_0| |K_0'|^{2/{p'}} |L_1|.
    \end{equation*}
    Using Hölder's inequality yields
    \begin{equation}\label{eq:x-est(((0 1 0 1) (0 0 2 2)) (((1) (2)) ((0) (0)) (nil nil))):6}
      \|T\|^2 \max_{K_0, K_0'}  |K_0|^{1/{p}} |K_0'|^{2/{p'}-1}
    \end{equation}
    Inserting $|K_0|, |K_0'| \leq \alpha$ (see~\eqref{eq:alpha-small}) into~\eqref{eq:x-est(((0 1 0
      1) (0 0 2 2)) (((1) (2)) ((0) (0)) (nil nil))):6}, we obtain the estimate
    \begin{equation}\label{eq:x-est(((0 1 0 1) (0 0 2 2)) (((1) (2)) ((0) (0)) (nil nil))):7}
      \|T\|^2 \alpha^{1/{p'}}.
    \end{equation}
  \end{proofcase}

  \begin{proofcase}[Case~\theproofcase, group~{\textrefp[a]{enu:proof:lem:var:X:a:4}}: $K_0 = K_1 \neq K_0' = K_1'$, $L_0 = L_0' \neq L_1 = L_1'$ (((0) (2)) (nil nil) ((1) (0))) -- right variant]\label{case:x-est(((0 1 0 1) (0 0 2 2)) (((0) (2)) (nil nil) ((1) (0)))):right}
    In this case, we have to estimate
    \begin{equation}\label{eq:x-est(((0 1 0 1) (0 0 2 2)) (((0) (2)) (nil nil) ((1) (0)))):0}
      \sum_{K_0, L_1, K_0', L_0} \langle Th_{K_0\times L_0}, h_{K_0'\times L_0} \rangle \langle Th_{K_0\times L_1}, h_{K_0'\times L_1} \rangle.
    \end{equation}
    We put $a_{K_0, K_0', L_0} = \langle T h_{K_0\times L_0}, h_{K_0'\times L_0} \rangle$ and note
    the estimate
    \begin{equation}\label{eq:x-est(((0 1 0 1) (0 0 2 2)) (((0) (2)) (nil nil) ((1) (0)))):1}
      |a_{K_0, K_0', L_0}|
      \leq \|T\| |K_0|^{1/{p}} |L_0|^{1/{q}} |K_0'|^{1/{p'}} |L_0|^{1/{q'}}.
    \end{equation}
    Now, we write~\eqref{eq:x-est(((0 1 0 1) (0 0 2 2)) (((0) (2)) (nil nil) ((1) (0)))):0} as
    follows:
    \begin{equation}\label{eq:x-est(((0 1 0 1) (0 0 2 2)) (((0) (2)) (nil nil) ((1) (0)))):2}
      \sum_{K_0, L_1} \Bigl\langle T  h_{K_0\times L_1}, \sum_{K_0', L_0} a_{K_0, K_0', L_0} h_{K_0'\times L_1} \Bigr\rangle.
    \end{equation}
    By duality, we obtain the subsequent upper estimate for~\eqref{eq:x-est(((0 1 0 1) (0 0 2 2))
      (((0) (2)) (nil nil) ((1) (0)))):2}:
    \begin{equation}\label{eq:x-est(((0 1 0 1) (0 0 2 2)) (((0) (2)) (nil nil) ((1) (0)))):3}
      \sum_{K_0, L_1}\left\|T  h_{K_0\times L_1}\right\|_{H^p(H^q)} \left\|\sum_{K_0', L_0} a_{K_0, K_0', L_0} h_{K_0'\times L_1}\right\|_{(H^p(H^q))^*}.
    \end{equation}
    Estimate~\eqref{eq:x-est(((0 1 0 1) (0 0 2 2)) (((0) (2)) (nil nil) ((1) (0)))):1} and the
    disjointness of the dyadic intervals (see~\textrefp[J]{enu:j1}) yield
    \begin{equation}\label{eq:x-est(((0 1 0 1) (0 0 2 2)) (((0) (2)) (nil nil) ((1) (0)))):4}
      \sum_{K_0, L_1}\|T\| \left\| h_{K_0\times L_1}\right\|_{H^p(H^q)}\left\| \sum_{K_0', L_0}\max_{K_0'}\left(\|T\| |K_0|^{1/{p}} |L_0|^{1/{q}} |K_0'|^{1/{p'}} |L_0|^{1/{q'}}\right) h_{K_0'\times L_1}\right\|_{(H^p(H^q))^*}.
    \end{equation}
    Consequently, we obtain
    \begin{equation}\label{eq:x-est(((0 1 0 1) (0 0 2 2)) (((0) (2)) (nil nil) ((1) (0)))):5}
      \sum_{K_0, L_1}\|T\| \left\| h_{K_0\times L_1}\right\|_{H^p(H^q)}\sum_{L_0}\max_{K_0'}\left(\|T\| |K_0|^{1/{p}} |L_0|^{1/{q}} |K_0'|^{1/{p'}} |L_0|^{1/{q'}}\right) \left\| \sum_{K_0'}h_{K_0'\times L_1}\right\|_{(H^p(H^q))^*}.
    \end{equation}
    Thus,~\eqref{eq:x-est(((0 1 0 1) (0 0 2 2)) (((0) (2)) (nil nil) ((1) (0)))):5} is bounded from
    above by
    \begin{equation*}
      \|T\|^2 \max_{K_0'} |K_0'|^{1/{p'}} \sum_{K_0, L_1} \sum_{L_0} |K_0|^{2/{p}} |L_0| |L_1|.
    \end{equation*}
    Using Hölder's inequality yields
    \begin{equation}\label{eq:x-est(((0 1 0 1) (0 0 2 2)) (((0) (2)) (nil nil) ((1) (0)))):6}
      \|T\|^2 \max_{K_0, K_0'} |K_0|^{2/{p}-1} |K_0'|^{1/{p'}}.
    \end{equation}
    Inserting $|K_0|, |K_0'| \leq \alpha$ (see~\eqref{eq:alpha-small}) into~\eqref{eq:x-est(((0 1 0
      1) (0 0 2 2)) (((0) (2)) (nil nil) ((1) (0)))):6}, we obtain the estimate
    \begin{equation}\label{eq:x-est(((0 1 0 1) (0 0 2 2)) (((0) (2)) (nil nil) ((1) (0)))):7}
      \|T\|^2 \alpha^{1/{p}}.
    \end{equation}
  \end{proofcase}
  
  \subsubsection*{Summary of \textref[Case~]{case:x-est(((0 1 0 1) (0 0 2 2)) (((1) (2)) ((0) (0))
      (nil nil))):left} and \textref[Case~]{case:x-est(((0 1 0 1) (0 0 2 2)) (((0) (2)) (nil nil)
      ((1) (0)))):right}}
  Combining~\eqref{eq:x-est(((0 1 0 1) (0 0 2 2)) (((1) (2)) ((0) (0)) (nil nil))):7}
  with~\eqref{eq:x-est(((0 1 0 1) (0 0 2 2)) (((0) (2)) (nil nil) ((1) (0)))):7} yields
  \begin{equation}\label{eq:x-est(((0 1 0 1) (0 0 2 2)) (((1) (2)) ((0) (0)) (nil nil)))(((0 1 0 1)
      (0 0 2 2)) (((0) (2)) (nil nil) ((1) (0))))-final}
    \cond_{\theta,\varepsilon} X^2
    \leq \|T\|^2 \alpha^{1/2}.
  \end{equation}

  \subsubsection*{Summary for $X$}
  Combining~\eqref{eq:x-est(((0 1 0 1) (0 0 0 0)) (((1) (0)) ((0) nil) (nil nil))):7}
  with~\eqref{eq:x-est(((0 1 0 1) (0 1 0 1)) (((1) (1)) ((0) (0)) (nil nil)))(((0 1 0 1) (0 1 0 1))
    (((0) (0)) (nil nil) ((1) (1))))-final}, \eqref{eq:x-est(((0 1 0 1) (0 1 1 0)) (((1) (0)) ((0)
    (1)) (nil nil)))(((0 1 0 1) (0 1 1 0)) (((0) (1)) (nil nil) ((1) (0))))-final}
  and~\eqref{eq:x-est(((0 1 0 1) (0 0 2 2)) (((1) (2)) ((0) (0)) (nil nil)))(((0 1 0 1) (0 0 2 2))
    (((0) (2)) (nil nil) ((1) (0))))-final} yields
  \begin{equation}\label{eq:x-est:final}
    \cond_{\theta,\varepsilon} X^2
    \leq 4 \|T\|^2 \alpha^{1/2}.
  \end{equation}
\end{myproof}

\subsection{Estimates for $Y$}

In this case, the following variables will always be summed over the following sets:
\begin{itemize}
\item
  $K_0$, $K_1$, $K_0'$, $K_1'$ over $\mathcal{X}_I$;
\item
  $L_0$, $L_1$ over $\mathcal{Y}_J$;
\item
  $L_0'$, $L_1'$ over $\mathcal{Y}_{J'}$.
\end{itemize}

\begin{myproof}
  Note that by~\eqref{eq:dfn:rv:Y} and~\eqref{eq:b:dfn} we obtain $Y^2(\theta,\varepsilon)$ is given
  by
  \begin{equation}\label{eq:proof:lem:var:Y}
    \sum_{\substack{K_0,K_1,K_0',K_1'\\L_0,L_1,L_0',L_1'}}
    \theta_{K_0}\theta_{K_1}\theta_{K_0'}\theta_{K_1'}
    \varepsilon_{L_0}\varepsilon_{L_1}\varepsilon_{L_0'}\varepsilon_{L_1'}
    \langle T h_{K_0\times L_0}, h_{K_0'\times L_0'}\rangle
    \langle T h_{K_1\times L_1}, h_{K_1'\times L_1'}\rangle
  \end{equation}
  Note that since in this case $J\neq J'$, we have that
  $\mathcal{Y}_J\cap \mathcal{Y}_{J'} = \emptyset$, by~\textrefp[J]{enu:j1}.  Thus,
  $\cond_{\varepsilon} \varepsilon_{L_0}\varepsilon_{L_0'}\varepsilon_{L_1}\varepsilon_{L_1'} \neq
  0$, only if $L_0 = L_1 \neq L_0' = L_1'$.  Hence, in view
  of~\textrefp[R]{enu:proof:lem:var:r1}--\textrefp[R]{enu:proof:lem:var:r4}, we decompose the index
  set in~\eqref{eq:proof:lem:var:Y} into the following four groups:
  \begin{enumerate}[({b}1)]
  \item\label{enu:proof:lem:var:Y:b:1} $K_0 = K_1 = K_0' = K_1'$ and $L_0 = L_1 \neq L_0' = L_1'$;
  \item\label{enu:proof:lem:var:Y:b:2} $K_0 = K_1 \neq K_0' = K_1'$ and
    $L_0 = L_1 \neq L_0' = L_1'$;
  \item\label{enu:proof:lem:var:Y:b:3} $K_0 = K_1' \neq K_0' = K_1$ and
    $L_0 = L_1 \neq L_0' = L_1'$;
  \item\label{enu:proof:lem:var:Y:b:4} $K_0 = K_0' \neq K_1 = K_1'$ and
    $L_0 = L_1 \neq L_0' = L_1'$.
  \end{enumerate}

  \begin{proofcase}[Case~\theproofcase, group~{\textrefp[b]{enu:proof:lem:var:Y:b:1}}: $K_0 = K_0' = K_1 = K_1'$, $L_0 = L_1 \neq L_0' = L_1'$
    (((0) (1)) (nil (0)) (nil nil)) -- left variant]\label{case:y-est(((0 0 0 0) (0 1 0 1)) (((0)
      (1)) (nil (0)) (nil nil))):left}
    In this case, we have to estimate
    \begin{equation}\label{eq:y-est(((0 0 0 0) (0 1 0 1)) (((0) (1)) (nil (0)) (nil nil))):0}
      \sum_{K_0, L_0', L_0} \langle Th_{K_0\times L_0}, h_{K_0\times L_0'} \rangle \langle Th_{K_0\times L_0}, h_{K_0\times L_0'} \rangle.
    \end{equation}
    We put $a_{K_0, L_0, L_0'} = \langle T h_{K_0\times L_0}, h_{K_0\times L_0'} \rangle$ and note
    the estimate
    \begin{equation}\label{eq:y-est(((0 0 0 0) (0 1 0 1)) (((0) (1)) (nil (0)) (nil nil))):1}
      |a_{K_0, L_0, L_0'}|
      \leq \|T\| |K_0|^{1/{p}} |L_0|^{1/{q}} |K_0|^{1/{p'}} |L_0'|^{1/{q'}}.
    \end{equation}
    Now, we write~\eqref{eq:y-est(((0 0 0 0) (0 1 0 1)) (((0) (1)) (nil (0)) (nil nil))):0} as
    follows:
    \begin{equation}\label{eq:y-est(((0 0 0 0) (0 1 0 1)) (((0) (1)) (nil (0)) (nil nil))):2}
      \sum_{K_0, L_0'} \Bigl\langle T \sum_{L_0} a_{K_0, L_0, L_0'} h_{K_0\times L_0},  h_{K_0\times L_0'} \Bigr\rangle.
    \end{equation}
    By duality, we obtain the subsequent upper estimate for~\eqref{eq:y-est(((0 0 0 0) (0 1 0 1))
      (((0) (1)) (nil (0)) (nil nil))):2}:
    \begin{equation}\label{eq:y-est(((0 0 0 0) (0 1 0 1)) (((0) (1)) (nil (0)) (nil nil))):3}
      \sum_{K_0, L_0'}\left\|T \sum_{L_0} a_{K_0, L_0, L_0'} h_{K_0\times L_0}\right\|_{H^p(H^q)} \left\| h_{K_0\times L_0'}\right\|_{(H^p(H^q))^*}.
    \end{equation}
    Estimate~\eqref{eq:y-est(((0 0 0 0) (0 1 0 1)) (((0) (1)) (nil (0)) (nil nil))):1} and the
    disjointness of the dyadic intervals (see~\textrefp[J]{enu:j1}) yield
    \begin{equation}\label{eq:y-est(((0 0 0 0) (0 1 0 1)) (((0) (1)) (nil (0)) (nil nil))):4}
      \sum_{K_0, L_0'}\|T\| \left\| \sum_{L_0}\max_{L_0}\left(\|T\| |K_0|^{1/{p}} |L_0|^{1/{q}} |K_0|^{1/{p'}} |L_0'|^{1/{q'}}\right) h_{K_0\times L_0}\right\|_{H^p(H^q)}\left\| h_{K_0\times L_0'}\right\|_{(H^p(H^q))^*}.
    \end{equation}
    Consequently, we obtain
    \begin{equation}\label{eq:y-est(((0 0 0 0) (0 1 0 1)) (((0) (1)) (nil (0)) (nil nil))):5}
      \sum_{K_0, L_0'}\|T\| \max_{L_0}\left(\|T\| |K_0|^{1/{p}} |L_0|^{1/{q}} |K_0|^{1/{p'}} |L_0'|^{1/{q'}}\right) \left\| \sum_{L_0}h_{K_0\times L_0}\right\|_{H^p(H^q)}\left\| h_{K_0\times L_0'}\right\|_{(H^p(H^q))^*}.
    \end{equation}
    Thus,~\eqref{eq:y-est(((0 0 0 0) (0 1 0 1)) (((0) (1)) (nil (0)) (nil nil))):5} is bounded from
    above by
    \begin{equation*}
      \|T\|^2 \max_{L_0} |L_0|^{1/{q}} \sum_{K_0, L_0'} |K_0|^2 |L_0'|^{2/{q'}}.
    \end{equation*}
    Using Hölder's inequality yields
    \begin{equation}\label{eq:y-est(((0 0 0 0) (0 1 0 1)) (((0) (1)) (nil (0)) (nil nil))):6}
      \|T\|^2\max_{K_0, L_0, L_0'} |K_0| |L_0|^{1/{q}} |L_0'|^{2/{q'}-1}.
    \end{equation}
    Inserting $|K_0|, |L_0|, |L_0'| \leq \alpha$ (see~\eqref{eq:alpha-small})
    into~\eqref{eq:y-est(((0 0 0 0) (0 1 0 1)) (((0) (1)) (nil (0)) (nil nil))):6}, we obtain the
    estimate
    \begin{equation}\label{eq:y-est(((0 0 0 0) (0 1 0 1)) (((0) (1)) (nil (0)) (nil nil))):7}
      \|T\|^2 \alpha^{1+1/{q'}}.
    \end{equation}
  \end{proofcase}

  \begin{proofcase}[Case~\theproofcase, group~{\textrefp[b]{enu:proof:lem:var:Y:b:2}}: $K_0 = K_1 \neq K_0' = K_1'$, $L_0 = L_1 \neq L_0' = L_1'$ (((1) (1)) ((0) (0)) (nil nil)) -- left variant]\label{case:y-est(((0 1 0 1) (0 1 0 1)) (((1) (1)) ((0) (0)) (nil nil))):left}
    In this case, we have to estimate
    \begin{equation}\label{eq:y-est(((0 1 0 1) (0 1 0 1)) (((1) (1)) ((0) (0)) (nil nil))):0}
      \sum_{K_0', L_0', K_0, L_0} \langle Th_{K_0\times L_0}, h_{K_0'\times L_0'} \rangle \langle Th_{K_0\times L_0}, h_{K_0'\times L_0'} \rangle.
    \end{equation}
    We put $a_{K_0, K_0', L_0, L_0'} = \langle T h_{K_0\times L_0}, h_{K_0'\times L_0'} \rangle$ and
    note the estimate
    \begin{equation}\label{eq:y-est(((0 1 0 1) (0 1 0 1)) (((1) (1)) ((0) (0)) (nil nil))):1}
      |a_{K_0, K_0', L_0, L_0'}|
      \leq \|T\| |K_0|^{1/{p}} |L_0|^{1/{q}} |K_0'|^{1/{p'}} |L_0'|^{1/{q'}}.
    \end{equation}
    Now, we write~\eqref{eq:y-est(((0 1 0 1) (0 1 0 1)) (((1) (1)) ((0) (0)) (nil nil))):0} as
    follows:
    \begin{equation}\label{eq:y-est(((0 1 0 1) (0 1 0 1)) (((1) (1)) ((0) (0)) (nil nil))):2}
      \sum_{K_0', L_0'} \Bigl\langle T \sum_{K_0, L_0} a_{K_0, K_0', L_0, L_0'} h_{K_0\times L_0},  h_{K_0'\times L_0'} \Bigr\rangle.
    \end{equation}
    By duality, we obtain the subsequent upper estimate for~\eqref{eq:y-est(((0 1 0 1) (0 1 0 1))
      (((1) (1)) ((0) (0)) (nil nil))):2}:
    \begin{equation}\label{eq:y-est(((0 1 0 1) (0 1 0 1)) (((1) (1)) ((0) (0)) (nil nil))):3}
      \sum_{K_0', L_0'}\left\|T \sum_{K_0, L_0} a_{K_0, K_0', L_0, L_0'} h_{K_0\times L_0}\right\|_{H^p(H^q)} \left\| h_{K_0'\times L_0'}\right\|_{(H^p(H^q))^*}.
    \end{equation}
    Estimate~\eqref{eq:y-est(((0 1 0 1) (0 1 0 1)) (((1) (1)) ((0) (0)) (nil nil))):1} and the
    disjointness of the dyadic intervals (see~\textrefp[J]{enu:j1}) yield
    \begin{equation}\label{eq:y-est(((0 1 0 1) (0 1 0 1)) (((1) (1)) ((0) (0)) (nil nil))):4}
      \sum_{K_0', L_0'}\|T\| \left\| \sum_{K_0, L_0}\max_{K_0, L_0}\left(\|T\| |K_0|^{1/{p}} |L_0|^{1/{q}} |K_0'|^{1/{p'}} |L_0'|^{1/{q'}}\right) h_{K_0\times L_0}\right\|_{H^p(H^q)}\left\| h_{K_0'\times L_0'}\right\|_{(H^p(H^q))^*}.
    \end{equation}
    Consequently, we obtain
    \begin{equation}\label{eq:y-est(((0 1 0 1) (0 1 0 1)) (((1) (1)) ((0) (0)) (nil nil))):5}
      \sum_{K_0', L_0'}\|T\| \max_{K_0, L_0}\left(\|T\| |K_0|^{1/{p}} |L_0|^{1/{q}} |K_0'|^{1/{p'}} |L_0'|^{1/{q'}}\right) \left\| \sum_{K_0, L_0}h_{K_0\times L_0}\right\|_{H^p(H^q)}\left\| h_{K_0'\times L_0'}\right\|_{(H^p(H^q))^*}.
    \end{equation}
    Thus,~\eqref{eq:y-est(((0 1 0 1) (0 1 0 1)) (((1) (1)) ((0) (0)) (nil nil))):5} is bounded from
    above by
    \begin{equation*}
      \|T\|^2 \max_{K_0, L_0}  |K_0|^{1/{p}} |L_0|^{1/{q}} \sum_{K_0', L_0'}  |K_0'|^{2/{p'}} |L_0'|^{2/{q'}}.
    \end{equation*}
    Using Hölder's inequality yields
    \begin{equation}\label{eq:y-est(((0 1 0 1) (0 1 0 1)) (((1) (1)) ((0) (0)) (nil nil))):6}
      \|T\|^2 \max_{K_0, L_0, K_0', L_0'}  |K_0|^{1/{p}} |L_0|^{1/{q}} |K_0'|^{2/{p'}-1} |L_0'|^{2/{q'}-1}.
    \end{equation}
    Inserting $|K_0|, |K_0'|, |L_0|, |L_0'| \leq \alpha$ (see~\eqref{eq:alpha-small})
    into~\eqref{eq:y-est(((0 1 0 1) (0 1 0 1)) (((1) (1)) ((0) (0)) (nil nil))):6}, we obtain the
    estimate
    \begin{equation}\label{eq:y-est(((0 1 0 1) (0 1 0 1)) (((1) (1)) ((0) (0)) (nil nil))):7}
      \|T\|^2 \alpha^{1/{p'}+1/{q'}}.
    \end{equation}
  \end{proofcase}

  \begin{proofcase}[Case~\theproofcase, group~{\textrefp[b]{enu:proof:lem:var:Y:b:2}}: $K_0 = K_1 \neq K_0' = K_1'$, $L_0 = L_1 \neq L_0' = L_1'$ (((0) (0)) (nil nil) ((1) (1))) -- right variant]\label{case:y-est(((0 1 0 1) (0 1 0 1)) (((0) (0)) (nil nil) ((1) (1)))):right}
    In this case, we have to estimate
    \begin{equation}\label{eq:y-est(((0 1 0 1) (0 1 0 1)) (((0) (0)) (nil nil) ((1) (1)))):0}
      \sum_{K_0, L_0, K_0', L_0'} \langle Th_{K_0\times L_0}, h_{K_0'\times L_0'} \rangle \langle Th_{K_0\times L_0}, h_{K_0'\times L_0'} \rangle.
    \end{equation}
    We put $a_{K_0, K_0', L_0, L_0'} = \langle T h_{K_0\times L_0}, h_{K_0'\times L_0'} \rangle$ and
    note the estimate
    \begin{equation}\label{eq:y-est(((0 1 0 1) (0 1 0 1)) (((0) (0)) (nil nil) ((1) (1)))):1}
      |a_{K_0, K_0', L_0, L_0'}|
      \leq \|T\| |K_0|^{1/{p}} |L_0|^{1/{q}} |K_0'|^{1/{p'}} |L_0'|^{1/{q'}}.
    \end{equation}
    Now, we write~\eqref{eq:y-est(((0 1 0 1) (0 1 0 1)) (((0) (0)) (nil nil) ((1) (1)))):0} as
    follows:
    \begin{equation}\label{eq:y-est(((0 1 0 1) (0 1 0 1)) (((0) (0)) (nil nil) ((1) (1)))):2}
      \sum_{K_0, L_0} \Bigl\langle T  h_{K_0\times L_0}, \sum_{K_0', L_0'} a_{K_0, K_0', L_0, L_0'} h_{K_0'\times L_0'} \Bigr\rangle.
    \end{equation}
    By duality, we obtain the subsequent upper estimate for~\eqref{eq:y-est(((0 1 0 1) (0 1 0 1))
      (((0) (0)) (nil nil) ((1) (1)))):2}:
    \begin{equation}\label{eq:y-est(((0 1 0 1) (0 1 0 1)) (((0) (0)) (nil nil) ((1) (1)))):3}
      \sum_{K_0, L_0}\left\|T  h_{K_0\times L_0}\right\|_{H^p(H^q)} \left\|\sum_{K_0', L_0'} a_{K_0, K_0', L_0, L_0'} h_{K_0'\times L_0'}\right\|_{(H^p(H^q))^*}.
    \end{equation}
    Estimate~\eqref{eq:y-est(((0 1 0 1) (0 1 0 1)) (((0) (0)) (nil nil) ((1) (1)))):1} and the
    disjointness of the dyadic intervals (see~\textrefp[J]{enu:j1}) yield
    \begin{equation}\label{eq:y-est(((0 1 0 1) (0 1 0 1)) (((0) (0)) (nil nil) ((1) (1)))):4}
      \sum_{K_0, L_0}\|T\| \left\| h_{K_0\times L_0}\right\|_{H^p(H^q)}\left\| \sum_{K_0', L_0'}\max_{K_0', L_0'}\left(\|T\| |K_0|^{1/{p}} |L_0|^{1/{q}} |K_0'|^{1/{p'}} |L_0'|^{1/{q'}}\right) h_{K_0'\times L_0'}\right\|_{(H^p(H^q))^*}.
    \end{equation}
    Consequently, we obtain
    \begin{equation}\label{eq:y-est(((0 1 0 1) (0 1 0 1)) (((0) (0)) (nil nil) ((1) (1)))):5}
      \sum_{K_0, L_0}\|T\| \left\| h_{K_0\times L_0}\right\|_{H^p(H^q)}\max_{K_0', L_0'}\left(\|T\| |K_0|^{1/{p}} |L_0|^{1/{q}} |K_0'|^{1/{p'}} |L_0'|^{1/{q'}}\right) \left\| \sum_{K_0', L_0'}h_{K_0'\times L_0'}\right\|_{(H^p(H^q))^*}.
    \end{equation}
    Thus,~\eqref{eq:y-est(((0 1 0 1) (0 1 0 1)) (((0) (0)) (nil nil) ((1) (1)))):5} is bounded from
    above by
    \begin{equation*}
      \|T\|^2 \max_{K_0', L_0'} |K_0'|^{1/{p'}} |L_0'|^{1/{q'}} \sum_{K_0, L_0} |K_0|^{2/{p}} |L_0|^{2/{q}}.
    \end{equation*}
    Using Hölder's inequality yields
    \begin{equation}\label{eq:y-est(((0 1 0 1) (0 1 0 1)) (((0) (0)) (nil nil) ((1) (1)))):6}
      \|T\|^2 \max_{K_0, K_0', L_0, L_0'} |K_0|^{2/{p}-1} |L_0|^{2/{q}-1} |K_0'|^{1/{p'}} |L_0'|^{1/{q'}}.
    \end{equation}
    Inserting $|K_0|, |K_0'|, |L_0|, |L_0'| \leq \alpha$ (see~\eqref{eq:alpha-small})
    into~\eqref{eq:y-est(((0 1 0 1) (0 1 0 1)) (((0) (0)) (nil nil) ((1) (1)))):6}, we obtain the
    estimate
    \begin{equation}\label{eq:y-est(((0 1 0 1) (0 1 0 1)) (((0) (0)) (nil nil) ((1) (1)))):7}
      \|T\|^2 \alpha^{1/{p}+1/{q}}.
    \end{equation}
  \end{proofcase}

  \subsubsection*{Summary of \textref[Case~]{case:y-est(((0 1 0 1) (0 1 0 1)) (((1) (1)) ((0) (0))
      (nil nil))):left} and \textref[Case~]{case:y-est(((0 1 0 1) (0 1 0 1)) (((0) (0)) (nil nil)
      ((1) (1)))):right}}
  Combining~\eqref{eq:y-est(((0 1 0 1) (0 1 0 1)) (((1) (1)) ((0) (0)) (nil nil))):7}
  with~\eqref{eq:y-est(((0 1 0 1) (0 1 0 1)) (((0) (0)) (nil nil) ((1) (1)))):7} yields
  \begin{equation}\label{eq:y-est(((0 1 0 1) (0 1 0 1)) (((1) (1)) ((0) (0)) (nil nil)))(((0 1 0 1)
      (0 1 0 1)) (((0) (0)) (nil nil) ((1) (1))))-final}
    \cond_{\theta,\varepsilon} Y^2
    \leq \|T\|^2 \alpha.
  \end{equation}

  \begin{proofcase}[Case~\theproofcase, group~{\textrefp[b]{enu:proof:lem:var:Y:b:3}}: $K_0 = K_1' \neq K_0' = K_1$, $L_0 = L_1 \neq L_0' = L_1'$ (((0) (1)) ((1) (0)) (nil nil)) -- left variant]\label{case:y-est(((0 1 1 0) (0 1 0 1)) (((0) (1)) ((1) (0)) (nil nil))):left}
    In this case, we have to estimate
    \begin{equation}\label{eq:y-est(((0 1 1 0) (0 1 0 1)) (((0) (1)) ((1) (0)) (nil nil))):0}
      \sum_{K_0, L_0', K_0', L_0} \langle Th_{K_0\times L_0}, h_{K_0'\times L_0'} \rangle \langle Th_{K_0'\times L_0}, h_{K_0\times L_0'} \rangle.
    \end{equation}
    We put $a_{K_0, K_0', L_0, L_0'} = \langle T h_{K_0\times L_0}, h_{K_0'\times L_0'} \rangle$ and
    note the estimate
    \begin{equation}\label{eq:y-est(((0 1 1 0) (0 1 0 1)) (((0) (1)) ((1) (0)) (nil nil))):1}
      |a_{K_0, K_0', L_0, L_0'}|
      \leq \|T\| |K_0|^{1/{p}} |L_0|^{1/{q}} |K_0'|^{1/{p'}} |L_0'|^{1/{q'}}.
    \end{equation}
    Now, we write~\eqref{eq:y-est(((0 1 1 0) (0 1 0 1)) (((0) (1)) ((1) (0)) (nil nil))):0} as
    follows:
    \begin{equation}\label{eq:y-est(((0 1 1 0) (0 1 0 1)) (((0) (1)) ((1) (0)) (nil nil))):2}
      \sum_{K_0, L_0'} \Bigl\langle T \sum_{K_0', L_0} a_{K_0, K_0', L_0, L_0'} h_{K_0'\times L_0},  h_{K_0\times L_0'} \Bigr\rangle.
    \end{equation}
    By duality, we obtain the subsequent upper estimate for~\eqref{eq:y-est(((0 1 1 0) (0 1 0 1))
      (((0) (1)) ((1) (0)) (nil nil))):2}:
    \begin{equation}\label{eq:y-est(((0 1 1 0) (0 1 0 1)) (((0) (1)) ((1) (0)) (nil nil))):3}
      \sum_{K_0, L_0'}\left\|T \sum_{K_0', L_0} a_{K_0, K_0', L_0, L_0'} h_{K_0'\times L_0}\right\|_{H^p(H^q)} \left\| h_{K_0\times L_0'}\right\|_{(H^p(H^q))^*}.
    \end{equation}
    Estimate~\eqref{eq:y-est(((0 1 1 0) (0 1 0 1)) (((0) (1)) ((1) (0)) (nil nil))):1} and the
    disjointness of the dyadic intervals (see~\textrefp[J]{enu:j1}) yield
    \begin{equation}\label{eq:y-est(((0 1 1 0) (0 1 0 1)) (((0) (1)) ((1) (0)) (nil nil))):4}
      \sum_{K_0, L_0'}\|T\| \left\| \sum_{K_0', L_0}\max_{K_0', L_0}\left(\|T\| |K_0|^{1/{p}} |L_0|^{1/{q}} |K_0'|^{1/{p'}} |L_0'|^{1/{q'}}\right) h_{K_0'\times L_0}\right\|_{H^p(H^q)}\left\| h_{K_0\times L_0'}\right\|_{(H^p(H^q))^*}.
    \end{equation}
    Consequently, we obtain
    \begin{equation}\label{eq:y-est(((0 1 1 0) (0 1 0 1)) (((0) (1)) ((1) (0)) (nil nil))):5}
      \sum_{K_0, L_0'}\|T\| \max_{K_0', L_0}\left(\|T\| |K_0|^{1/{p}} |L_0|^{1/{q}} |K_0'|^{1/{p'}} |L_0'|^{1/{q'}}\right) \left\| \sum_{K_0', L_0}h_{K_0'\times L_0}\right\|_{H^p(H^q)}\left\| h_{K_0\times L_0'}\right\|_{(H^p(H^q))^*}.
    \end{equation}
    Thus,~\eqref{eq:y-est(((0 1 1 0) (0 1 0 1)) (((0) (1)) ((1) (0)) (nil nil))):5} is bounded from
    above by
    \begin{equation*}
      \|T\|^2 \max_{K_0', L_0} |L_0|^{1/{q}} |K_0'|^{1/{p'}} \sum_{K_0, L_0'} |K_0| |L_0'|^{2/{q'}}.
    \end{equation*}
    Using Hölder's inequality yields
    \begin{equation}\label{eq:y-est(((0 1 1 0) (0 1 0 1)) (((0) (1)) ((1) (0)) (nil nil))):6}
      \|T\|^2 \max_{K_0', L_0, L_0'} |L_0|^{1/{q}} |K_0'|^{1/{p'}} |L_0'|^{2/{q'}-1}.
    \end{equation}
    Inserting $|K_0'|, |L_0|, |L_0'| \leq \alpha$ (see~\eqref{eq:alpha-small})
    into~\eqref{eq:y-est(((0 1 1 0) (0 1 0 1)) (((0) (1)) ((1) (0)) (nil nil))):6}, we obtain the
    estimate
    \begin{equation}\label{eq:y-est(((0 1 1 0) (0 1 0 1)) (((0) (1)) ((1) (0)) (nil nil))):7}
      \|T\|^2 \alpha^{1/{p'}+1/{q'}}.
    \end{equation}
  \end{proofcase}

  \begin{proofcase}[Case~\theproofcase, group~{\textrefp[b]{enu:proof:lem:var:Y:b:3}}: $K_0 = K_1' \neq K_0' = K_1$, $L_0 = L_1 \neq L_0' = L_1'$ (((1) (0)) (nil nil) ((0) (1))) -- right variant]\label{case:y-est(((0 1 1 0) (0 1 0 1)) (((1) (0)) (nil nil) ((0) (1)))):right}
    In this case, we have to estimate
    \begin{equation}\label{eq:y-est(((0 1 1 0) (0 1 0 1)) (((1) (0)) (nil nil) ((0) (1)))):0}
      \sum_{K_0', L_0, K_0, L_0'} \langle Th_{K_0\times L_0}, h_{K_0'\times L_0'} \rangle \langle Th_{K_0'\times L_0}, h_{K_0\times L_0'} \rangle.
    \end{equation}
    We put $a_{K_0, K_0', L_0, L_0'} = \langle T h_{K_0\times L_0}, h_{K_0'\times L_0'} \rangle$ and
    note the estimate
    \begin{equation}\label{eq:y-est(((0 1 1 0) (0 1 0 1)) (((1) (0)) (nil nil) ((0) (1)))):1}
      |a_{K_0, K_0', L_0, L_0'}|
      \leq \|T\| |K_0|^{1/{p}} |L_0|^{1/{q}} |K_0'|^{1/{p'}} |L_0'|^{1/{q'}}.
    \end{equation}
    Now, we write~\eqref{eq:y-est(((0 1 1 0) (0 1 0 1)) (((1) (0)) (nil nil) ((0) (1)))):0} as
    follows:
    \begin{equation}\label{eq:y-est(((0 1 1 0) (0 1 0 1)) (((1) (0)) (nil nil) ((0) (1)))):2}
      \sum_{K_0', L_0} \Bigl\langle T  h_{K_0'\times L_0}, \sum_{K_0, L_0'} a_{K_0, K_0', L_0, L_0'} h_{K_0\times L_0'} \Bigr\rangle.
    \end{equation}
    By duality, we obtain the subsequent upper estimate for~\eqref{eq:y-est(((0 1 1 0) (0 1 0 1))
      (((1) (0)) (nil nil) ((0) (1)))):2}:
    \begin{equation}\label{eq:y-est(((0 1 1 0) (0 1 0 1)) (((1) (0)) (nil nil) ((0) (1)))):3}
      \sum_{K_0', L_0}\left\|T  h_{K_0'\times L_0}\right\|_{H^p(H^q)} \left\|\sum_{K_0, L_0'} a_{K_0, K_0', L_0, L_0'} h_{K_0\times L_0'}\right\|_{(H^p(H^q))^*}.
    \end{equation}
    Estimate~\eqref{eq:y-est(((0 1 1 0) (0 1 0 1)) (((1) (0)) (nil nil) ((0) (1)))):1} and the
    disjointness of the dyadic intervals (see~\textrefp[J]{enu:j1}) yield
    \begin{equation}\label{eq:y-est(((0 1 1 0) (0 1 0 1)) (((1) (0)) (nil nil) ((0) (1)))):4}
      \sum_{K_0', L_0}\|T\| \left\| h_{K_0'\times L_0}\right\|_{H^p(H^q)}\left\| \sum_{K_0, L_0'}\max_{K_0, L_0'}\left(\|T\| |K_0|^{1/{p}} |L_0|^{1/{q}} |K_0'|^{1/{p'}} |L_0'|^{1/{q'}}\right) h_{K_0\times L_0'}\right\|_{(H^p(H^q))^*}.
    \end{equation}
    Consequently, we obtain
    \begin{equation}\label{eq:y-est(((0 1 1 0) (0 1 0 1)) (((1) (0)) (nil nil) ((0) (1)))):5}
      \sum_{K_0', L_0}\|T\| \left\| h_{K_0'\times L_0}\right\|_{H^p(H^q)}\max_{K_0, L_0'}\left(\|T\| |K_0|^{1/{p}} |L_0|^{1/{q}} |K_0'|^{1/{p'}} |L_0'|^{1/{q'}}\right) \left\| \sum_{K_0, L_0'}h_{K_0\times L_0'}\right\|_{(H^p(H^q))^*}.
    \end{equation}
    Thus,~\eqref{eq:y-est(((0 1 1 0) (0 1 0 1)) (((1) (0)) (nil nil) ((0) (1)))):5} is bounded from
    above by
    \begin{equation*}
      \|T\|^2 \max_{K_0, L_0'} |K_0|^{1/{p}} |L_0'|^{1/{q'}} \sum_{K_0', L_0} |L_0|^{2/{q}} |K_0'|.
    \end{equation*}
    Using Hölder's inequality yields
    \begin{equation}\label{eq:y-est(((0 1 1 0) (0 1 0 1)) (((1) (0)) (nil nil) ((0) (1)))):6}
      \|T\|^2 \max_{K_0, L_0, L_0'} |K_0|^{1/{p}} |L_0'|^{1/{q'}} |L_0|^{2/{q}-1}.
    \end{equation}
    Inserting $|K_0|, |L_0|, |L_0'| \leq \alpha$ (see~\eqref{eq:alpha-small})
    into~\eqref{eq:y-est(((0 1 1 0) (0 1 0 1)) (((1) (0)) (nil nil) ((0) (1)))):6}, we obtain the
    estimate
    \begin{equation}\label{eq:y-est(((0 1 1 0) (0 1 0 1)) (((1) (0)) (nil nil) ((0) (1)))):7}
      \|T\|^2 \alpha^{1/{p}+1/{q}}.
    \end{equation}
  \end{proofcase}

  \subsubsection*{Summary of \textref[Case~]{case:y-est(((0 1 1 0) (0 1 0 1)) (((0) (1)) ((1) (0))
      (nil nil))):left} and \textref[Case~]{case:y-est(((0 1 1 0) (0 1 0 1)) (((1) (0)) (nil nil)
      ((0) (1)))):right}}
  Combining~\eqref{eq:y-est(((0 1 1 0) (0 1 0 1)) (((0) (1)) ((1) (0)) (nil nil))):7}
  with~\eqref{eq:y-est(((0 1 1 0) (0 1 0 1)) (((1) (0)) (nil nil) ((0) (1)))):7} yields
  \begin{equation}\label{eq:y-est(((0 1 1 0) (0 1 0 1)) (((0) (1)) ((1) (0)) (nil nil)))(((0 1 1 0)
      (0 1 0 1)) (((1) (0)) (nil nil) ((0) (1))))-final}
    \cond_{\theta,\varepsilon} Y^2
    \leq \|T\|^2 \alpha.
  \end{equation}

  \begin{proofcase}[Case~\theproofcase, group~{\textrefp[b]{enu:proof:lem:var:Y:b:4}}: $K_0 = K_0' \neq K_1 = K_1'$, $L_0 = L_1 \neq L_0' = L_1'$ (((2) (1)) ((0) (0)) (nil nil)) -- left variant]\label{case:y-est(((0 0 2 2) (0 1 0 1)) (((2) (1)) ((0) (0)) (nil nil))):left}
    In this case, we have to estimate
    \begin{equation}\label{eq:y-est(((0 0 2 2) (0 1 0 1)) (((2) (1)) ((0) (0)) (nil nil))):0}
      \sum_{K_1, L_0', K_0, L_0} \langle Th_{K_0\times L_0}, h_{K_0\times L_0'} \rangle \langle Th_{K_1\times L_0}, h_{K_1\times L_0'} \rangle.
    \end{equation}
    We put $a_{K_0, L_0, L_0'} = \langle T h_{K_0\times L_0}, h_{K_0\times L_0'} \rangle$ and note
    the estimate
    \begin{equation}\label{eq:y-est(((0 0 2 2) (0 1 0 1)) (((2) (1)) ((0) (0)) (nil nil))):1}
      |a_{K_0, L_0, L_0'}|
      \leq \|T\| |K_0|^{1/{p}} |L_0|^{1/{q}} |K_0|^{1/{p'}} |L_0'|^{1/{q'}}.
    \end{equation}
    Now, we write~\eqref{eq:y-est(((0 0 2 2) (0 1 0 1)) (((2) (1)) ((0) (0)) (nil nil))):0} as
    follows:
    \begin{equation}\label{eq:y-est(((0 0 2 2) (0 1 0 1)) (((2) (1)) ((0) (0)) (nil nil))):2}
      \sum_{K_1, L_0'} \Bigl\langle T \sum_{K_0, L_0} a_{K_0, L_0, L_0'} h_{K_1\times L_0},  h_{K_1\times L_0'} \Bigr\rangle.
    \end{equation}
    By duality, we obtain the subsequent upper estimate for~\eqref{eq:y-est(((0 0 2 2) (0 1 0 1))
      (((2) (1)) ((0) (0)) (nil nil))):2}:
    \begin{equation}\label{eq:y-est(((0 0 2 2) (0 1 0 1)) (((2) (1)) ((0) (0)) (nil nil))):3}
      \sum_{K_1, L_0'}\left\|T \sum_{K_0, L_0} a_{K_0, L_0, L_0'} h_{K_1\times L_0}\right\|_{H^p(H^q)} \left\| h_{K_1\times L_0'}\right\|_{(H^p(H^q))^*}.
    \end{equation}
    Estimate~\eqref{eq:y-est(((0 0 2 2) (0 1 0 1)) (((2) (1)) ((0) (0)) (nil nil))):1} and the
    disjointness of the dyadic intervals (see~\textrefp[J]{enu:j1}) yield
    \begin{equation}\label{eq:y-est(((0 0 2 2) (0 1 0 1)) (((2) (1)) ((0) (0)) (nil nil))):4}
      \sum_{K_1, L_0'}\|T\| \left\| \sum_{K_0, L_0}\max_{L_0}\left(\|T\| |K_0|^{1/{p}} |L_0|^{1/{q}} |K_0|^{1/{p'}} |L_0'|^{1/{q'}}\right) h_{K_1\times L_0}\right\|_{H^p(H^q)}\left\| h_{K_1\times L_0'}\right\|_{(H^p(H^q))^*}.
    \end{equation}
    Consequently, we obtain
    \begin{equation}\label{eq:y-est(((0 0 2 2) (0 1 0 1)) (((2) (1)) ((0) (0)) (nil nil))):5}
      \sum_{K_1, L_0'}\|T\| \sum_{K_0}\max_{L_0}\left(\|T\| |K_0|^{1/{p}} |L_0|^{1/{q}} |K_0|^{1/{p'}} |L_0'|^{1/{q'}}\right) \left\| \sum_{L_0}h_{K_1\times L_0}\right\|_{H^p(H^q)}\left\| h_{K_1\times L_0'}\right\|_{(H^p(H^q))^*}.
    \end{equation}
    Thus,~\eqref{eq:y-est(((0 0 2 2) (0 1 0 1)) (((2) (1)) ((0) (0)) (nil nil))):5} is bounded from
    above by
    \begin{equation*}
      \|T\|^2 \max_{L_0}  |L_0|^{1/{q}} \sum_{K_1, L_0'} \sum_{K_0}|K_0| |L_0'|^{2/{q'}} |K_1|.
    \end{equation*}
    Using Hölder's inequality yields
    \begin{equation}\label{eq:y-est(((0 0 2 2) (0 1 0 1)) (((2) (1)) ((0) (0)) (nil nil))):6}
      \|T\|^2 \max_{L_0, L_0'}  |L_0|^{1/{q}}|L_0'|^{2/{q'}-1}.
    \end{equation}
    Inserting $|L_0|, |L_0'| \leq \alpha$ (see~\eqref{eq:alpha-small}) into~\eqref{eq:y-est(((0 0 2
      2) (0 1 0 1)) (((2) (1)) ((0) (0)) (nil nil))):6}, we obtain the estimate
    \begin{equation}\label{eq:y-est(((0 0 2 2) (0 1 0 1)) (((2) (1)) ((0) (0)) (nil nil))):7}
      \|T\|^2 \alpha^{1/{q'}}.
    \end{equation}
  \end{proofcase}

  \begin{proofcase}[Case~\theproofcase, group~{\textrefp[b]{enu:proof:lem:var:Y:b:4}}: $K_0 = K_0' \neq K_1 = K_1'$, $L_0 = L_1 \neq L_0' = L_1'$ (((2) (0)) (nil nil) ((0) (1))) -- right variant]\label{case:y-est(((0 0 2 2) (0 1 0 1)) (((2) (0)) (nil nil) ((0) (1)))):right}
    In this case, we have to estimate
    \begin{equation}\label{eq:y-est(((0 0 2 2) (0 1 0 1)) (((2) (0)) (nil nil) ((0) (1)))):0}
      \sum_{K_1, L_0, K_0, L_0'} \langle Th_{K_0\times L_0}, h_{K_0\times L_0'} \rangle \langle Th_{K_1\times L_0}, h_{K_1\times L_0'} \rangle.
    \end{equation}
    We put $a_{K_0, L_0, L_0'} = \langle T h_{K_0\times L_0}, h_{K_0\times L_0'} \rangle$ and note
    the estimate
    \begin{equation}\label{eq:y-est(((0 0 2 2) (0 1 0 1)) (((2) (0)) (nil nil) ((0) (1)))):1}
      |a_{K_0, L_0, L_0'}|
      \leq \|T\| |K_0|^{1/{p}} |L_0|^{1/{q}} |K_0|^{1/{p'}} |L_0'|^{1/{q'}}.
    \end{equation}
    Now, we write~\eqref{eq:y-est(((0 0 2 2) (0 1 0 1)) (((2) (0)) (nil nil) ((0) (1)))):0} as
    follows:
    \begin{equation}\label{eq:y-est(((0 0 2 2) (0 1 0 1)) (((2) (0)) (nil nil) ((0) (1)))):2}
      \sum_{K_1, L_0} \Bigl\langle T  h_{K_1\times L_0}, \sum_{K_0, L_0'} a_{K_0, L_0, L_0'} h_{K_1\times L_0'} \Bigr\rangle.
    \end{equation}
    By duality, we obtain the subsequent upper estimate for~\eqref{eq:y-est(((0 0 2 2) (0 1 0 1))
      (((2) (0)) (nil nil) ((0) (1)))):2}:
    \begin{equation}\label{eq:y-est(((0 0 2 2) (0 1 0 1)) (((2) (0)) (nil nil) ((0) (1)))):3}
      \sum_{K_1, L_0}\left\|T  h_{K_1\times L_0}\right\|_{H^p(H^q)} \left\|\sum_{K_0, L_0'} a_{K_0, L_0, L_0'} h_{K_1\times L_0'}\right\|_{(H^p(H^q))^*}.
    \end{equation}
    Estimate~\eqref{eq:y-est(((0 0 2 2) (0 1 0 1)) (((2) (0)) (nil nil) ((0) (1)))):1} and the
    disjointness of the dyadic intervals (see~\textrefp[J]{enu:j1}) yield
    \begin{equation}\label{eq:y-est(((0 0 2 2) (0 1 0 1)) (((2) (0)) (nil nil) ((0) (1)))):4}
      \sum_{K_1, L_0}\|T\| \left\| h_{K_1\times L_0}\right\|_{H^p(H^q)}\left\| \sum_{K_0, L_0'}\max_{L_0'}\left(\|T\| |K_0|^{1/{p}} |L_0|^{1/{q}} |K_0|^{1/{p'}} |L_0'|^{1/{q'}}\right) h_{K_1\times L_0'}\right\|_{(H^p(H^q))^*}.
    \end{equation}
    Consequently, we obtain
    \begin{equation}\label{eq:y-est(((0 0 2 2) (0 1 0 1)) (((2) (0)) (nil nil) ((0) (1)))):5}
      \sum_{K_1, L_0}\|T\| \left\| h_{K_1\times L_0}\right\|_{H^p(H^q)}\sum_{K_0}\max_{L_0'}\left(\|T\| |K_0|^{1/{p}} |L_0|^{1/{q}} |K_0|^{1/{p'}} |L_0'|^{1/{q'}}\right) \left\| \sum_{L_0'}h_{K_1\times L_0'}\right\|_{(H^p(H^q))^*}.
    \end{equation}
    Thus,~\eqref{eq:y-est(((0 0 2 2) (0 1 0 1)) (((2) (0)) (nil nil) ((0) (1)))):5} is bounded from
    above by
    \begin{equation*}
      \|T\|^2 \max_{L_0'} |L_0'|^{1/{q'}} \sum_{K_1, L_0} \sum_{K_0} |K_0| |K_1| |L_0|^{2/{q}}.
    \end{equation*}
    Using Hölder's inequality yields
    \begin{equation}\label{eq:y-est(((0 0 2 2) (0 1 0 1)) (((2) (0)) (nil nil) ((0) (1)))):6}
      \|T\|^2 \max_{L_0, L_0'} |L_0'|^{1/{q'}} |L_0|^{2/{q}-1}.
    \end{equation}
    Inserting $|L_0|, |L_0'| \leq \alpha$ (see~\eqref{eq:alpha-small}) into~\eqref{eq:y-est(((0 0 2
      2) (0 1 0 1)) (((2) (0)) (nil nil) ((0) (1)))):6}, we obtain the estimate
    \begin{equation}\label{eq:y-est(((0 0 2 2) (0 1 0 1)) (((2) (0)) (nil nil) ((0) (1)))):7}
      \|T\|^2 \alpha^{1/{q}}.
    \end{equation}
  \end{proofcase}

  \subsubsection*{Summary of \textref[Case~]{case:y-est(((0 0 2 2) (0 1 0 1)) (((2) (1)) ((0) (0))
      (nil nil))):left} and \textref[Case~]{case:y-est(((0 0 2 2) (0 1 0 1)) (((2) (0)) (nil nil)
      ((0) (1)))):right}}
  Combining~\eqref{eq:y-est(((0 0 2 2) (0 1 0 1)) (((2) (1)) ((0) (0)) (nil nil))):7}
  with~\eqref{eq:y-est(((0 0 2 2) (0 1 0 1)) (((2) (0)) (nil nil) ((0) (1)))):7} yields
  \begin{equation}\label{eq:y-est(((0 0 2 2) (0 1 0 1)) (((2) (1)) ((0) (0)) (nil nil)))(((0 0 2 2)
      (0 1 0 1)) (((2) (0)) (nil nil) ((0) (1))))-final}
    \cond_{\theta,\varepsilon} Y^2
    \leq \|T\|^2 \alpha^{1/2}.
  \end{equation}

  \subsubsection*{Summary for $Y$}
  Combining~\eqref{eq:y-est(((0 0 0 0) (0 1 0 1)) (((0) (1)) (nil (0)) (nil nil))):7}
  with~\eqref{eq:y-est(((0 1 0 1) (0 1 0 1)) (((1) (1)) ((0) (0)) (nil nil)))(((0 1 0 1) (0 1 0 1))
    (((0) (0)) (nil nil) ((1) (1))))-final}, \eqref{eq:y-est(((0 1 1 0) (0 1 0 1)) (((0) (1)) ((1)
    (0)) (nil nil)))(((0 1 1 0) (0 1 0 1)) (((1) (0)) (nil nil) ((0) (1))))-final}
  and~\eqref{eq:y-est(((0 0 2 2) (0 1 0 1)) (((2) (1)) ((0) (0)) (nil nil)))(((0 0 2 2) (0 1 0 1))
    (((2) (0)) (nil nil) ((0) (1))))-final} yields
  \begin{equation}\label{eq:y-est:final}
    \cond_{\theta,\varepsilon} Y^2
    \leq 4 \|T\|^2 \alpha^{1/2}.
  \end{equation}
\end{myproof}

\subsection{Estimates for $Z$}

In this case, the following variables will always be summed over the following sets:
\begin{itemize}
\item
  $K_0$, $K_1$, $K_0'$, $K_1'$ over $\mathcal{X}_I$,
\item
  $L_0$, $L_1$, $L_0'$, $L_1'$ over $\mathcal{Y}_J$,
\end{itemize}
such that
\begin{equation}\label{eq:proof:lem:var:constr}
  \bigl(K_0\neq K_0'
  \quad\text{or}\quad
  L_0\neq L_0'\bigr)
  \qquad\text{and}\qquad
  \bigl(K_1\neq K_1'
  \quad\text{or}\quad
  L_1\neq L_1'\bigr).
\end{equation}

\begin{myproof}
  Note that by~\eqref{eq:dfn:rv:Z} and~\eqref{eq:b:dfn}, we obtain $Z^2(\theta,\varepsilon)$ is
  given by
  \begin{equation}\label{eq:proof:lem:var:Z}
    \sum_{\substack{K_0,K_1,K_0',K_1'\\L_0,L_1,L_0',L_1'}}
    \theta_{K_0}\theta_{K_1}\theta_{K_0'}\theta_{K_1'}
    \varepsilon_{L_0}\varepsilon_{L_1}\varepsilon_{L_0'}\varepsilon_{L_1'}
    \langle T h_{K_0\times L_0}, h_{K_0'\times L_0'}\rangle
    \langle T h_{K_1\times L_1}, h_{K_1'\times L_1'}\rangle
  \end{equation}
  Hence, in view of~\textrefp[R]{enu:proof:lem:var:r1}--\textrefp[R]{enu:proof:lem:var:r4}, we
  decompose the index set in~\eqref{eq:proof:lem:var:Y} into the following fifteen groups:
  \begin{enumerate}[({d}1)]
  \item\label{enu:proof:lem:var:Z:d:1} $K_0 = K_1 = K_0' = K_1'$ and $L_0 = L_1 \neq L_0' = L_1'$;
  \item\label{enu:proof:lem:var:Z:d:2} $K_0 = K_1 = K_0' = K_1'$ and $L_0 = L_1' \neq L_0' = L_1$;
  \item\label{enu:proof:lem:var:Z:d:3} $K_0 = K_1 = K_0' = K_1'$ and $L_0 = L_0' \neq L_1 = L_1'$
    (excluded by~\eqref{eq:proof:lem:var:constr});
  \end{enumerate}
  \begin{enumerate}[({e}1)]
  \item\label{enu:proof:lem:var:Z:e:1} $K_0 = K_1 \neq K_0' = K_1'$ and $L_0 = L_1 = L_0' = L_1'$;
  \item\label{enu:proof:lem:var:Z:e:2} $K_0 = K_1' \neq K_0' = K_1$ and $L_0 = L_1 = L_0' = L_1'$;
  \item\label{enu:proof:lem:var:Z:e:3} $K_0 = K_0' \neq K_1 = K_1'$ and $L_0 = L_1 = L_0' = L_1'$
    (excluded by~\eqref{eq:proof:lem:var:constr});
  \end{enumerate}
    
  \begin{enumerate}[({f}1)]
  \item\label{enu:proof:lem:var:Z:f:1} $K_0 = K_1 \neq K_0' = K_1'$ and
    $L_0 = L_1 \neq L_0' = L_1'$;
  \item\label{enu:proof:lem:var:Z:f:2} $K_0 = K_1 \neq K_0' = K_1'$ and
    $L_0 = L_1' \neq L_0' = L_1$;
  \item\label{enu:proof:lem:var:Z:f:3} $K_0 = K_1 \neq K_0' = K_1'$ and
    $L_0 = L_0' \neq L_1 = L_1'$;
  \item\label{enu:proof:lem:var:Z:f:4} $K_0 = K_1' \neq K_0' = K_1$ and
    $L_0 = L_1 \neq L_0' = L_1'$;
  \item\label{enu:proof:lem:var:Z:f:5} $K_0 = K_1' \neq K_0' = K_1$ and
    $L_0 = L_1' \neq L_0' = L_1$;
  \item\label{enu:proof:lem:var:Z:f:6} $K_0 = K_1' \neq K_0' = K_1$ and
    $L_0 = L_0' \neq L_1 = L_1'$;
  \item\label{enu:proof:lem:var:Z:f:7} $K_0 = K_0' \neq K_1 = K_1'$ and
    $L_0 = L_1 \neq L_0' = L_1'$;
  \item\label{enu:proof:lem:var:Z:f:8} $K_0 = K_0' \neq K_1 = K_1'$ and
    $L_0 = L_1' \neq L_0' = L_1$;
  \item\label{enu:proof:lem:var:Z:f:9} $K_0 = K_0' \neq K_1 = K_1'$ and $L_0 = L_0' \neq L_1 = L_1'$
    (excluded by~\eqref{eq:proof:lem:var:constr}).
  \end{enumerate}
  As we indicated above, the
  cases~\textrefp[d]{enu:proof:lem:var:Z:d:1},~\textrefp[e]{enu:proof:lem:var:Z:e:1}
  and~\textrefp[f]{enu:proof:lem:var:Z:f:1} are contradicting the
  constraint~\eqref{eq:proof:lem:var:constr}, and are thereby excluded.

  \begin{proofcase}[Case~\theproofcase, group~{\textrefp[d]{enu:proof:lem:var:Z:d:1}}: $K_0 = K_0' = K_1 = K_1'$, $L_0 = L_1 \neq L_0' = L_1'$
    (((0) (1)) (nil (0)) (nil nil)) -- left variant]\label{case:z-est(((0 0 0 0) (0 1 0 1)) (((0)
      (1)) (nil (0)) (nil nil))):left}
    In this case, we have to estimate
    \begin{equation}\label{eq:z-est(((0 0 0 0) (0 1 0 1)) (((0) (1)) (nil (0)) (nil nil))):0}
      \sum_{K_0, L_0', L_0} \langle Th_{K_0\times L_0}, h_{K_0\times L_0'} \rangle \langle Th_{K_0\times L_0}, h_{K_0\times L_0'} \rangle.
    \end{equation}
    We put $a_{K_0, L_0, L_0'} = \langle T h_{K_0\times L_0}, h_{K_0\times L_0'} \rangle$ and note
    the estimate
    \begin{equation}\label{eq:z-est(((0 0 0 0) (0 1 0 1)) (((0) (1)) (nil (0)) (nil nil))):1}
      |a_{K_0, L_0, L_0'}|
      \leq \|T\| |K_0|^{1/{p}} |L_0|^{1/{q}} |K_0|^{1/{p'}} |L_0'|^{1/{q'}}.
    \end{equation}
    Now, we write~\eqref{eq:z-est(((0 0 0 0) (0 1 0 1)) (((0) (1)) (nil (0)) (nil nil))):0} as
    follows:
    \begin{equation}\label{eq:z-est(((0 0 0 0) (0 1 0 1)) (((0) (1)) (nil (0)) (nil nil))):2}
      \sum_{K_0, L_0'} \Bigl\langle T \sum_{L_0} a_{K_0, L_0, L_0'} h_{K_0\times L_0},  h_{K_0\times L_0'} \Bigr\rangle.
    \end{equation}
    By duality, we obtain the subsequent upper estimate for~\eqref{eq:z-est(((0 0 0 0) (0 1 0 1))
      (((0) (1)) (nil (0)) (nil nil))):2}:
    \begin{equation}\label{eq:z-est(((0 0 0 0) (0 1 0 1)) (((0) (1)) (nil (0)) (nil nil))):3}
      \sum_{K_0, L_0'}\left\|T \sum_{L_0} a_{K_0, L_0, L_0'} h_{K_0\times L_0}\right\|_{H^p(H^q)} \left\| h_{K_0\times L_0'}\right\|_{(H^p(H^q))^*}.
    \end{equation}
    Estimate~\eqref{eq:z-est(((0 0 0 0) (0 1 0 1)) (((0) (1)) (nil (0)) (nil nil))):1} and the
    disjointness of the dyadic intervals (see~\textrefp[J]{enu:j1}) yield
    \begin{equation}\label{eq:z-est(((0 0 0 0) (0 1 0 1)) (((0) (1)) (nil (0)) (nil nil))):4}
      \sum_{K_0, L_0'}\|T\| \left\| \sum_{L_0}\max_{L_0}\left(\|T\| |K_0|^{1/{p}} |L_0|^{1/{q}} |K_0|^{1/{p'}} |L_0'|^{1/{q'}}\right) h_{K_0\times L_0}\right\|_{H^p(H^q)}\left\| h_{K_0\times L_0'}\right\|_{(H^p(H^q))^*}.
    \end{equation}
    Consequently, we obtain
    \begin{equation}\label{eq:z-est(((0 0 0 0) (0 1 0 1)) (((0) (1)) (nil (0)) (nil nil))):5}
      \sum_{K_0, L_0'}\|T\| \max_{L_0}\left(\|T\| |K_0|^{1/{p}} |L_0|^{1/{q}} |K_0|^{1/{p'}} |L_0'|^{1/{q'}}\right) \left\| \sum_{L_0}h_{K_0\times L_0}\right\|_{H^p(H^q)}\left\| h_{K_0\times L_0'}\right\|_{(H^p(H^q))^*}.
    \end{equation}
    Thus,~\eqref{eq:z-est(((0 0 0 0) (0 1 0 1)) (((0) (1)) (nil (0)) (nil nil))):5} is bounded from
    above by
    \begin{equation*}
      \|T\|^2 \max_{L_0} |L_0|^{1/{q}} \sum_{K_0, L_0'} |K_0|^2 |L_0'|^{2/{q'}}.
    \end{equation*}
    Using Hölder's inequality yields
    \begin{equation}\label{eq:z-est(((0 0 0 0) (0 1 0 1)) (((0) (1)) (nil (0)) (nil nil))):6}
      \|T\|^2 \max_{K_0, L_0, L_0'} |K_0| |L_0|^{1/{q}} |L_0'|^{2/{q'}-1}.
    \end{equation}
    Inserting $|K_0|, |L_0|, |L_0'| \leq \alpha$ (see~\eqref{eq:alpha-small})
    into~\eqref{eq:z-est(((0 0 0 0) (0 1 0 1)) (((0) (1)) (nil (0)) (nil nil))):6}, we obtain the
    estimate
    \begin{equation}\label{eq:z-est(((0 0 0 0) (0 1 0 1)) (((0) (1)) (nil (0)) (nil nil))):7}
      \|T\|^2 \alpha^{1+1/{q'}}.
    \end{equation}
  \end{proofcase}
  
  \begin{proofcase}[Case~\theproofcase, group~{\textrefp[d]{enu:proof:lem:var:Z:d:2}}: $K_0 = K_0' = K_1 = K_1'$, $L_0 = L_1' \neq L_0' = L_1$
    (((0) (0)) (nil (1)) (nil nil)) -- left variant]\label{case:z-est(((0 0 0 0) (0 1 1 0)) (((0) (0)) (nil (1)) (nil nil))):left}
    In this case, we have to estimate
    \begin{equation}\label{eq:z-est(((0 0 0 0) (0 1 1 0)) (((0) (0)) (nil (1)) (nil nil))):0}
      \sum_{K_0, L_0, L_0'} \langle Th_{K_0\times L_0}, h_{K_0\times L_0'} \rangle \langle Th_{K_0\times L_0'}, h_{K_0\times L_0} \rangle.
    \end{equation}
    We put $a_{K_0, L_0, L_0'} = \langle T h_{K_0\times L_0}, h_{K_0\times L_0'} \rangle$ and note
    the estimate
    \begin{equation}\label{eq:z-est(((0 0 0 0) (0 1 1 0)) (((0) (0)) (nil (1)) (nil nil))):1}
      |a_{K_0, L_0, L_0'}|
      \leq \|T\| |K_0|^{1/{p}} |L_0|^{1/{q}} |K_0|^{1/{p'}} |L_0'|^{1/{q'}}.
    \end{equation}
    Now, we write~\eqref{eq:z-est(((0 0 0 0) (0 1 1 0)) (((0) (0)) (nil (1)) (nil nil))):0} as
    follows:
    \begin{equation}\label{eq:z-est(((0 0 0 0) (0 1 1 0)) (((0) (0)) (nil (1)) (nil nil))):2}
      \sum_{K_0, L_0} \Bigl\langle T \sum_{L_0'} a_{K_0, L_0, L_0'} h_{K_0\times L_0'},  h_{K_0\times L_0} \Bigr\rangle.
    \end{equation}
    By duality, we obtain the subsequent upper estimate for~\eqref{eq:z-est(((0 0 0 0) (0 1 1 0))
      (((0) (0)) (nil (1)) (nil nil))):2}:
    \begin{equation}\label{eq:z-est(((0 0 0 0) (0 1 1 0)) (((0) (0)) (nil (1)) (nil nil))):3}
      \sum_{K_0, L_0}\left\|T \sum_{L_0'} a_{K_0, L_0, L_0'} h_{K_0\times L_0'}\right\|_{H^p(H^q)} \left\| h_{K_0\times L_0}\right\|_{(H^p(H^q))^*}.
    \end{equation}
    Estimate~\eqref{eq:z-est(((0 0 0 0) (0 1 1 0)) (((0) (0)) (nil (1)) (nil nil))):1} and the
    disjointness of the dyadic intervals (see~\textrefp[J]{enu:j1}) yield
    \begin{equation}\label{eq:z-est(((0 0 0 0) (0 1 1 0)) (((0) (0)) (nil (1)) (nil nil))):4}
      \sum_{K_0, L_0}\|T\| \left\| \sum_{L_0'}\max_{L_0'}\left(\|T\| |K_0|^{1/{p}} |L_0|^{1/{q}} |K_0|^{1/{p'}} |L_0'|^{1/{q'}}\right) h_{K_0\times L_0'}\right\|_{H^p(H^q)}\left\| h_{K_0\times L_0}\right\|_{(H^p(H^q))^*}.
    \end{equation}
    Consequently, we obtain
    \begin{equation}\label{eq:z-est(((0 0 0 0) (0 1 1 0)) (((0) (0)) (nil (1)) (nil nil))):5}
      \sum_{K_0, L_0}\|T\| \max_{L_0'}\left(\|T\| |K_0|^{1/{p}} |L_0|^{1/{q}} |K_0|^{1/{p'}} |L_0'|^{1/{q'}}\right) \left\| \sum_{L_0'}h_{K_0\times L_0'}\right\|_{H^p(H^q)}\left\| h_{K_0\times L_0}\right\|_{(H^p(H^q))^*}.
    \end{equation}
    Thus,~\eqref{eq:z-est(((0 0 0 0) (0 1 1 0)) (((0) (0)) (nil (1)) (nil nil))):5} is bounded from
    above by
    \begin{equation*}
      \|T\|^2 \max_{L_0'} |L_0'|^{1/{q'}} \sum_{K_0, L_0} |K_0|^2 |L_0|.
    \end{equation*}
    Using Hölder's inequality yields
    \begin{equation}\label{eq:z-est(((0 0 0 0) (0 1 1 0)) (((0) (0)) (nil (1)) (nil nil))):6}
      \|T\|^2 \max_{K_0, L_0, L_0'} |L_0'|^{1/{q'}}|K_0|.
    \end{equation}
    Inserting $|K_0|, |L_0|, |L_0'| \leq \alpha$ (see~\eqref{eq:alpha-small})
    into~\eqref{eq:z-est(((0 0 0 0) (0 1 1 0)) (((0) (0)) (nil (1)) (nil nil))):6}, we obtain the
    estimate
    \begin{equation}\label{eq:z-est(((0 0 0 0) (0 1 1 0)) (((0) (0)) (nil (1)) (nil nil))):7}
      \|T\|^2 \alpha^{1+1/{q'}}.
    \end{equation}
  \end{proofcase}
  
  \begin{proofcase}[Case~\theproofcase, group~{\textrefp[e]{enu:proof:lem:var:Z:e:1}}: $K_0 = K_1 \neq K_0' = K_1'$, $L_0 = L_0' = L_1 = L_1'$
    (((1) (0)) ((0) nil) (nil nil)) -- left variant]\label{case:z-est(((0 1 0 1) (0 0 0 0)) (((1) (0)) ((0) nil) (nil nil))):left}
    In this case, we have to estimate
    \begin{equation}\label{eq:z-est(((0 1 0 1) (0 0 0 0)) (((1) (0)) ((0) nil) (nil nil))):0}
      \sum_{K_0', L_0, K_0} \langle Th_{K_0\times L_0}, h_{K_0'\times L_0} \rangle \langle Th_{K_0\times L_0}, h_{K_0'\times L_0} \rangle.
    \end{equation}
    We put $a_{K_0, K_0', L_0} = \langle T h_{K_0\times L_0}, h_{K_0'\times L_0} \rangle$ and note
    the estimate
    \begin{equation}\label{eq:z-est(((0 1 0 1) (0 0 0 0)) (((1) (0)) ((0) nil) (nil nil))):1}
      |a_{K_0, K_0', L_0}|
      \leq \|T\| |K_0|^{1/{p}} |L_0|^{1/{q}} |K_0'|^{1/{p'}} |L_0|^{1/{q'}}.
    \end{equation}
    Now, we write~\eqref{eq:z-est(((0 1 0 1) (0 0 0 0)) (((1) (0)) ((0) nil) (nil nil))):0} as
    follows:
    \begin{equation}\label{eq:z-est(((0 1 0 1) (0 0 0 0)) (((1) (0)) ((0) nil) (nil nil))):2}
      \sum_{K_0', L_0} \Bigl\langle T \sum_{K_0} a_{K_0, K_0', L_0} h_{K_0\times L_0},  h_{K_0'\times L_0} \Bigr\rangle.
    \end{equation}
    By duality, we obtain the subsequent upper estimate for~\eqref{eq:z-est(((0 1 0 1) (0 0 0 0))
      (((1) (0)) ((0) nil) (nil nil))):2}:
    \begin{equation}\label{eq:z-est(((0 1 0 1) (0 0 0 0)) (((1) (0)) ((0) nil) (nil nil))):3}
      \sum_{K_0', L_0}\left\|T \sum_{K_0} a_{K_0, K_0', L_0} h_{K_0\times L_0}\right\|_{H^p(H^q)} \left\| h_{K_0'\times L_0}\right\|_{(H^p(H^q))^*}.
    \end{equation}
    Estimate~\eqref{eq:z-est(((0 1 0 1) (0 0 0 0)) (((1) (0)) ((0) nil) (nil nil))):1} and the
    disjointness of the dyadic intervals (see~\textrefp[J]{enu:j1}) yield
    \begin{equation}\label{eq:z-est(((0 1 0 1) (0 0 0 0)) (((1) (0)) ((0) nil) (nil nil))):4}
      \sum_{K_0', L_0}\|T\| \left\| \sum_{K_0}\max_{K_0}\left(\|T\| |K_0|^{1/{p}} |L_0|^{1/{q}} |K_0'|^{1/{p'}} |L_0|^{1/{q'}}\right) h_{K_0\times L_0}\right\|_{H^p(H^q)}\left\| h_{K_0'\times L_0}\right\|_{(H^p(H^q))^*}.
    \end{equation}
    Consequently, we obtain
    \begin{equation}\label{eq:z-est(((0 1 0 1) (0 0 0 0)) (((1) (0)) ((0) nil) (nil nil))):5}
      \sum_{K_0', L_0}\|T\| \max_{K_0}\left(\|T\| |K_0|^{1/{p}} |L_0|^{1/{q}} |K_0'|^{1/{p'}} |L_0|^{1/{q'}}\right) \left\| \sum_{K_0}h_{K_0\times L_0}\right\|_{H^p(H^q)}\left\| h_{K_0'\times L_0}\right\|_{(H^p(H^q))^*}.
    \end{equation}
    Thus,~\eqref{eq:z-est(((0 1 0 1) (0 0 0 0)) (((1) (0)) ((0) nil) (nil nil))):5} is bounded from
    above by
    \begin{equation*}
      \|T\|^2 \max_{K_0} |K_0|^{1/{p}} \sum_{K_0', L_0} |L_0|^2 |K_0'|^{2/{p'}}.
    \end{equation*}
    Using Hölder's inequality yields
    \begin{equation}\label{eq:z-est(((0 1 0 1) (0 0 0 0)) (((1) (0)) ((0) nil) (nil nil))):6}
      \|T\|^2 \max_{K_0, K_0', L_0} |K_0|^{1/{p}} |L_0| |K_0'|^{2/{p'}-1}.
    \end{equation}
    Inserting $|K_0|, |K_0'|, |L_0| \leq \alpha$ (see~\eqref{eq:alpha-small})
    into~\eqref{eq:z-est(((0 1 0 1) (0 0 0 0)) (((1) (0)) ((0) nil) (nil nil))):6}, we obtain the
    estimate
    \begin{equation}\label{eq:z-est(((0 1 0 1) (0 0 0 0)) (((1) (0)) ((0) nil) (nil nil))):7}
      \|T\|^2 \alpha^{1+1/{p'}}.
    \end{equation}
  \end{proofcase}

  \begin{proofcase}[Case~\theproofcase, group~{\textrefp[e]{enu:proof:lem:var:Z:e:2}}: $K_0 = K_1' \neq K_0' = K_1$, $L_0 = L_0' = L_1 = L_1'$ (((0) (0)) ((1) nil) (nil nil)) -- left variant]\label{case:z-est(((0 1 1 0) (0 0 0 0)) (((0) (0)) ((1) nil) (nil nil))):left}
    In this case, we have to estimate
    \begin{equation}\label{eq:z-est(((0 1 1 0) (0 0 0 0)) (((0) (0)) ((1) nil) (nil nil))):0}
      \sum_{K_0, L_0, K_0'} \langle Th_{K_0\times L_0}, h_{K_0'\times L_0} \rangle \langle Th_{K_0'\times L_0}, h_{K_0\times L_0} \rangle.
    \end{equation}
    We put $a_{K_0, K_0', L_0} = \langle T h_{K_0\times L_0}, h_{K_0'\times L_0} \rangle$ and note
    the estimate
    \begin{equation}\label{eq:z-est(((0 1 1 0) (0 0 0 0)) (((0) (0)) ((1) nil) (nil nil))):1}
      |a_{K_0, K_0', L_0}|
      \leq \|T\| |K_0|^{1/{p}} |L_0|^{1/{q}} |K_0'|^{1/{p'}} |L_0|^{1/{q'}}.
    \end{equation}
    Now, we write~\eqref{eq:z-est(((0 1 1 0) (0 0 0 0)) (((0) (0)) ((1) nil) (nil nil))):0} as
    follows:
    \begin{equation}\label{eq:z-est(((0 1 1 0) (0 0 0 0)) (((0) (0)) ((1) nil) (nil nil))):2}
      \sum_{K_0, L_0} \Bigl\langle T \sum_{K_0'} a_{K_0, K_0', L_0} h_{K_0'\times L_0},  h_{K_0\times L_0} \Bigr\rangle.
    \end{equation}
    By duality, we obtain the subsequent upper estimate for~\eqref{eq:z-est(((0 1 1 0) (0 0 0 0))
      (((0) (0)) ((1) nil) (nil nil))):2}:
    \begin{equation}\label{eq:z-est(((0 1 1 0) (0 0 0 0)) (((0) (0)) ((1) nil) (nil nil))):3}
      \sum_{K_0, L_0}\left\|T \sum_{K_0'} a_{K_0, K_0', L_0} h_{K_0'\times L_0}\right\|_{H^p(H^q)} \left\| h_{K_0\times L_0}\right\|_{(H^p(H^q))^*}.
    \end{equation}
    Estimate~\eqref{eq:z-est(((0 1 1 0) (0 0 0 0)) (((0) (0)) ((1) nil) (nil nil))):1} and the
    disjointness of the dyadic intervals (see~\textrefp[J]{enu:j1}) yield
    \begin{equation}\label{eq:z-est(((0 1 1 0) (0 0 0 0)) (((0) (0)) ((1) nil) (nil nil))):4}
      \sum_{K_0, L_0}\|T\| \left\| \sum_{K_0'}\max_{K_0'}\left(\|T\| |K_0|^{1/{p}} |L_0|^{1/{q}} |K_0'|^{1/{p'}} |L_0|^{1/{q'}}\right) h_{K_0'\times L_0}\right\|_{H^p(H^q)}\left\| h_{K_0\times L_0}\right\|_{(H^p(H^q))^*}.
    \end{equation}
    Consequently, we obtain
    \begin{equation}\label{eq:z-est(((0 1 1 0) (0 0 0 0)) (((0) (0)) ((1) nil) (nil nil))):5}
      \sum_{K_0, L_0}\|T\| \max_{K_0'}\left(\|T\| |K_0|^{1/{p}} |L_0|^{1/{q}} |K_0'|^{1/{p'}} |L_0|^{1/{q'}}\right) \left\| \sum_{K_0'}h_{K_0'\times L_0}\right\|_{H^p(H^q)}\left\| h_{K_0\times L_0}\right\|_{(H^p(H^q))^*}.
    \end{equation}
    Thus,~\eqref{eq:z-est(((0 1 1 0) (0 0 0 0)) (((0) (0)) ((1) nil) (nil nil))):5} is bounded from
    above by
    \begin{equation*}
      \|T\|^2 \max_{K_0'}  |K_0'|^{1/{p'}} \sum_{K_0, L_0}  |K_0| |L_0|^2.
    \end{equation*}
    Using Hölder's inequality yields
    \begin{equation}\label{eq:z-est(((0 1 1 0) (0 0 0 0)) (((0) (0)) ((1) nil) (nil nil))):6}
      \|T\|^2 \max_{K_0', L_0}  |K_0'|^{1/{p'}} |L_0|.
    \end{equation}
    Inserting $|K_0'|, |L_0'|\leq \alpha$ (see~\eqref{eq:alpha-small}) into~\eqref{eq:z-est(((0 1 1
      0) (0 0 0 0)) (((0) (0)) ((1) nil) (nil nil))):6}, we obtain the estimate
    \begin{equation}\label{eq:z-est(((0 1 1 0) (0 0 0 0)) (((0) (0)) ((1) nil) (nil nil))):7}
      \|T\|^2 \alpha^{1+1/{p'}}.
    \end{equation}
  \end{proofcase}

  \begin{proofcase}[Case~\theproofcase, group~{\textrefp[f]{enu:proof:lem:var:Z:f:1}}: $K_0 = K_1 \neq K_0' = K_1'$, $L_0 = L_1 \neq L_0' = L_1'$
    (((1) (1)) ((0) (0)) (nil nil)) -- left variant]\label{case:z-est(((0 1 0 1) (0 1 0 1)) (((1) (1)) ((0) (0)) (nil nil))):left}
    In this case, we have to estimate
    \begin{equation}\label{eq:z-est(((0 1 0 1) (0 1 0 1)) (((1) (1)) ((0) (0)) (nil nil))):0}
      \sum_{K_0', L_0', K_0, L_0} \langle Th_{K_0\times L_0}, h_{K_0'\times L_0'} \rangle \langle Th_{K_0\times L_0}, h_{K_0'\times L_0'} \rangle.
    \end{equation}
    We put $a_{K_0, K_0', L_0, L_0'} = \langle T h_{K_0\times L_0}, h_{K_0'\times L_0'} \rangle$ and
    note the estimate
    \begin{equation}\label{eq:z-est(((0 1 0 1) (0 1 0 1)) (((1) (1)) ((0) (0)) (nil nil))):1}
      |a_{K_0, K_0', L_0, L_0'}|
      \leq \|T\| |K_0|^{1/{p}} |L_0|^{1/{q}} |K_0'|^{1/{p'}} |L_0'|^{1/{q'}}.
    \end{equation}
    Now, we write~\eqref{eq:z-est(((0 1 0 1) (0 1 0 1)) (((1) (1)) ((0) (0)) (nil nil))):0} as
    follows:
    \begin{equation}\label{eq:z-est(((0 1 0 1) (0 1 0 1)) (((1) (1)) ((0) (0)) (nil nil))):2}
      \sum_{K_0', L_0'} \Bigl\langle T \sum_{K_0, L_0} a_{K_0, K_0', L_0, L_0'} h_{K_0\times L_0},  h_{K_0'\times L_0'} \Bigr\rangle.
    \end{equation}
    By duality, we obtain the subsequent upper estimate for~\eqref{eq:z-est(((0 1 0 1) (0 1 0 1))
      (((1) (1)) ((0) (0)) (nil nil))):2}:
    \begin{equation}\label{eq:z-est(((0 1 0 1) (0 1 0 1)) (((1) (1)) ((0) (0)) (nil nil))):3}
      \sum_{K_0', L_0'}\left\|T \sum_{K_0, L_0} a_{K_0, K_0', L_0, L_0'} h_{K_0\times L_0}\right\|_{H^p(H^q)} \left\| h_{K_0'\times L_0'}\right\|_{(H^p(H^q))^*}.
    \end{equation}
    Estimate~\eqref{eq:z-est(((0 1 0 1) (0 1 0 1)) (((1) (1)) ((0) (0)) (nil nil))):1} and the
    disjointness of the dyadic intervals (see~\textrefp[J]{enu:j1}) yield
    \begin{equation}\label{eq:z-est(((0 1 0 1) (0 1 0 1)) (((1) (1)) ((0) (0)) (nil nil))):4}
      \sum_{K_0', L_0'}\|T\| \left\| \sum_{K_0, L_0}\max_{K_0, L_0}\left(\|T\| |K_0|^{1/{p}} |L_0|^{1/{q}} |K_0'|^{1/{p'}} |L_0'|^{1/{q'}}\right) h_{K_0\times L_0}\right\|_{H^p(H^q)}\left\| h_{K_0'\times L_0'}\right\|_{(H^p(H^q))^*}.
    \end{equation}
    Consequently, we obtain
    \begin{equation}\label{eq:z-est(((0 1 0 1) (0 1 0 1)) (((1) (1)) ((0) (0)) (nil nil))):5}
      \sum_{K_0', L_0'}\|T\| \max_{K_0, L_0}\left(\|T\| |K_0|^{1/{p}} |L_0|^{1/{q}} |K_0'|^{1/{p'}} |L_0'|^{1/{q'}}\right) \left\| \sum_{K_0, L_0}h_{K_0\times L_0}\right\|_{H^p(H^q)}\left\| h_{K_0'\times L_0'}\right\|_{(H^p(H^q))^*}.
    \end{equation}
    Thus,~\eqref{eq:z-est(((0 1 0 1) (0 1 0 1)) (((1) (1)) ((0) (0)) (nil nil))):5} is bounded from
    above by
    \begin{equation*}
      \|T\|^2 \max_{K_0, L_0} |K_0|^{1/{p}} |L_0|^{1/{q}} \sum_{K_0', L_0'} |K_0'|^{2/{p'}} |L_0'|^{2/{q'}}.
    \end{equation*}
    Using Hölder's inequality yields
    \begin{equation}\label{eq:z-est(((0 1 0 1) (0 1 0 1)) (((1) (1)) ((0) (0)) (nil nil))):6}
      \|T\|^2 \max_{K_0, L_0, K_0', L_0'} |K_0|^{1/{p}} |L_0|^{1/{q}} |K_0'|^{2/{p'}-1} |L_0'|^{2/{q'}-1}.
    \end{equation}
    Inserting $|K_0|, |K_0'|, |L_0|, |L_0'| \leq \alpha$ (see~\eqref{eq:alpha-small})
    into~\eqref{eq:z-est(((0 1 0 1) (0 1 0 1)) (((1) (1)) ((0) (0)) (nil nil))):6}, we obtain the
    estimate
    \begin{equation}\label{eq:z-est(((0 1 0 1) (0 1 0 1)) (((1) (1)) ((0) (0)) (nil nil))):7}
      \|T\|^2 \alpha^{1/{p'}+1/{q'}}.
    \end{equation}
  \end{proofcase}

  \begin{proofcase}[Case~\theproofcase, group~{\textrefp[f]{enu:proof:lem:var:Z:f:1}}: $K_0 = K_1 \neq K_0' = K_1'$, $L_0 = L_1 \neq L_0' = L_1'$
    (((0) (0)) (nil nil) ((1) (1))) -- right variant]\label{case:z-est(((0 1 0 1) (0 1 0 1)) (((0)
      (0)) (nil nil) ((1) (1)))):right}
    In this case, we have to estimate
    \begin{equation}\label{eq:z-est(((0 1 0 1) (0 1 0 1)) (((0) (0)) (nil nil) ((1) (1)))):0}
      \sum_{K_0, L_0, K_0', L_0'} \langle Th_{K_0\times L_0}, h_{K_0'\times L_0'} \rangle \langle Th_{K_0\times L_0}, h_{K_0'\times L_0'} \rangle.
    \end{equation}
    We put $a_{K_0, K_0', L_0, L_0'} = \langle T h_{K_0\times L_0}, h_{K_0'\times L_0'} \rangle$ and
    note the estimate
    \begin{equation}\label{eq:z-est(((0 1 0 1) (0 1 0 1)) (((0) (0)) (nil nil) ((1) (1)))):1}
      |a_{K_0, K_0', L_0, L_0'}|
      \leq \|T\| |K_0|^{1/{p}} |L_0|^{1/{q}} |K_0'|^{1/{p'}} |L_0'|^{1/{q'}}.
    \end{equation}
    Now, we write~\eqref{eq:z-est(((0 1 0 1) (0 1 0 1)) (((0) (0)) (nil nil) ((1) (1)))):0} as
    follows:
    \begin{equation}\label{eq:z-est(((0 1 0 1) (0 1 0 1)) (((0) (0)) (nil nil) ((1) (1)))):2}
      \sum_{K_0, L_0} \Bigl\langle T  h_{K_0\times L_0}, \sum_{K_0', L_0'} a_{K_0, K_0', L_0, L_0'} h_{K_0'\times L_0'} \Bigr\rangle.
    \end{equation}
    By duality, we obtain the subsequent upper estimate for~\eqref{eq:z-est(((0 1 0 1) (0 1 0 1))
      (((0) (0)) (nil nil) ((1) (1)))):2}:
    \begin{equation}\label{eq:z-est(((0 1 0 1) (0 1 0 1)) (((0) (0)) (nil nil) ((1) (1)))):3}
      \sum_{K_0, L_0}\left\|T  h_{K_0\times L_0}\right\|_{H^p(H^q)} \left\|\sum_{K_0', L_0'} a_{K_0, K_0', L_0, L_0'} h_{K_0'\times L_0'}\right\|_{(H^p(H^q))^*}.
    \end{equation}
    Estimate~\eqref{eq:z-est(((0 1 0 1) (0 1 0 1)) (((0) (0)) (nil nil) ((1) (1)))):1} and the
    disjointness of the dyadic intervals (see~\textrefp[J]{enu:j1}) yield
    \begin{equation}\label{eq:z-est(((0 1 0 1) (0 1 0 1)) (((0) (0)) (nil nil) ((1) (1)))):4}
      \sum_{K_0, L_0}\|T\| \left\| h_{K_0\times L_0}\right\|_{H^p(H^q)}\left\| \sum_{K_0', L_0'}\max_{K_0', L_0'}\left(\|T\| |K_0|^{1/{p}} |L_0|^{1/{q}} |K_0'|^{1/{p'}} |L_0'|^{1/{q'}}\right) h_{K_0'\times L_0'}\right\|_{(H^p(H^q))^*}.
    \end{equation}
    Consequently, we obtain
    \begin{equation}\label{eq:z-est(((0 1 0 1) (0 1 0 1)) (((0) (0)) (nil nil) ((1) (1)))):5}
      \sum_{K_0, L_0}\|T\| \left\| h_{K_0\times L_0}\right\|_{H^p(H^q)}\max_{K_0', L_0'}\left(\|T\| |K_0|^{1/{p}} |L_0|^{1/{q}} |K_0'|^{1/{p'}} |L_0'|^{1/{q'}}\right) \left\| \sum_{K_0', L_0'}h_{K_0'\times L_0'}\right\|_{(H^p(H^q))^*}.
    \end{equation}
    Thus,~\eqref{eq:z-est(((0 1 0 1) (0 1 0 1)) (((0) (0)) (nil nil) ((1) (1)))):5} is bounded from
    above by
    \begin{equation*}
      \|T\|^2 \max_{K_0', L_0'}  |K_0'|^{1/{p'}} |L_0'|^{1/{q'}} \sum_{K_0, L_0} |K_0|^{2/{p}} |L_0|^{2/{q}}.
    \end{equation*}
    Using Hölder's inequality yields
    \begin{equation}\label{eq:z-est(((0 1 0 1) (0 1 0 1)) (((0) (0)) (nil nil) ((1) (1)))):6}
      \|T\|^2 \max_{K_0, L_0, K_0', L_0'}  |K_0'|^{1/{p'}} |L_0'|^{1/{q'}}  |K_0|^{2/{p}-1} |L_0|^{2/{q}-1}.
    \end{equation}
    Inserting $|K_0|, |K_0'|, |L_0|, |L_0'| \leq \alpha$ (see~\eqref{eq:alpha-small})
    into~\eqref{eq:z-est(((0 1 0 1) (0 1 0 1)) (((0) (0)) (nil nil) ((1) (1)))):6}, we obtain the
    estimate
    \begin{equation}\label{eq:z-est(((0 1 0 1) (0 1 0 1)) (((0) (0)) (nil nil) ((1) (1)))):7}
      \|T\|^2 \alpha^{1/{p}+1/{q}}.
    \end{equation}
  \end{proofcase}

  \subsubsection*{Summary of \textref[Case~]{case:z-est(((0 1 0 1) (0 1 0 1)) (((1) (1)) ((0) (0))
      (nil nil))):left} and \textref[Case~]{case:z-est(((0 1 0 1) (0 1 0 1)) (((0) (0)) (nil nil)
      ((1) (1)))):right}}
  Combining~\eqref{eq:z-est(((0 1 0 1) (0 1 0 1)) (((1) (1)) ((0) (0)) (nil nil))):7}
  with~\eqref{eq:z-est(((0 1 0 1) (0 1 0 1)) (((0) (0)) (nil nil) ((1) (1)))):7} yields
  \begin{equation}\label{eq:z-est(((0 1 0 1) (0 1 0 1)) (((1) (1)) ((0) (0)) (nil nil)))(((0 1 0 1)
      (0 1 0 1)) (((0) (0)) (nil nil) ((1) (1))))-final}
    \cond_{\theta,\varepsilon} Z^2
    \leq \|T\|^2 \alpha.
  \end{equation}

  \begin{proofcase}[Case~\theproofcase, group~{\textrefp[f]{enu:proof:lem:var:Z:f:2}}: $K_0 = K_1 \neq K_0' = K_1'$, $L_0 = L_1' \neq L_0' = L_1$ (((1) (0)) ((0) (1)) (nil nil)) -- left variant]\label{case:z-est(((0 1 0 1) (0 1 1 0)) (((1) (0)) ((0) (1)) (nil nil))):left}
    In this case, we have to estimate
    \begin{equation}\label{eq:z-est(((0 1 0 1) (0 1 1 0)) (((1) (0)) ((0) (1)) (nil nil))):0}
      \sum_{K_0', L_0, K_0, L_0'} \langle Th_{K_0\times L_0}, h_{K_0'\times L_0'} \rangle \langle Th_{K_0\times L_0'}, h_{K_0'\times L_0} \rangle.
    \end{equation}
    We put $a_{K_0, K_0', L_0, L_0'} = \langle T h_{K_0\times L_0}, h_{K_0'\times L_0'} \rangle$ and
    note the estimate
    \begin{equation}\label{eq:z-est(((0 1 0 1) (0 1 1 0)) (((1) (0)) ((0) (1)) (nil nil))):1}
      |a_{K_0, K_0', L_0, L_0'}|
      \leq \|T\| |K_0|^{1/{p}} |L_0|^{1/{q}} |K_0'|^{1/{p'}} |L_0'|^{1/{q'}}.
    \end{equation}
    Now, we write~\eqref{eq:z-est(((0 1 0 1) (0 1 1 0)) (((1) (0)) ((0) (1)) (nil nil))):0} as
    follows:
    \begin{equation}\label{eq:z-est(((0 1 0 1) (0 1 1 0)) (((1) (0)) ((0) (1)) (nil nil))):2}
      \sum_{K_0', L_0} \Bigl\langle T \sum_{K_0, L_0'} a_{K_0, K_0', L_0, L_0'} h_{K_0\times L_0'},  h_{K_0'\times L_0} \Bigr\rangle.
    \end{equation}
    By duality, we obtain the subsequent upper estimate for~\eqref{eq:z-est(((0 1 0 1) (0 1 1 0))
      (((1) (0)) ((0) (1)) (nil nil))):2}:
    \begin{equation}\label{eq:z-est(((0 1 0 1) (0 1 1 0)) (((1) (0)) ((0) (1)) (nil nil))):3}
      \sum_{K_0', L_0}\left\|T \sum_{K_0, L_0'} a_{K_0, K_0', L_0, L_0'} h_{K_0\times L_0'}\right\|_{H^p(H^q)} \left\| h_{K_0'\times L_0}\right\|_{(H^p(H^q))^*}.
    \end{equation}
    Estimate~\eqref{eq:z-est(((0 1 0 1) (0 1 1 0)) (((1) (0)) ((0) (1)) (nil nil))):1} and the
    disjointness of the dyadic intervals (see~\textrefp[J]{enu:j1}) yield
    \begin{equation}\label{eq:z-est(((0 1 0 1) (0 1 1 0)) (((1) (0)) ((0) (1)) (nil nil))):4}
      \sum_{K_0', L_0}\|T\| \left\| \sum_{K_0, L_0'}\max_{K_0, L_0'}\left(\|T\| |K_0|^{1/{p}} |L_0|^{1/{q}} |K_0'|^{1/{p'}} |L_0'|^{1/{q'}}\right) h_{K_0\times L_0'}\right\|_{H^p(H^q)}\left\| h_{K_0'\times L_0}\right\|_{(H^p(H^q))^*}.
    \end{equation}
    Consequently, we obtain
    \begin{equation}\label{eq:z-est(((0 1 0 1) (0 1 1 0)) (((1) (0)) ((0) (1)) (nil nil))):5}
      \sum_{K_0', L_0}\|T\| \max_{K_0, L_0'}\left(\|T\| |K_0|^{1/{p}} |L_0|^{1/{q}} |K_0'|^{1/{p'}} |L_0'|^{1/{q'}}\right) \left\| \sum_{K_0, L_0'}h_{K_0\times L_0'}\right\|_{H^p(H^q)}\left\| h_{K_0'\times L_0}\right\|_{(H^p(H^q))^*}.
    \end{equation}
    Thus,~\eqref{eq:z-est(((0 1 0 1) (0 1 1 0)) (((1) (0)) ((0) (1)) (nil nil))):5} is bounded from
    above by
    \begin{equation*}
      \|T\|^2 \max_{K_0, L_0'} |K_0|^{1/{p}} |L_0'|^{1/{q'}}  \sum_{K_0', L_0} |L_0| |K_0'|^{2/{p'}}.
    \end{equation*}
    Using Hölder's inequality yields
    \begin{equation}\label{eq:z-est(((0 1 0 1) (0 1 1 0)) (((1) (0)) ((0) (1)) (nil nil))):6}
      \|T\|^2 \max_{K_0, K_0', L_0'} |K_0|^{1/{p}} |L_0'|^{1/{q'}} |K_0'|^{2/{p'}-1}.
    \end{equation}
    Inserting $|K_0|, |K_0'|, |L_0'| \leq \alpha$ (see~\eqref{eq:alpha-small})
    into~\eqref{eq:z-est(((0 1 0 1) (0 1 1 0)) (((1) (0)) ((0) (1)) (nil nil))):6}, we obtain the
    estimate
    \begin{equation}\label{eq:z-est(((0 1 0 1) (0 1 1 0)) (((1) (0)) ((0) (1)) (nil nil))):7}
      \|T\|^2 \alpha^{1/{p'}+1/{q'}}.
    \end{equation}
  \end{proofcase}

  \begin{proofcase}[Case~\theproofcase, group~{\textrefp[f]{enu:proof:lem:var:Z:f:2}}: $K_0 = K_1 \neq K_0' = K_1'$, $L_0 = L_1' \neq L_0' = L_1$
    (((0) (1)) (nil nil) ((1) (0))) -- right variant]\label{case:z-est(((0 1 0 1) (0 1 1 0)) (((0)
      (1)) (nil nil) ((1) (0)))):right}
    In this case, we have to estimate
    \begin{equation}\label{eq:z-est(((0 1 0 1) (0 1 1 0)) (((0) (1)) (nil nil) ((1) (0)))):0}
      \sum_{K_0, L_0', K_0', L_0} \langle Th_{K_0\times L_0}, h_{K_0'\times L_0'} \rangle \langle Th_{K_0\times L_0'}, h_{K_0'\times L_0} \rangle.
    \end{equation}
    We put $a_{K_0, K_0', L_0, L_0'} = \langle T h_{K_0\times L_0}, h_{K_0'\times L_0'} \rangle$ and
    note the estimate
    \begin{equation}\label{eq:z-est(((0 1 0 1) (0 1 1 0)) (((0) (1)) (nil nil) ((1) (0)))):1}
      |a_{K_0, K_0', L_0, L_0'}|
      \leq \|T\| |K_0|^{1/{p}} |L_0|^{1/{q}} |K_0'|^{1/{p'}} |L_0'|^{1/{q'}}.
    \end{equation}
    Now, we write~\eqref{eq:z-est(((0 1 0 1) (0 1 1 0)) (((0) (1)) (nil nil) ((1) (0)))):0} as
    follows:
    \begin{equation}\label{eq:z-est(((0 1 0 1) (0 1 1 0)) (((0) (1)) (nil nil) ((1) (0)))):2}
      \sum_{K_0, L_0'} \Bigl\langle T  h_{K_0\times L_0'}, \sum_{K_0', L_0} a_{K_0, K_0', L_0, L_0'} h_{K_0'\times L_0} \Bigr\rangle.
    \end{equation}
    By duality, we obtain the subsequent upper estimate for~\eqref{eq:z-est(((0 1 0 1) (0 1 1 0))
      (((0) (1)) (nil nil) ((1) (0)))):2}:
    \begin{equation}\label{eq:z-est(((0 1 0 1) (0 1 1 0)) (((0) (1)) (nil nil) ((1) (0)))):3}
      \sum_{K_0, L_0'}\left\|T  h_{K_0\times L_0'}\right\|_{H^p(H^q)} \left\|\sum_{K_0', L_0} a_{K_0, K_0', L_0, L_0'} h_{K_0'\times L_0}\right\|_{(H^p(H^q))^*}.
    \end{equation}
    Estimate~\eqref{eq:z-est(((0 1 0 1) (0 1 1 0)) (((0) (1)) (nil nil) ((1) (0)))):1} and the
    disjointness of the dyadic intervals (see~\textrefp[J]{enu:j1}) yield
    \begin{equation}\label{eq:z-est(((0 1 0 1) (0 1 1 0)) (((0) (1)) (nil nil) ((1) (0)))):4}
      \sum_{K_0, L_0'}\|T\| \left\| h_{K_0\times L_0'}\right\|_{H^p(H^q)}\left\| \sum_{K_0', L_0}\max_{K_0', L_0}\left(\|T\| |K_0|^{1/{p}} |L_0|^{1/{q}} |K_0'|^{1/{p'}} |L_0'|^{1/{q'}}\right) h_{K_0'\times L_0}\right\|_{(H^p(H^q))^*}.
    \end{equation}
    Consequently, we obtain
    \begin{equation}\label{eq:z-est(((0 1 0 1) (0 1 1 0)) (((0) (1)) (nil nil) ((1) (0)))):5}
      \sum_{K_0, L_0'}\|T\| \left\| h_{K_0\times L_0'}\right\|_{H^p(H^q)}\max_{K_0', L_0}\left(\|T\| |K_0|^{1/{p}} |L_0|^{1/{q}} |K_0'|^{1/{p'}} |L_0'|^{1/{q'}}\right) \left\| \sum_{K_0', L_0}h_{K_0'\times L_0}\right\|_{(H^p(H^q))^*}.
    \end{equation}
    Thus,~\eqref{eq:z-est(((0 1 0 1) (0 1 1 0)) (((0) (1)) (nil nil) ((1) (0)))):5} is bounded from
    above by
    \begin{equation*}
      \|T\|^2 \max_{K_0', L_0}  |L_0|^{1/{q}} |K_0'|^{1/{p'}} \sum_{K_0, L_0'} |K_0|^{2/{p}} |L_0'|.
    \end{equation*}
    Using Hölder's inequality yields
    \begin{equation}\label{eq:z-est(((0 1 0 1) (0 1 1 0)) (((0) (1)) (nil nil) ((1) (0)))):6}
      \|T\|^2 \max_{K_0, K_0', L_0}  |L_0|^{1/{q}} |K_0'|^{1/{p'}}  |K_0|^{2/{p}-1}.
    \end{equation}
    Inserting $|K_0|, |K_0'|, |L_0| \leq \alpha$ (see~\eqref{eq:alpha-small})
    into~\eqref{eq:z-est(((0 1 0 1) (0 1 1 0)) (((0) (1)) (nil nil) ((1) (0)))):6}, we obtain the
    estimate
    \begin{equation}\label{eq:z-est(((0 1 0 1) (0 1 1 0)) (((0) (1)) (nil nil) ((1) (0)))):7}
      \|T\|^2 \alpha^{1/{p}+1/{q}}.
    \end{equation}
  \end{proofcase}

  \subsubsection*{Summary of \textref[Case~]{case:z-est(((0 1 0 1) (0 1 1 0)) (((1) (0)) ((0) (1))
      (nil nil))):left} and \textref[Case~]{case:z-est(((0 1 0 1) (0 1 1 0)) (((0) (1)) (nil nil)
      ((1) (0)))):right}}
  Combining~\eqref{eq:z-est(((0 1 0 1) (0 1 1 0)) (((1) (0)) ((0) (1)) (nil nil))):7}
  with~\eqref{eq:z-est(((0 1 0 1) (0 1 1 0)) (((0) (1)) (nil nil) ((1) (0)))):7} yields
  \begin{equation}\label{eq:z-est(((0 1 0 1) (0 1 1 0)) (((0) (1)) (nil nil) ((1) (0))))(((0 1 0 1)
      (0 1 1 0)) (((1) (0)) ((0) (1)) (nil nil)))-final}
    \cond_{\theta,\varepsilon} Z^2
    \leq \|T\|^2 \alpha.
  \end{equation}

  \begin{proofcase}[Case~\theproofcase, group~{\textrefp[f]{enu:proof:lem:var:Z:f:3}}: $K_0 = K_1 \neq K_0' = K_1'$, $L_0 = L_0' \neq L_1 = L_1'$ (((1) (2)) ((0) (0)) (nil nil)) -- left variant]\label{case:z-est(((0 1 0 1) (0 0 2 2)) (((1) (2)) ((0) (0)) (nil nil))):left}
    In this case, we have to estimate
    \begin{equation}\label{eq:z-est(((0 1 0 1) (0 0 2 2)) (((1) (2)) ((0) (0)) (nil nil))):0}
      \sum_{K_0', L_1, K_0, L_0} \langle Th_{K_0\times L_0}, h_{K_0'\times L_0} \rangle \langle Th_{K_0\times L_1}, h_{K_0'\times L_1} \rangle.
    \end{equation}
    We put $a_{K_0, K_0', L_0} = \langle T h_{K_0\times L_0}, h_{K_0'\times L_0} \rangle$ and note
    the estimate
    \begin{equation}\label{eq:z-est(((0 1 0 1) (0 0 2 2)) (((1) (2)) ((0) (0)) (nil nil))):1}
      |a_{K_0, K_0', L_0}|
      \leq \|T\| |K_0|^{1/{p}} |L_0|^{1/{q}} |K_0'|^{1/{p'}} |L_0|^{1/{q'}}.
    \end{equation}
    Now, we write~\eqref{eq:z-est(((0 1 0 1) (0 0 2 2)) (((1) (2)) ((0) (0)) (nil nil))):0} as
    follows:
    \begin{equation}\label{eq:z-est(((0 1 0 1) (0 0 2 2)) (((1) (2)) ((0) (0)) (nil nil))):2}
      \sum_{K_0', L_1} \Bigl\langle T \sum_{K_0, L_0} a_{K_0, K_0', L_0} h_{K_0\times L_1},  h_{K_0'\times L_1} \Bigr\rangle.
    \end{equation}
    By duality, we obtain the subsequent upper estimate for~\eqref{eq:z-est(((0 1 0 1) (0 0 2 2))
      (((1) (2)) ((0) (0)) (nil nil))):2}:
    \begin{equation}\label{eq:z-est(((0 1 0 1) (0 0 2 2)) (((1) (2)) ((0) (0)) (nil nil))):3}
      \sum_{K_0', L_1}\left\|T \sum_{K_0, L_0} a_{K_0, K_0', L_0} h_{K_0\times L_1}\right\|_{H^p(H^q)} \left\| h_{K_0'\times L_1}\right\|_{(H^p(H^q))^*}.
    \end{equation}
    Estimate~\eqref{eq:z-est(((0 1 0 1) (0 0 2 2)) (((1) (2)) ((0) (0)) (nil nil))):1} and the
    disjointness of the dyadic intervals (see~\textrefp[J]{enu:j1}) yield
    \begin{equation}\label{eq:z-est(((0 1 0 1) (0 0 2 2)) (((1) (2)) ((0) (0)) (nil nil))):4}
      \sum_{K_0', L_1}\|T\| \left\| \sum_{K_0, L_0}\max_{K_0}\left(\|T\| |K_0|^{1/{p}} |L_0|^{1/{q}} |K_0'|^{1/{p'}} |L_0|^{1/{q'}}\right) h_{K_0\times L_1}\right\|_{H^p(H^q)}\left\| h_{K_0'\times L_1}\right\|_{(H^p(H^q))^*}.
    \end{equation}
    Consequently, we obtain
    \begin{equation}\label{eq:z-est(((0 1 0 1) (0 0 2 2)) (((1) (2)) ((0) (0)) (nil nil))):5}
      \sum_{K_0', L_1}\|T\| \sum_{L_0}\max_{K_0}\left(\|T\| |K_0|^{1/{p}} |L_0|^{1/{q}} |K_0'|^{1/{p'}} |L_0|^{1/{q'}}\right) \left\| \sum_{K_0}h_{K_0\times L_1}\right\|_{H^p(H^q)}\left\| h_{K_0'\times L_1}\right\|_{(H^p(H^q))^*}.
    \end{equation}
    Thus,~\eqref{eq:z-est(((0 1 0 1) (0 0 2 2)) (((1) (2)) ((0) (0)) (nil nil))):5} is bounded from
    above by
    \begin{equation*}
      \|T\|^2 \max_{K_0} |K_0|^{1/{p}} \sum_{K_0', L_1} \sum_{L_0} |L_0| |K_0'|^{2/{p'}} |L_1|.
    \end{equation*}
    Using Hölder's inequality yields
    \begin{equation}\label{eq:z-est(((0 1 0 1) (0 0 2 2)) (((1) (2)) ((0) (0)) (nil nil))):6}
      \|T\|^2 \max_{K_0, K_0'} |K_0|^{1/{p}} |K_0'|^{2/{p'}-1}.
    \end{equation}
    Inserting $|K_0|, |K_0'|\leq \alpha$ (see~\eqref{eq:alpha-small}) into~\eqref{eq:z-est(((0 1 0
      1) (0 0 2 2)) (((1) (2)) ((0) (0)) (nil nil))):6}, we obtain the estimate
    \begin{equation}\label{eq:z-est(((0 1 0 1) (0 0 2 2)) (((1) (2)) ((0) (0)) (nil nil))):7}
      \|T\|^2 \alpha^{1/{p'}}.
    \end{equation}
  \end{proofcase}

  \begin{proofcase}[Case~\theproofcase, group~{\textrefp[f]{enu:proof:lem:var:Z:f:3}}: $K_0 = K_1 \neq K_0' = K_1'$, $L_0 = L_0' \neq L_1 = L_1'$ (((0) (2)) (nil nil) ((1) (0))) -- right variant]\label{case:z-est(((0 1 0 1) (0 0 2 2)) (((0) (2)) (nil nil) ((1) (0)))):right}
    In this case, we have to estimate
    \begin{equation}\label{eq:z-est(((0 1 0 1) (0 0 2 2)) (((0) (2)) (nil nil) ((1) (0)))):0}
      \sum_{K_0, L_1, K_0', L_0} \langle Th_{K_0\times L_0}, h_{K_0'\times L_0} \rangle \langle Th_{K_0\times L_1}, h_{K_0'\times L_1} \rangle.
    \end{equation}
    We put $a_{K_0, K_0', L_0} = \langle T h_{K_0\times L_0}, h_{K_0'\times L_0} \rangle$ and note
    the estimate
    \begin{equation}\label{eq:z-est(((0 1 0 1) (0 0 2 2)) (((0) (2)) (nil nil) ((1) (0)))):1}
      |a_{K_0, K_0', L_0}|
      \leq \|T\| |K_0|^{1/{p}} |L_0|^{1/{q}} |K_0'|^{1/{p'}} |L_0|^{1/{q'}}.
    \end{equation}
    Now, we write~\eqref{eq:z-est(((0 1 0 1) (0 0 2 2)) (((0) (2)) (nil nil) ((1) (0)))):0} as
    follows:
    \begin{equation}\label{eq:z-est(((0 1 0 1) (0 0 2 2)) (((0) (2)) (nil nil) ((1) (0)))):2}
      \sum_{K_0, L_1} \Bigl\langle T  h_{K_0\times L_1}, \sum_{K_0', L_0} a_{K_0, K_0', L_0} h_{K_0'\times L_1} \Bigr\rangle.
    \end{equation}
    By duality, we obtain the subsequent upper estimate for~\eqref{eq:z-est(((0 1 0 1) (0 0 2 2))
      (((0) (2)) (nil nil) ((1) (0)))):2}:
    \begin{equation}\label{eq:z-est(((0 1 0 1) (0 0 2 2)) (((0) (2)) (nil nil) ((1) (0)))):3}
      \sum_{K_0, L_1}\left\|T  h_{K_0\times L_1}\right\|_{H^p(H^q)} \left\|\sum_{K_0', L_0} a_{K_0, K_0', L_0} h_{K_0'\times L_1}\right\|_{(H^p(H^q))^*}.
    \end{equation}
    Estimate~\eqref{eq:z-est(((0 1 0 1) (0 0 2 2)) (((0) (2)) (nil nil) ((1) (0)))):1} and the
    disjointness of the dyadic intervals (see~\textrefp[J]{enu:j1}) yield
    \begin{equation}\label{eq:z-est(((0 1 0 1) (0 0 2 2)) (((0) (2)) (nil nil) ((1) (0)))):4}
      \sum_{K_0, L_1}\|T\| \left\| h_{K_0\times L_1}\right\|_{H^p(H^q)}\left\| \sum_{K_0', L_0}\max_{K_0'}\left(\|T\| |K_0|^{1/{p}} |L_0|^{1/{q}} |K_0'|^{1/{p'}} |L_0|^{1/{q'}}\right) h_{K_0'\times L_1}\right\|_{(H^p(H^q))^*}.
    \end{equation}
    Consequently, we obtain
    \begin{equation}\label{eq:z-est(((0 1 0 1) (0 0 2 2)) (((0) (2)) (nil nil) ((1) (0)))):5}
      \sum_{K_0, L_1}\|T\| \left\| h_{K_0\times L_1}\right\|_{H^p(H^q)}\sum_{L_0}\max_{K_0'}\left(\|T\| |K_0|^{1/{p}} |L_0|^{1/{q}} |K_0'|^{1/{p'}} |L_0|^{1/{q'}}\right) \left\| \sum_{K_0'}h_{K_0'\times L_1}\right\|_{(H^p(H^q))^*}.
    \end{equation}
    Thus,~\eqref{eq:z-est(((0 1 0 1) (0 0 2 2)) (((0) (2)) (nil nil) ((1) (0)))):5} is bounded from
    above by
    \begin{equation*}
      \|T\|^2 \max_{K_0'}  |K_0'|^{1/{p'}} \sum_{K_0, L_1} \sum_{L_0} |K_0|^{2/{p}} |L_0| |L_1|.
    \end{equation*}
    Using Hölder's inequality yields
    \begin{equation}\label{eq:z-est(((0 1 0 1) (0 0 2 2)) (((0) (2)) (nil nil) ((1) (0)))):6}
      \|T\|^2 \max_{K_0, K_0'}  |K_0'|^{1/{p'}} |K_0|^{2/{p}-1}.
    \end{equation}
    Inserting $|K_0|, |K_0'|\leq \alpha$ (see~\eqref{eq:alpha-small}) into~\eqref{eq:z-est(((0 1 0
      1) (0 0 2 2)) (((0) (2)) (nil nil) ((1) (0)))):6}, we obtain the estimate
    \begin{equation}\label{eq:z-est(((0 1 0 1) (0 0 2 2)) (((0) (2)) (nil nil) ((1) (0)))):7}
      \|T\|^2 \alpha^{1/{p}}.
    \end{equation}
  \end{proofcase}

  \subsubsection*{Summary of \textref[Case~]{case:z-est(((0 1 0 1) (0 0 2 2)) (((1) (2)) ((0) (0))
      (nil nil))):left} and \textref[Case~]{case:z-est(((0 1 0 1) (0 0 2 2)) (((0) (2)) (nil nil)
      ((1) (0)))):right}}
  Combining~\eqref{eq:z-est(((0 1 0 1) (0 0 2 2)) (((1) (2)) ((0) (0)) (nil nil))):7}
  with~\eqref{eq:z-est(((0 1 0 1) (0 0 2 2)) (((0) (2)) (nil nil) ((1) (0)))):7} yields
  \begin{equation}\label{eq:z-est(((0 1 0 1) (0 0 2 2)) (((1) (2)) ((0) (0)) (nil nil)))(((0 1 0 1)
      (0 0 2 2)) (((0) (2)) (nil nil) ((1) (0))))-final}
    \cond_{\theta,\varepsilon} Z^2
    \leq \|T\|^2 \alpha^{1/2}.
  \end{equation}
    
  \begin{proofcase}[Case~\theproofcase, group~{\textrefp[f]{enu:proof:lem:var:Z:f:4}}: $K_0 = K_1' \neq K_0' = K_1$, $L_0 = L_1 \neq L_0' = L_1'$ (((0) (1)) ((1) (0)) (nil nil)) -- left variant]\label{case:z-est(((0 1 1 0) (0 1 0 1)) (((0) (1)) ((1) (0)) (nil nil))):left}
    In this case, we have to estimate
    \begin{equation}\label{eq:z-est(((0 1 1 0) (0 1 0 1)) (((0) (1)) ((1) (0)) (nil nil))):0}
      \sum_{K_0, L_0', K_0', L_0} \langle Th_{K_0\times L_0}, h_{K_0'\times L_0'} \rangle \langle Th_{K_0'\times L_0}, h_{K_0\times L_0'} \rangle.
    \end{equation}
    We put $a_{K_0, K_0', L_0, L_0'} = \langle T h_{K_0\times L_0}, h_{K_0'\times L_0'} \rangle$ and
    note the estimate
    \begin{equation}\label{eq:z-est(((0 1 1 0) (0 1 0 1)) (((0) (1)) ((1) (0)) (nil nil))):1}
      |a_{K_0, K_0', L_0, L_0'}|
      \leq \|T\| |K_0|^{1/{p}} |L_0|^{1/{q}} |K_0'|^{1/{p'}} |L_0'|^{1/{q'}}.
    \end{equation}
    Now, we write~\eqref{eq:z-est(((0 1 1 0) (0 1 0 1)) (((0) (1)) ((1) (0)) (nil nil))):0} as
    follows:
    \begin{equation}\label{eq:z-est(((0 1 1 0) (0 1 0 1)) (((0) (1)) ((1) (0)) (nil nil))):2}
      \sum_{K_0, L_0'} \Bigl\langle T \sum_{K_0', L_0} a_{K_0, K_0', L_0, L_0'} h_{K_0'\times L_0},  h_{K_0\times L_0'} \Bigr\rangle.
    \end{equation}
    By duality, we obtain the subsequent upper estimate for~\eqref{eq:z-est(((0 1 1 0) (0 1 0 1))
      (((0) (1)) ((1) (0)) (nil nil))):2}:
    \begin{equation}\label{eq:z-est(((0 1 1 0) (0 1 0 1)) (((0) (1)) ((1) (0)) (nil nil))):3}
      \sum_{K_0, L_0'}\left\|T \sum_{K_0', L_0} a_{K_0, K_0', L_0, L_0'} h_{K_0'\times L_0}\right\|_{H^p(H^q)} \left\| h_{K_0\times L_0'}\right\|_{(H^p(H^q))^*}.
    \end{equation}
    Estimate~\eqref{eq:z-est(((0 1 1 0) (0 1 0 1)) (((0) (1)) ((1) (0)) (nil nil))):1} and the
    disjointness of the dyadic intervals (see~\textrefp[J]{enu:j1}) yield
    \begin{equation}\label{eq:z-est(((0 1 1 0) (0 1 0 1)) (((0) (1)) ((1) (0)) (nil nil))):4}
      \sum_{K_0, L_0'}\|T\| \left\| \sum_{K_0', L_0}\max_{K_0', L_0}\left(\|T\| |K_0|^{1/{p}} |L_0|^{1/{q}} |K_0'|^{1/{p'}} |L_0'|^{1/{q'}}\right) h_{K_0'\times L_0}\right\|_{H^p(H^q)}\left\| h_{K_0\times L_0'}\right\|_{(H^p(H^q))^*}.
    \end{equation}
    Consequently, we obtain
    \begin{equation}\label{eq:z-est(((0 1 1 0) (0 1 0 1)) (((0) (1)) ((1) (0)) (nil nil))):5}
      \sum_{K_0, L_0'}\|T\| \max_{K_0', L_0}\left(\|T\| |K_0|^{1/{p}} |L_0|^{1/{q}} |K_0'|^{1/{p'}} |L_0'|^{1/{q'}}\right) \left\| \sum_{K_0', L_0}h_{K_0'\times L_0}\right\|_{H^p(H^q)}\left\| h_{K_0\times L_0'}\right\|_{(H^p(H^q))^*}.
    \end{equation}
    Thus,~\eqref{eq:z-est(((0 1 1 0) (0 1 0 1)) (((0) (1)) ((1) (0)) (nil nil))):5} is bounded from
    above by
    \begin{equation*}
      \|T\|^2 \max_{K_0', L_0} |L_0|^{1/{q}} |K_0'|^{1/{p'}} \sum_{K_0, L_0'} |K_0| |L_0'|^{2/{q'}}.
    \end{equation*}
    Using Hölder's inequality yields
    \begin{equation}\label{eq:z-est(((0 1 1 0) (0 1 0 1)) (((0) (1)) ((1) (0)) (nil nil))):6}
      \|T\|^2 \max_{K_0', L_0, L_0'} |L_0|^{1/{q}} |K_0'|^{1/{p'}} |L_0'|^{2/{q'}-1}.
    \end{equation}
    Inserting $|K_0'|, |L_0|, |L_0'|\leq \alpha$ (see~\eqref{eq:alpha-small})
    into~\eqref{eq:z-est(((0 1 1 0) (0 1 0 1)) (((0) (1)) ((1) (0)) (nil nil))):6}, we obtain the
    estimate
    \begin{equation}\label{eq:z-est(((0 1 1 0) (0 1 0 1)) (((0) (1)) ((1) (0)) (nil nil))):7}
      \|T\|^2 \alpha^{1/{p'}+1/{q'}}.
    \end{equation}
  \end{proofcase}

  \begin{proofcase}[Case~\theproofcase, group~{\textrefp[f]{enu:proof:lem:var:Z:f:4}}: $K_0 = K_1' \neq K_0' = K_1$, $L_0 = L_1 \neq L_0' = L_1'$ (((1) (0)) (nil nil) ((0) (1))) -- right variant]\label{case:z-est(((0 1 1 0) (0 1 0 1)) (((1) (0)) (nil nil) ((0) (1)))):right}
    In this case, we have to estimate
    \begin{equation}\label{eq:z-est(((0 1 1 0) (0 1 0 1)) (((1) (0)) (nil nil) ((0) (1)))):0}
      \sum_{K_0', L_0, K_0, L_0'} \langle Th_{K_0\times L_0}, h_{K_0'\times L_0'} \rangle \langle Th_{K_0'\times L_0}, h_{K_0\times L_0'} \rangle.
    \end{equation}
    We put $a_{K_0, K_0', L_0, L_0'} = \langle T h_{K_0\times L_0}, h_{K_0'\times L_0'} \rangle$ and
    note the estimate
    \begin{equation}\label{eq:z-est(((0 1 1 0) (0 1 0 1)) (((1) (0)) (nil nil) ((0) (1)))):1}
      |a_{K_0, K_0', L_0, L_0'}|
      \leq \|T\| |K_0|^{1/{p}} |L_0|^{1/{q}} |K_0'|^{1/{p'}} |L_0'|^{1/{q'}}.
    \end{equation}
    Now, we write~\eqref{eq:z-est(((0 1 1 0) (0 1 0 1)) (((1) (0)) (nil nil) ((0) (1)))):0} as
    follows:
    \begin{equation}\label{eq:z-est(((0 1 1 0) (0 1 0 1)) (((1) (0)) (nil nil) ((0) (1)))):2}
      \sum_{K_0', L_0} \Bigl\langle T  h_{K_0'\times L_0}, \sum_{K_0, L_0'} a_{K_0, K_0', L_0, L_0'} h_{K_0\times L_0'} \Bigr\rangle.
    \end{equation}
    By duality, we obtain the subsequent upper estimate for~\eqref{eq:z-est(((0 1 1 0) (0 1 0 1))
      (((1) (0)) (nil nil) ((0) (1)))):2}:
    \begin{equation}\label{eq:z-est(((0 1 1 0) (0 1 0 1)) (((1) (0)) (nil nil) ((0) (1)))):3}
      \sum_{K_0', L_0}\left\|T  h_{K_0'\times L_0}\right\|_{H^p(H^q)} \left\|\sum_{K_0, L_0'} a_{K_0, K_0', L_0, L_0'} h_{K_0\times L_0'}\right\|_{(H^p(H^q))^*}.
    \end{equation}
    Estimate~\eqref{eq:z-est(((0 1 1 0) (0 1 0 1)) (((1) (0)) (nil nil) ((0) (1)))):1} and the
    disjointness of the dyadic intervals (see~\textrefp[J]{enu:j1}) yield
    \begin{equation}\label{eq:z-est(((0 1 1 0) (0 1 0 1)) (((1) (0)) (nil nil) ((0) (1)))):4}
      \sum_{K_0', L_0}\|T\| \left\| h_{K_0'\times L_0}\right\|_{H^p(H^q)}\left\| \sum_{K_0, L_0'}\max_{K_0, L_0'}\left(\|T\| |K_0|^{1/{p}} |L_0|^{1/{q}} |K_0'|^{1/{p'}} |L_0'|^{1/{q'}}\right) h_{K_0\times L_0'}\right\|_{(H^p(H^q))^*}.
    \end{equation}
    Consequently, we obtain
    \begin{equation}\label{eq:z-est(((0 1 1 0) (0 1 0 1)) (((1) (0)) (nil nil) ((0) (1)))):5}
      \sum_{K_0', L_0}\|T\| \left\| h_{K_0'\times L_0}\right\|_{H^p(H^q)}\max_{K_0, L_0'}\left(\|T\| |K_0|^{1/{p}} |L_0|^{1/{q}} |K_0'|^{1/{p'}} |L_0'|^{1/{q'}}\right) \left\| \sum_{K_0, L_0'}h_{K_0\times L_0'}\right\|_{(H^p(H^q))^*}.
    \end{equation}
    Thus,~\eqref{eq:z-est(((0 1 1 0) (0 1 0 1)) (((1) (0)) (nil nil) ((0) (1)))):5} is bounded from
    above by
    \begin{equation*}
      \|T\|^2 \max_{K_0, L_0'} |K_0|^{1/{p}}  |L_0'|^{1/{q'}} \sum_{K_0', L_0} |L_0|^{2/{q}} |K_0'|.
    \end{equation*}
    Using Hölder's inequality yields
    \begin{equation}\label{eq:z-est(((0 1 1 0) (0 1 0 1)) (((1) (0)) (nil nil) ((0) (1)))):6}
      \|T\|^2 \max_{K_0, L_0, L_0'} |K_0|^{1/{p}}  |L_0'|^{1/{q'}} |L_0|^{2/{q}-1}.
    \end{equation}
    Inserting $|K_0|, |L_0|, |L_0'|\leq \alpha$ (see~\eqref{eq:alpha-small})
    into~\eqref{eq:z-est(((0 1 1 0) (0 1 0 1)) (((1) (0)) (nil nil) ((0) (1)))):6}, we obtain the
    estimate
    \begin{equation}\label{eq:z-est(((0 1 1 0) (0 1 0 1)) (((1) (0)) (nil nil) ((0) (1)))):7}
      \|T\|^2 \alpha^{1/{p}+1/{q}}.
    \end{equation}
  \end{proofcase}

  \subsubsection*{Summary of \textref[Case~]{case:z-est(((0 1 1 0) (0 1 0 1)) (((0) (1)) ((1) (0))
      (nil nil))):left} and \textref[Case~]{case:z-est(((0 1 1 0) (0 1 0 1)) (((1) (0)) (nil nil)
      ((0) (1)))):right}}
  Combining~\eqref{eq:z-est(((0 1 1 0) (0 1 0 1)) (((0) (1)) ((1) (0)) (nil nil))):7}
  with~\eqref{eq:z-est(((0 1 1 0) (0 1 0 1)) (((1) (0)) (nil nil) ((0) (1)))):7} yields
  \begin{equation}\label{eq:z-est(((0 1 1 0) (0 1 0 1)) (((0) (1)) ((1) (0)) (nil nil)))(((0 1 1 0)
      (0 1 0 1)) (((1) (0)) (nil nil) ((0) (1))))-final}
    \cond_{\theta,\varepsilon} Z^2
    \leq \|T\|^2 \alpha.
  \end{equation}

  \begin{proofcase}[Case~\theproofcase, group~{\textrefp[f]{enu:proof:lem:var:Z:f:5}}: $K_0 = K_1' \neq K_0' = K_1$, $L_0 = L_1' \neq L_0' = L_1$ (((0) (0)) ((1) (1)) (nil nil)) -- left variant]\label{case:z-est(((0 1 1 0) (0 1 1 0)) (((0) (0)) ((1) (1)) (nil nil))):left}
    In this case, we have to estimate
    \begin{equation}\label{eq:z-est(((0 1 1 0) (0 1 1 0)) (((0) (0)) ((1) (1)) (nil nil))):0}
      \sum_{K_0, L_0, K_0', L_0'} \langle Th_{K_0\times L_0}, h_{K_0'\times L_0'} \rangle \langle Th_{K_0'\times L_0'}, h_{K_0\times L_0} \rangle.
    \end{equation}
    We put $a_{K_0, K_0', L_0, L_0'} = \langle T h_{K_0\times L_0}, h_{K_0'\times L_0'} \rangle$ and
    note the estimate
    \begin{equation}\label{eq:z-est(((0 1 1 0) (0 1 1 0)) (((0) (0)) ((1) (1)) (nil nil))):1}
      |a_{K_0, K_0', L_0, L_0'}|
      \leq \|T\| |K_0|^{1/{p}} |L_0|^{1/{q}} |K_0'|^{1/{p'}} |L_0'|^{1/{q'}}.
    \end{equation}
    Now, we write~\eqref{eq:z-est(((0 1 1 0) (0 1 1 0)) (((0) (0)) ((1) (1)) (nil nil))):0} as
    follows:
    \begin{equation}\label{eq:z-est(((0 1 1 0) (0 1 1 0)) (((0) (0)) ((1) (1)) (nil nil))):2}
      \sum_{K_0, L_0} \Bigl\langle T \sum_{K_0', L_0'} a_{K_0, K_0', L_0, L_0'} h_{K_0'\times L_0'},  h_{K_0\times L_0} \Bigr\rangle.
    \end{equation}
    By duality, we obtain the subsequent upper estimate for~\eqref{eq:z-est(((0 1 1 0) (0 1 1 0))
      (((0) (0)) ((1) (1)) (nil nil))):2}:
    \begin{equation}\label{eq:z-est(((0 1 1 0) (0 1 1 0)) (((0) (0)) ((1) (1)) (nil nil))):3}
      \sum_{K_0, L_0}\left\|T \sum_{K_0', L_0'} a_{K_0, K_0', L_0, L_0'} h_{K_0'\times L_0'}\right\|_{H^p(H^q)} \left\| h_{K_0\times L_0}\right\|_{(H^p(H^q))^*}.
    \end{equation}
    Estimate~\eqref{eq:z-est(((0 1 1 0) (0 1 1 0)) (((0) (0)) ((1) (1)) (nil nil))):1} and the
    disjointness of the dyadic intervals (see~\textrefp[J]{enu:j1}) yield
    \begin{equation}\label{eq:z-est(((0 1 1 0) (0 1 1 0)) (((0) (0)) ((1) (1)) (nil nil))):4}
      \sum_{K_0, L_0}\|T\| \left\| \sum_{K_0', L_0'}\max_{K_0', L_0'}\left(\|T\| |K_0|^{1/{p}} |L_0|^{1/{q}} |K_0'|^{1/{p'}} |L_0'|^{1/{q'}}\right) h_{K_0'\times L_0'}\right\|_{H^p(H^q)}\left\| h_{K_0\times L_0}\right\|_{(H^p(H^q))^*}.
    \end{equation}
    Consequently, we obtain
    \begin{equation}\label{eq:z-est(((0 1 1 0) (0 1 1 0)) (((0) (0)) ((1) (1)) (nil nil))):5}
      \sum_{K_0, L_0}\|T\| \max_{K_0', L_0'}\left(\|T\| |K_0|^{1/{p}} |L_0|^{1/{q}} |K_0'|^{1/{p'}} |L_0'|^{1/{q'}}\right) \left\| \sum_{K_0', L_0'}h_{K_0'\times L_0'}\right\|_{H^p(H^q)}\left\| h_{K_0\times L_0}\right\|_{(H^p(H^q))^*}.
    \end{equation}
    Thus,~\eqref{eq:z-est(((0 1 1 0) (0 1 1 0)) (((0) (0)) ((1) (1)) (nil nil))):5} is bounded from
    above by
    \begin{equation}\label{eq:z-est(((0 1 1 0) (0 1 1 0)) (((0) (0)) ((1) (1)) (nil nil))):6}
      \|T\|^2 \max_{K_0', L_0'} |K_0'|^{1/{p'}} |L_0'|^{1/{q'}} \sum_{K_0, L_0}  |K_0| |L_0|
      \leq \|T\|^2 \max_{K_0', L_0'} |K_0'|^{1/{p'}} |L_0'|^{1/{q'}}.
    \end{equation}
    Inserting $|K_0'|, |L_0'|\leq \alpha$ (see~\eqref{eq:alpha-small}) into~\eqref{eq:z-est(((0 1 1
      0) (0 1 1 0)) (((0) (0)) ((1) (1)) (nil nil))):6}, we obtain the estimate
    \begin{equation}\label{eq:z-est(((0 1 1 0) (0 1 1 0)) (((0) (0)) ((1) (1)) (nil nil))):7}
      \|T\|^2 \alpha^{1/{p'}+1/{q'}}.
    \end{equation}
  \end{proofcase}

  \begin{proofcase}[Case~\theproofcase, group~{\textrefp[f]{enu:proof:lem:var:Z:f:5}}: $K_0 = K_1' \neq K_0' = K_1$, $L_0 = L_1' \neq L_0' = L_1$ (((1) (1)) (nil nil) ((0) (0))) -- right variant]\label{case:z-est(((0 1 1 0) (0 1 1 0)) (((1) (1)) (nil nil) ((0) (0)))):right}
    In this case, we have to estimate
    \begin{equation}\label{eq:z-est(((0 1 1 0) (0 1 1 0)) (((1) (1)) (nil nil) ((0) (0)))):0}
      \sum_{K_0', L_0', K_0, L_0} \langle Th_{K_0\times L_0}, h_{K_0'\times L_0'} \rangle \langle Th_{K_0'\times L_0'}, h_{K_0\times L_0} \rangle.
    \end{equation}
    We put $a_{K_0, K_0', L_0, L_0'} = \langle T h_{K_0\times L_0}, h_{K_0'\times L_0'} \rangle$ and
    note the estimate
    \begin{equation}\label{eq:z-est(((0 1 1 0) (0 1 1 0)) (((1) (1)) (nil nil) ((0) (0)))):1}
      |a_{K_0, K_0', L_0, L_0'}|
      \leq \|T\| |K_0|^{1/{p}} |L_0|^{1/{q}} |K_0'|^{1/{p'}} |L_0'|^{1/{q'}}.
    \end{equation}
    Now, we write~\eqref{eq:z-est(((0 1 1 0) (0 1 1 0)) (((1) (1)) (nil nil) ((0) (0)))):0} as
    follows:
    \begin{equation}\label{eq:z-est(((0 1 1 0) (0 1 1 0)) (((1) (1)) (nil nil) ((0) (0)))):2}
      \sum_{K_0', L_0'} \Bigl\langle T  h_{K_0'\times L_0'}, \sum_{K_0, L_0} a_{K_0, K_0', L_0, L_0'} h_{K_0\times L_0} \Bigr\rangle.
    \end{equation}
    By duality, we obtain the subsequent upper estimate for~\eqref{eq:z-est(((0 1 1 0) (0 1 1 0))
      (((1) (1)) (nil nil) ((0) (0)))):2}:
    \begin{equation}\label{eq:z-est(((0 1 1 0) (0 1 1 0)) (((1) (1)) (nil nil) ((0) (0)))):3}
      \sum_{K_0', L_0'}\left\|T  h_{K_0'\times L_0'}\right\|_{H^p(H^q)} \left\|\sum_{K_0, L_0} a_{K_0, K_0', L_0, L_0'} h_{K_0\times L_0}\right\|_{(H^p(H^q))^*}.
    \end{equation}
    Estimate~\eqref{eq:z-est(((0 1 1 0) (0 1 1 0)) (((1) (1)) (nil nil) ((0) (0)))):1} and the
    disjointness of the dyadic intervals (see~\textrefp[J]{enu:j1}) yield
    \begin{equation}\label{eq:z-est(((0 1 1 0) (0 1 1 0)) (((1) (1)) (nil nil) ((0) (0)))):4}
      \sum_{K_0', L_0'}\|T\| \left\| h_{K_0'\times L_0'}\right\|_{H^p(H^q)}\left\| \sum_{K_0, L_0}\max_{K_0, L_0}\left(\|T\| |K_0|^{1/{p}} |L_0|^{1/{q}} |K_0'|^{1/{p'}} |L_0'|^{1/{q'}}\right) h_{K_0\times L_0}\right\|_{(H^p(H^q))^*}.
    \end{equation}
    Consequently, we obtain
    \begin{equation}\label{eq:z-est(((0 1 1 0) (0 1 1 0)) (((1) (1)) (nil nil) ((0) (0)))):5}
      \sum_{K_0', L_0'}\|T\| \left\| h_{K_0'\times L_0'}\right\|_{H^p(H^q)}\max_{K_0, L_0}\left(\|T\| |K_0|^{1/{p}} |L_0|^{1/{q}} |K_0'|^{1/{p'}} |L_0'|^{1/{q'}}\right) \left\| \sum_{K_0, L_0}h_{K_0\times L_0}\right\|_{(H^p(H^q))^*}.
    \end{equation}
    Thus,~\eqref{eq:z-est(((0 1 1 0) (0 1 1 0)) (((1) (1)) (nil nil) ((0) (0)))):5} is bounded from
    above by
    \begin{equation}\label{eq:z-est(((0 1 1 0) (0 1 1 0)) (((1) (1)) (nil nil) ((0) (0)))):6}
      \|T\|^2 \max_{K_0, L_0} |K_0|^{1/{p}} |L_0|^{1/{q}} \sum_{K_0', L_0'} |K_0'| |L_0'|
      \leq \|T\|^2 \max_{K_0, L_0} |K_0|^{1/{p}} |L_0|^{1/{q}}.
    \end{equation}
    Inserting $|K_0|, |L_0|\leq \alpha$ (see~\eqref{eq:alpha-small}) into~\eqref{eq:z-est(((0 1 1 0)
      (0 1 1 0)) (((1) (1)) (nil nil) ((0) (0)))):6}, we obtain the estimate
    \begin{equation}\label{eq:z-est(((0 1 1 0) (0 1 1 0)) (((1) (1)) (nil nil) ((0) (0)))):7}
      \|T\|^2 \alpha^{1/{p}+1/{q}}.
    \end{equation}
  \end{proofcase}

  \subsubsection*{Summary of \textref[Case~]{case:z-est(((0 1 1 0) (0 1 1 0)) (((0) (0)) ((1) (1))
      (nil nil))):left} and \textref[Case~]{case:z-est(((0 1 1 0) (0 1 1 0)) (((1) (1)) (nil nil)
      ((0) (0)))):right}}
  Combining~\eqref{eq:z-est(((0 1 1 0) (0 1 1 0)) (((0) (0)) ((1) (1)) (nil nil))):7}
  with~\eqref{eq:z-est(((0 1 1 0) (0 1 1 0)) (((1) (1)) (nil nil) ((0) (0)))):7} yields
  \begin{equation}\label{eq:z-est(((0 1 1 0) (0 1 1 0)) (((0) (0)) ((1) (1)) (nil nil)))(((0 1 1 0)
      (0 1 1 0)) (((1) (1)) (nil nil) ((0) (0))))-final}
    \cond_{\theta,\varepsilon} Z^2
    \leq \|T\|^2 \alpha.
  \end{equation}

  \begin{proofcase}[Case~\theproofcase, group~{\textrefp[f]{enu:proof:lem:var:Z:f:6}}: $K_0 = K_1' \neq K_0' = K_1$, $L_0 = L_0' \neq L_1 = L_1'$ (((0) (2)) ((1) (0)) (nil nil)) -- left variant]\label{case:z-est(((0 1 1 0) (0 0 2 2)) (((0) (2)) ((1) (0)) (nil nil))):left}
    In this case, we have to estimate
    \begin{equation}\label{eq:z-est(((0 1 1 0) (0 0 2 2)) (((0) (2)) ((1) (0)) (nil nil))):0}
      \sum_{K_0, L_1, K_0', L_0} \langle Th_{K_0\times L_0}, h_{K_0'\times L_0} \rangle \langle Th_{K_0'\times L_1}, h_{K_0\times L_1} \rangle.
    \end{equation}
    We put $a_{K_0, K_0', L_0} = \langle T h_{K_0\times L_0}, h_{K_0'\times L_0} \rangle$ and note
    the estimate
    \begin{equation}\label{eq:z-est(((0 1 1 0) (0 0 2 2)) (((0) (2)) ((1) (0)) (nil nil))):1}
      |a_{K_0, K_0', L_0}|
      \leq \|T\| |K_0|^{1/{p}} |L_0|^{1/{q}} |K_0'|^{1/{p'}} |L_0|^{1/{q'}}.
    \end{equation}
    Now, we write~\eqref{eq:z-est(((0 1 1 0) (0 0 2 2)) (((0) (2)) ((1) (0)) (nil nil))):0} as
    follows:
    \begin{equation}\label{eq:z-est(((0 1 1 0) (0 0 2 2)) (((0) (2)) ((1) (0)) (nil nil))):2}
      \sum_{K_0, L_1} \Bigl\langle T \sum_{K_0', L_0} a_{K_0, K_0', L_0} h_{K_0'\times L_1},  h_{K_0\times L_1} \Bigr\rangle.
    \end{equation}
    By duality, we obtain the subsequent upper estimate for~\eqref{eq:z-est(((0 1 1 0) (0 0 2 2))
      (((0) (2)) ((1) (0)) (nil nil))):2}:
    \begin{equation}\label{eq:z-est(((0 1 1 0) (0 0 2 2)) (((0) (2)) ((1) (0)) (nil nil))):3}
      \sum_{K_0, L_1}\left\|T \sum_{K_0', L_0} a_{K_0, K_0', L_0} h_{K_0'\times L_1}\right\|_{H^p(H^q)} \left\| h_{K_0\times L_1}\right\|_{(H^p(H^q))^*}.
    \end{equation}
    Estimate~\eqref{eq:z-est(((0 1 1 0) (0 0 2 2)) (((0) (2)) ((1) (0)) (nil nil))):1} and the
    disjointness of the dyadic intervals (see~\textrefp[J]{enu:j1}) yield
    \begin{equation}\label{eq:z-est(((0 1 1 0) (0 0 2 2)) (((0) (2)) ((1) (0)) (nil nil))):4}
      \sum_{K_0, L_1}\|T\| \left\| \sum_{K_0', L_0}\max_{K_0'}\left(\|T\| |K_0|^{1/{p}} |L_0|^{1/{q}} |K_0'|^{1/{p'}} |L_0|^{1/{q'}}\right) h_{K_0'\times L_1}\right\|_{H^p(H^q)}\left\| h_{K_0\times L_1}\right\|_{(H^p(H^q))^*}.
    \end{equation}
    Consequently, we obtain
    \begin{equation}\label{eq:z-est(((0 1 1 0) (0 0 2 2)) (((0) (2)) ((1) (0)) (nil nil))):5}
      \sum_{K_0, L_1}\|T\| \sum_{L_0}\max_{K_0'}\left(\|T\| |K_0|^{1/{p}} |L_0|^{1/{q}} |K_0'|^{1/{p'}} |L_0|^{1/{q'}}\right) \left\| \sum_{K_0'}h_{K_0'\times L_1}\right\|_{H^p(H^q)}\left\| h_{K_0\times L_1}\right\|_{(H^p(H^q))^*}.
    \end{equation}
    Thus,~\eqref{eq:z-est(((0 1 1 0) (0 0 2 2)) (((0) (2)) ((1) (0)) (nil nil))):5} is bounded from
    above by
    \begin{equation}\label{eq:z-est(((0 1 1 0) (0 0 2 2)) (((0) (2)) ((1) (0)) (nil nil))):6}
      \|T\|^2 \max_{K_0'} |K_0'|^{1/{p'}} \sum_{K_0, L_1} \sum_{L_0} |K_0| |L_0| |L_1|
      \leq \|T\|^2 \max_{K_0'} |K_0'|^{1/{p'}}.
    \end{equation}
    Inserting $|K_0'|\leq \alpha$ (see~\eqref{eq:alpha-small}) into~\eqref{eq:z-est(((0 1 1 0) (0 0
      2 2)) (((0) (2)) ((1) (0)) (nil nil))):6}, we obtain the estimate
    \begin{equation}\label{eq:z-est(((0 1 1 0) (0 0 2 2)) (((0) (2)) ((1) (0)) (nil nil))):7}
      \|T\|^2 \alpha^{1/{p'}}.
    \end{equation}
  \end{proofcase}

  \begin{proofcase}[Case~\theproofcase, group~{\textrefp[f]{enu:proof:lem:var:Z:f:6}}: $K_0 = K_1' \neq K_0' = K_1$, $L_0 = L_0' \neq L_1 = L_1'$
    (((1) (2)) (nil nil) ((0) (0))) -- right variant]\label{case:z-est(((0 1 1 0) (0 0 2 2)) (((1)
      (2)) (nil nil) ((0) (0)))):right}
    In this case, we have to estimate
    \begin{equation}\label{eq:z-est(((0 1 1 0) (0 0 2 2)) (((1) (2)) (nil nil) ((0) (0)))):0}
      \sum_{K_0', L_1, K_0, L_0} \langle Th_{K_0\times L_0}, h_{K_0'\times L_0} \rangle \langle Th_{K_0'\times L_1}, h_{K_0\times L_1} \rangle.
    \end{equation}
    We put $a_{K_0, K_0', L_0} = \langle T h_{K_0\times L_0}, h_{K_0'\times L_0} \rangle$ and note
    the estimate
    \begin{equation}\label{eq:z-est(((0 1 1 0) (0 0 2 2)) (((1) (2)) (nil nil) ((0) (0)))):1}
      |a_{K_0, K_0', L_0}|
      \leq \|T\| |K_0|^{1/{p}} |L_0|^{1/{q}} |K_0'|^{1/{p'}} |L_0|^{1/{q'}}.
    \end{equation}
    Now, we write~\eqref{eq:z-est(((0 1 1 0) (0 0 2 2)) (((1) (2)) (nil nil) ((0) (0)))):0} as
    follows:
    \begin{equation}\label{eq:z-est(((0 1 1 0) (0 0 2 2)) (((1) (2)) (nil nil) ((0) (0)))):2}
      \sum_{K_0', L_1} \Bigl\langle T  h_{K_0'\times L_1}, \sum_{K_0, L_0} a_{K_0, K_0', L_0} h_{K_0\times L_1} \Bigr\rangle.
    \end{equation}
    By duality, we obtain the subsequent upper estimate for~\eqref{eq:z-est(((0 1 1 0) (0 0 2 2))
      (((1) (2)) (nil nil) ((0) (0)))):2}:
    \begin{equation}\label{eq:z-est(((0 1 1 0) (0 0 2 2)) (((1) (2)) (nil nil) ((0) (0)))):3}
      \sum_{K_0', L_1}\left\|T  h_{K_0'\times L_1}\right\|_{H^p(H^q)} \left\|\sum_{K_0, L_0} a_{K_0, K_0', L_0} h_{K_0\times L_1}\right\|_{(H^p(H^q))^*}.
    \end{equation}
    Estimate~\eqref{eq:z-est(((0 1 1 0) (0 0 2 2)) (((1) (2)) (nil nil) ((0) (0)))):1} and the
    disjointness of the dyadic intervals (see~\textrefp[J]{enu:j1}) yield
    \begin{equation}\label{eq:z-est(((0 1 1 0) (0 0 2 2)) (((1) (2)) (nil nil) ((0) (0)))):4}
      \sum_{K_0', L_1}\|T\| \left\| h_{K_0'\times L_1}\right\|_{H^p(H^q)}\left\| \sum_{K_0, L_0}\max_{K_0}\left(\|T\| |K_0|^{1/{p}} |L_0|^{1/{q}} |K_0'|^{1/{p'}} |L_0|^{1/{q'}}\right) h_{K_0\times L_1}\right\|_{(H^p(H^q))^*}.
    \end{equation}
    Consequently, we obtain
    \begin{equation}\label{eq:z-est(((0 1 1 0) (0 0 2 2)) (((1) (2)) (nil nil) ((0) (0)))):5}
      \sum_{K_0', L_1}\|T\| \left\| h_{K_0'\times L_1}\right\|_{H^p(H^q)}\sum_{L_0}\max_{K_0}\left(\|T\| |K_0|^{1/{p}} |L_0|^{1/{q}} |K_0'|^{1/{p'}} |L_0|^{1/{q'}}\right) \left\| \sum_{K_0}h_{K_0\times L_1}\right\|_{(H^p(H^q))^*}.
    \end{equation}
    Thus,~\eqref{eq:z-est(((0 1 1 0) (0 0 2 2)) (((1) (2)) (nil nil) ((0) (0)))):5} is bounded from
    above by
    \begin{equation}\label{eq:z-est(((0 1 1 0) (0 0 2 2)) (((1) (2)) (nil nil) ((0) (0)))):6}
      \|T\|^2 \max_{K_0}  |K_0|^{1/{p}} \sum_{K_0', L_1} \sum_{L_0} |L_0| |K_0'| |L_1|
      \leq \|T\|^2 \max_{K_0}  |K_0|^{1/{p}}.
    \end{equation}
    Inserting $|K_0|\leq \alpha$ (see~\eqref{eq:alpha-small}) into~\eqref{eq:z-est(((0 1 1 0) (0 0 2
      2)) (((1) (2)) (nil nil) ((0) (0)))):6}, we obtain the estimate
    \begin{equation}\label{eq:z-est(((0 1 1 0) (0 0 2 2)) (((1) (2)) (nil nil) ((0) (0)))):7}
      \|T\|^2 \alpha^{1/{p}}.
    \end{equation}
  \end{proofcase}

  \subsubsection*{Summary of \textref[Case~]{case:z-est(((0 1 1 0) (0 0 2 2)) (((0) (2)) ((1) (0))
      (nil nil))):left} and \textref[Case~]{case:z-est(((0 1 1 0) (0 0 2 2)) (((1) (2)) (nil nil)
      ((0) (0)))):right}}
  Combining~\eqref{eq:z-est(((0 1 1 0) (0 0 2 2)) (((0) (2)) ((1) (0)) (nil nil))):7}
  with~\eqref{eq:z-est(((0 1 1 0) (0 0 2 2)) (((1) (2)) (nil nil) ((0) (0)))):7} yields
  \begin{equation}\label{eq:z-est(((0 1 1 0) (0 0 2 2)) (((0) (2)) ((1) (0)) (nil nil)))(((0 1 1 0)
      (0 0 2 2)) (((1) (2)) (nil nil) ((0) (0))))-final}
    \cond_{\theta,\varepsilon} Z^2
    \leq \|T\|^2 \alpha^{1/2}.
  \end{equation}

  \begin{proofcase}[Case~\theproofcase, group~{\textrefp[f]{enu:proof:lem:var:Z:f:7}}: $K_0 = K_0' \neq K_1 = K_1'$, $L_0 = L_1 \neq L_0' = L_1'$
    (((2) (1)) ((0) (0)) (nil nil)) -- left variant]\label{case:z-est(((0 0 2 2) (0 1 0 1)) (((2)
      (1)) ((0) (0)) (nil nil))):left}
    In this case, we have to estimate
    \begin{equation}\label{eq:z-est(((0 0 2 2) (0 1 0 1)) (((2) (1)) ((0) (0)) (nil nil))):0}
      \sum_{K_1, L_0', K_0, L_0} \langle Th_{K_0\times L_0}, h_{K_0\times L_0'} \rangle \langle Th_{K_1\times L_0}, h_{K_1\times L_0'} \rangle.
    \end{equation}
    We put $a_{K_0, L_0, L_0'} = \langle T h_{K_0\times L_0}, h_{K_0\times L_0'} \rangle$ and note
    the estimate
    \begin{equation}\label{eq:z-est(((0 0 2 2) (0 1 0 1)) (((2) (1)) ((0) (0)) (nil nil))):1}
      |a_{K_0, L_0, L_0'}|
      \leq \|T\| |K_0|^{1/{p}} |L_0|^{1/{q}} |K_0|^{1/{p'}} |L_0'|^{1/{q'}}.
    \end{equation}
    Now, we write~\eqref{eq:z-est(((0 0 2 2) (0 1 0 1)) (((2) (1)) ((0) (0)) (nil nil))):0} as
    follows:
    \begin{equation}\label{eq:z-est(((0 0 2 2) (0 1 0 1)) (((2) (1)) ((0) (0)) (nil nil))):2}
      \sum_{K_1, L_0'} \Bigl\langle T \sum_{K_0, L_0} a_{K_0, L_0, L_0'} h_{K_1\times L_0},  h_{K_1\times L_0'} \Bigr\rangle.
    \end{equation}
    By duality, we obtain the subsequent upper estimate for~\eqref{eq:z-est(((0 0 2 2) (0 1 0 1))
      (((2) (1)) ((0) (0)) (nil nil))):2}:
    \begin{equation}\label{eq:z-est(((0 0 2 2) (0 1 0 1)) (((2) (1)) ((0) (0)) (nil nil))):3}
      \sum_{K_1, L_0'}\left\|T \sum_{K_0, L_0} a_{K_0, L_0, L_0'} h_{K_1\times L_0}\right\|_{H^p(H^q)} \left\| h_{K_1\times L_0'}\right\|_{(H^p(H^q))^*}.
    \end{equation}
    Estimate~\eqref{eq:z-est(((0 0 2 2) (0 1 0 1)) (((2) (1)) ((0) (0)) (nil nil))):1} and the
    disjointness of the dyadic intervals (see~\textrefp[J]{enu:j1}) yield
    \begin{equation}\label{eq:z-est(((0 0 2 2) (0 1 0 1)) (((2) (1)) ((0) (0)) (nil nil))):4}
      \sum_{K_1, L_0'}\|T\| \left\| \sum_{K_0, L_0}\max_{L_0}\left(\|T\| |K_0|^{1/{p}} |L_0|^{1/{q}} |K_0|^{1/{p'}} |L_0'|^{1/{q'}}\right) h_{K_1\times L_0}\right\|_{H^p(H^q)}\left\| h_{K_1\times L_0'}\right\|_{(H^p(H^q))^*}.
    \end{equation}
    Consequently, we obtain
    \begin{equation}\label{eq:z-est(((0 0 2 2) (0 1 0 1)) (((2) (1)) ((0) (0)) (nil nil))):5}
      \sum_{K_1, L_0'}\|T\| \sum_{K_0}\max_{L_0}\left(\|T\| |K_0|^{1/{p}} |L_0|^{1/{q}} |K_0|^{1/{p'}} |L_0'|^{1/{q'}}\right) \left\| \sum_{L_0}h_{K_1\times L_0}\right\|_{H^p(H^q)}\left\| h_{K_1\times L_0'}\right\|_{(H^p(H^q))^*}.
    \end{equation}
    Thus,~\eqref{eq:z-est(((0 0 2 2) (0 1 0 1)) (((2) (1)) ((0) (0)) (nil nil))):5} is bounded from
    above by
    \begin{equation*}
      \|T\|^2 \max_{L_0}  |L_0|^{1/{q}} \sum_{K_1, L_0'} \sum_{K_0}  |K_0| |L_0'|^{2/{q'}} |K_1|.
    \end{equation*}
    Using Hölder's inequality yields
    \begin{equation}\label{eq:z-est(((0 0 2 2) (0 1 0 1)) (((2) (1)) ((0) (0)) (nil nil))):6}
      \|T\|^2 \max_{L_0, L_0'}  |L_0|^{1/{q}} |L_0'|^{2/{q'}-1}.
    \end{equation}
    Inserting $|L_0|, |L_0'|\leq \alpha$ (see~\eqref{eq:alpha-small}) into~\eqref{eq:z-est(((0 0 2
      2) (0 1 0 1)) (((2) (1)) ((0) (0)) (nil nil))):6}, we obtain the estimate
    \begin{equation}\label{eq:z-est(((0 0 2 2) (0 1 0 1)) (((2) (1)) ((0) (0)) (nil nil))):7}
      \|T\|^2 \alpha^{1/{q'}}.
    \end{equation}
  \end{proofcase}

  \begin{proofcase}[Case~\theproofcase, group~{\textrefp[f]{enu:proof:lem:var:Z:f:7}}: $K_0 = K_0' \neq K_1 = K_1'$, $L_0 = L_1 \neq L_0' = L_1'$ (((2) (0)) (nil nil) ((0) (1))) -- right variant]\label{case:z-est(((0 0 2 2) (0 1 0 1)) (((2) (0)) (nil nil) ((0) (1)))):right}
    In this case, we have to estimate
    \begin{equation}\label{eq:z-est(((0 0 2 2) (0 1 0 1)) (((2) (0)) (nil nil) ((0) (1)))):0}
      \sum_{K_1, L_0, K_0, L_0'} \langle Th_{K_0\times L_0}, h_{K_0\times L_0'} \rangle \langle Th_{K_1\times L_0}, h_{K_1\times L_0'} \rangle.
    \end{equation}
    We put $a_{K_0, L_0, L_0'} = \langle T h_{K_0\times L_0}, h_{K_0\times L_0'} \rangle$ and note
    the estimate
    \begin{equation}\label{eq:z-est(((0 0 2 2) (0 1 0 1)) (((2) (0)) (nil nil) ((0) (1)))):1}
      |a_{K_0, L_0, L_0'}|
      \leq \|T\| |K_0|^{1/{p}} |L_0|^{1/{q}} |K_0|^{1/{p'}} |L_0'|^{1/{q'}}.
    \end{equation}
    Now, we write~\eqref{eq:z-est(((0 0 2 2) (0 1 0 1)) (((2) (0)) (nil nil) ((0) (1)))):0} as
    follows:
    \begin{equation}\label{eq:z-est(((0 0 2 2) (0 1 0 1)) (((2) (0)) (nil nil) ((0) (1)))):2}
      \sum_{K_1, L_0} \Bigl\langle T  h_{K_1\times L_0}, \sum_{K_0, L_0'} a_{K_0, L_0, L_0'} h_{K_1\times L_0'} \Bigr\rangle.
    \end{equation}
    By duality, we obtain the subsequent upper estimate for~\eqref{eq:z-est(((0 0 2 2) (0 1 0 1))
      (((2) (0)) (nil nil) ((0) (1)))):2}:
    \begin{equation}\label{eq:z-est(((0 0 2 2) (0 1 0 1)) (((2) (0)) (nil nil) ((0) (1)))):3}
      \sum_{K_1, L_0}\left\|T  h_{K_1\times L_0}\right\|_{H^p(H^q)} \left\|\sum_{K_0, L_0'} a_{K_0, L_0, L_0'} h_{K_1\times L_0'}\right\|_{(H^p(H^q))^*}.
    \end{equation}
    Estimate~\eqref{eq:z-est(((0 0 2 2) (0 1 0 1)) (((2) (0)) (nil nil) ((0) (1)))):1} and the
    disjointness of the dyadic intervals (see~\textrefp[J]{enu:j1}) yield
    \begin{equation}\label{eq:z-est(((0 0 2 2) (0 1 0 1)) (((2) (0)) (nil nil) ((0) (1)))):4}
      \sum_{K_1, L_0}\|T\| \left\| h_{K_1\times L_0}\right\|_{H^p(H^q)}\left\| \sum_{K_0, L_0'}\max_{L_0'}\left(\|T\| |K_0|^{1/{p}} |L_0|^{1/{q}} |K_0|^{1/{p'}} |L_0'|^{1/{q'}}\right) h_{K_1\times L_0'}\right\|_{(H^p(H^q))^*}.
    \end{equation}
    Consequently, we obtain
    \begin{equation}\label{eq:z-est(((0 0 2 2) (0 1 0 1)) (((2) (0)) (nil nil) ((0) (1)))):5}
      \sum_{K_1, L_0}\|T\| \left\| h_{K_1\times L_0}\right\|_{H^p(H^q)}\sum_{K_0}\max_{L_0'}\left(\|T\| |K_0|^{1/{p}} |L_0|^{1/{q}} |K_0|^{1/{p'}} |L_0'|^{1/{q'}}\right) \left\| \sum_{L_0'}h_{K_1\times L_0'}\right\|_{(H^p(H^q))^*}.
    \end{equation}
    Thus,~\eqref{eq:z-est(((0 0 2 2) (0 1 0 1)) (((2) (0)) (nil nil) ((0) (1)))):5} is bounded from
    above by
    \begin{equation*}
      \|T\|^2 \max_{L_0'} |L_0'|^{1/{q'}} \sum_{K_1, L_0} \sum_{K_0} |K_0| |L_0|^{2/{q}} |K_1|.
    \end{equation*}
    Using Hölder's inequality yields
    \begin{equation}\label{eq:z-est(((0 0 2 2) (0 1 0 1)) (((2) (0)) (nil nil) ((0) (1)))):6}
      \|T\|^2 \max_{L_0'} |L_0'|^{1/{q'}} |L_0|^{2/{q}-1}.
    \end{equation}
    Inserting $|L_0|, |L_0'|\leq \alpha$ (see~\eqref{eq:alpha-small}) into~\eqref{eq:z-est(((0 0 2
      2) (0 1 0 1)) (((2) (0)) (nil nil) ((0) (1)))):6}, we obtain the estimate
    \begin{equation}\label{eq:z-est(((0 0 2 2) (0 1 0 1)) (((2) (0)) (nil nil) ((0) (1)))):7}
      \|T\|^2 \alpha^{1/{q}}.
    \end{equation}
  \end{proofcase}

  \subsubsection*{Summary of \textref[Case~]{case:z-est(((0 0 2 2) (0 1 0 1)) (((2) (1)) ((0) (0))
      (nil nil))):left} and \textref[Case~]{case:z-est(((0 0 2 2) (0 1 0 1)) (((2) (0)) (nil nil)
      ((0) (1)))):right}}
  Combining~\eqref{eq:z-est(((0 0 2 2) (0 1 0 1)) (((2) (1)) ((0) (0)) (nil nil))):7}
  with~\eqref{eq:z-est(((0 0 2 2) (0 1 0 1)) (((2) (0)) (nil nil) ((0) (1)))):7} yields
  \begin{equation}\label{eq:z-est(((0 0 2 2) (0 1 0 1)) (((2) (1)) ((0) (0)) (nil nil)))(((0 0 2 2)
      (0 1 0 1)) (((2) (0)) (nil nil) ((0) (1))))-final}
    \cond_{\theta,\varepsilon} Z^2
    \leq \|T\|^2 \alpha^{1/2}.
  \end{equation}

  \begin{proofcase}[Case~\theproofcase, group~{\textrefp[f]{enu:proof:lem:var:Z:f:8}}: $K_0 = K_0' \neq K_1 = K_1'$, $L_0 = L_1' \neq L_0' = L_1$
    (((2) (0)) ((0) (1)) (nil nil)) -- left variant]\label{case:z-est(((0 0 2 2) (0 1 1 0)) (((2)
      (0)) ((0) (1)) (nil nil))):left}
    In this case, we have to estimate
    \begin{equation}\label{eq:z-est(((0 0 2 2) (0 1 1 0)) (((2) (0)) ((0) (1)) (nil nil))):0}
      \sum_{K_1, L_0, K_0, L_0'} \langle Th_{K_0\times L_0}, h_{K_0\times L_0'} \rangle \langle Th_{K_1\times L_0'}, h_{K_1\times L_0} \rangle.
    \end{equation}
    We put $a_{K_0, L_0, L_0'} = \langle T h_{K_0\times L_0}, h_{K_0\times L_0'} \rangle$ and note
    the estimate
    \begin{equation}\label{eq:z-est(((0 0 2 2) (0 1 1 0)) (((2) (0)) ((0) (1)) (nil nil))):1}
      |a_{K_0, L_0, L_0'}|
      \leq \|T\| |K_0|^{1/{p}} |L_0|^{1/{q}} |K_0|^{1/{p'}} |L_0'|^{1/{q'}}.
    \end{equation}
    Now, we write~\eqref{eq:z-est(((0 0 2 2) (0 1 1 0)) (((2) (0)) ((0) (1)) (nil nil))):0} as
    follows:
    \begin{equation}\label{eq:z-est(((0 0 2 2) (0 1 1 0)) (((2) (0)) ((0) (1)) (nil nil))):2}
      \sum_{K_1, L_0} \Bigl\langle T \sum_{K_0, L_0'} a_{K_0, L_0, L_0'} h_{K_1\times L_0'},  h_{K_1\times L_0} \Bigr\rangle.
    \end{equation}
    By duality, we obtain the subsequent upper estimate for~\eqref{eq:z-est(((0 0 2 2) (0 1 1 0))
      (((2) (0)) ((0) (1)) (nil nil))):2}:
    \begin{equation}\label{eq:z-est(((0 0 2 2) (0 1 1 0)) (((2) (0)) ((0) (1)) (nil nil))):3}
      \sum_{K_1, L_0}\left\|T \sum_{K_0, L_0'} a_{K_0, L_0, L_0'} h_{K_1\times L_0'}\right\|_{H^p(H^q)} \left\| h_{K_1\times L_0}\right\|_{(H^p(H^q))^*}.
    \end{equation}
    Estimate~\eqref{eq:z-est(((0 0 2 2) (0 1 1 0)) (((2) (0)) ((0) (1)) (nil nil))):1} and the
    disjointness of the dyadic intervals (see~\textrefp[J]{enu:j1}) yield
    \begin{equation}\label{eq:z-est(((0 0 2 2) (0 1 1 0)) (((2) (0)) ((0) (1)) (nil nil))):4}
      \sum_{K_1, L_0}\|T\| \left\| \sum_{K_0, L_0'}\max_{L_0'}\left(\|T\| |K_0|^{1/{p}} |L_0|^{1/{q}} |K_0|^{1/{p'}} |L_0'|^{1/{q'}}\right) h_{K_1\times L_0'}\right\|_{H^p(H^q)}\left\| h_{K_1\times L_0}\right\|_{(H^p(H^q))^*}.
    \end{equation}
    Consequently, we obtain
    \begin{equation}\label{eq:z-est(((0 0 2 2) (0 1 1 0)) (((2) (0)) ((0) (1)) (nil nil))):5}
      \sum_{K_1, L_0}\|T\| \sum_{K_0}\max_{L_0'}\left(\|T\| |K_0|^{1/{p}} |L_0|^{1/{q}} |K_0|^{1/{p'}} |L_0'|^{1/{q'}}\right) \left\| \sum_{L_0'}h_{K_1\times L_0'}\right\|_{H^p(H^q)}\left\| h_{K_1\times L_0}\right\|_{(H^p(H^q))^*}.
    \end{equation}
    Thus,~\eqref{eq:z-est(((0 0 2 2) (0 1 1 0)) (((2) (0)) ((0) (1)) (nil nil))):5} is bounded from
    above by
    \begin{equation}\label{eq:z-est(((0 0 2 2) (0 1 1 0)) (((2) (0)) ((0) (1)) (nil nil))):6}
      \|T\|^2 \max_{L_0'} |L_0'|^{1/{q'}} \sum_{K_1, L_0} \sum_{K_0} |K_0| |L_0| |K_1|
      \leq \|T\|^2 \max_{L_0'} |L_0'|^{1/{q'}}.
    \end{equation}
    Inserting $|L_0'|\leq \alpha$ (see~\eqref{eq:alpha-small}) into~\eqref{eq:z-est(((0 0 2 2) (0 1
      1 0)) (((2) (0)) ((0) (1)) (nil nil))):6}, we obtain the estimate
    \begin{equation}\label{eq:z-est(((0 0 2 2) (0 1 1 0)) (((2) (0)) ((0) (1)) (nil nil))):7}
      \|T\|^2 \alpha^{1/{q'}}.
    \end{equation}
  \end{proofcase}

  \begin{proofcase}[Case~\theproofcase, group~{\textrefp[f]{enu:proof:lem:var:Z:f:8}}: $K_0 = K_0' \neq K_1 = K_1'$, $L_0 = L_1' \neq L_0' = L_1$ (((2) (1)) (nil nil) ((0) (0))) -- right variant]\label{case:z-est(((0 0 2 2) (0 1 1 0)) (((2) (1)) (nil nil) ((0) (0)))):right}
    In this case, we have to estimate
    \begin{equation}\label{eq:z-est(((0 0 2 2) (0 1 1 0)) (((2) (1)) (nil nil) ((0) (0)))):0}
      \sum_{K_1, L_0', K_0, L_0} \langle Th_{K_0\times L_0}, h_{K_0\times L_0'} \rangle \langle Th_{K_1\times L_0'}, h_{K_1\times L_0} \rangle.
    \end{equation}
    We put $a_{K_0, L_0, L_0'} = \langle T h_{K_0\times L_0}, h_{K_0\times L_0'} \rangle$ and note
    the estimate
    \begin{equation}\label{eq:z-est(((0 0 2 2) (0 1 1 0)) (((2) (1)) (nil nil) ((0) (0)))):1}
      |a_{K_0, L_0, L_0'}|
      \leq \|T\| |K_0|^{1/{p}} |L_0|^{1/{q}} |K_0|^{1/{p'}} |L_0'|^{1/{q'}}.
    \end{equation}
    Now, we write~\eqref{eq:z-est(((0 0 2 2) (0 1 1 0)) (((2) (1)) (nil nil) ((0) (0)))):0} as
    follows:
    \begin{equation}\label{eq:z-est(((0 0 2 2) (0 1 1 0)) (((2) (1)) (nil nil) ((0) (0)))):2}
      \sum_{K_1, L_0'} \Bigl\langle T  h_{K_1\times L_0'}, \sum_{K_0, L_0} a_{K_0, L_0, L_0'} h_{K_1\times L_0} \Bigr\rangle.
    \end{equation}
    By duality, we obtain the subsequent upper estimate for~\eqref{eq:z-est(((0 0 2 2) (0 1 1 0))
      (((2) (1)) (nil nil) ((0) (0)))):2}:
    \begin{equation}\label{eq:z-est(((0 0 2 2) (0 1 1 0)) (((2) (1)) (nil nil) ((0) (0)))):3}
      \sum_{K_1, L_0'}\left\|T  h_{K_1\times L_0'}\right\|_{H^p(H^q)} \left\|\sum_{K_0, L_0} a_{K_0, L_0, L_0'} h_{K_1\times L_0}\right\|_{(H^p(H^q))^*}.
    \end{equation}
    Estimate~\eqref{eq:z-est(((0 0 2 2) (0 1 1 0)) (((2) (1)) (nil nil) ((0) (0)))):1} and the
    disjointness of the dyadic intervals (see~\textrefp[J]{enu:j1}) yield
    \begin{equation}\label{eq:z-est(((0 0 2 2) (0 1 1 0)) (((2) (1)) (nil nil) ((0) (0)))):4}
      \sum_{K_1, L_0'}\|T\| \left\| h_{K_1\times L_0'}\right\|_{H^p(H^q)}\left\| \sum_{K_0, L_0}\max_{L_0}\left(\|T\| |K_0|^{1/{p}} |L_0|^{1/{q}} |K_0|^{1/{p'}} |L_0'|^{1/{q'}}\right) h_{K_1\times L_0}\right\|_{(H^p(H^q))^*}.
    \end{equation}
    Consequently, we obtain
    \begin{equation}\label{eq:z-est(((0 0 2 2) (0 1 1 0)) (((2) (1)) (nil nil) ((0) (0)))):5}
      \sum_{K_1, L_0'}\|T\| \left\| h_{K_1\times L_0'}\right\|_{H^p(H^q)}\sum_{K_0}\max_{L_0}\left(\|T\| |K_0|^{1/{p}} |L_0|^{1/{q}} |K_0|^{1/{p'}} |L_0'|^{1/{q'}}\right) \left\| \sum_{L_0}h_{K_1\times L_0}\right\|_{(H^p(H^q))^*}.
    \end{equation}
    Thus,~\eqref{eq:z-est(((0 0 2 2) (0 1 1 0)) (((2) (1)) (nil nil) ((0) (0)))):5} is bounded from
    above by
    \begin{equation}\label{eq:z-est(((0 0 2 2) (0 1 1 0)) (((2) (1)) (nil nil) ((0) (0)))):6}
      \|T\|^2 \max_{L_0}  |L_0|^{1/{q}} \sum_{K_1, L_0'} \sum_{K_0} |K_1| |K_0| |L_0'|
      \leq \|T\|^2 \max_{L_0}  |L_0|^{1/{q}}.
    \end{equation}
    Inserting $|L_0|\leq \alpha$ (see~\eqref{eq:alpha-small}) into~\eqref{eq:z-est(((0 0 2 2) (0 1 1
      0)) (((2) (1)) (nil nil) ((0) (0)))):6}, we obtain the estimate
    \begin{equation}\label{eq:z-est(((0 0 2 2) (0 1 1 0)) (((2) (1)) (nil nil) ((0) (0)))):7}
      \|T\|^2 \alpha^{1/{q}}.
    \end{equation}
  \end{proofcase}

  \subsubsection*{Summary of \textref[Case~]{case:z-est(((0 0 2 2) (0 1 1 0)) (((2) (0)) ((0) (1))
      (nil nil))):left} and \textref[Case~]{case:z-est(((0 0 2 2) (0 1 1 0)) (((2) (1)) (nil nil)
      ((0) (0)))):right}}
  Combining~\eqref{eq:z-est(((0 0 2 2) (0 1 1 0)) (((2) (0)) ((0) (1)) (nil nil))):7}
  with~\eqref{eq:z-est(((0 0 2 2) (0 1 1 0)) (((2) (1)) (nil nil) ((0) (0)))):7} yields
  \begin{equation}\label{eq:z-est(((0 0 2 2) (0 1 1 0)) (((2) (0)) ((0) (1)) (nil nil)))(((0 0 2 2)
      (0 1 1 0)) (((2) (1)) (nil nil) ((0) (0))))-final}
    \cond_{\theta,\varepsilon} Z^2
    \leq \|T\|^2 \alpha^{1/2}.
  \end{equation}

  \subsubsection*{Summary for $Z$}
  
  Combining~\eqref{eq:z-est(((0 0 0 0) (0 1 0 1)) (((0) (1)) (nil (0)) (nil nil))):7},
  \eqref{eq:z-est(((0 0 0 0) (0 1 1 0)) (((0) (0)) (nil (1)) (nil nil))):7}, \eqref{eq:z-est(((0 1 0
    1) (0 0 0 0)) (((1) (0)) ((0) nil) (nil nil))):7}, \eqref{eq:z-est(((0 1 0 1) (0 1 0 1)) (((1)
    (1)) ((0) (0)) (nil nil)))(((0 1 0 1) (0 1 0 1)) (((0) (0)) (nil nil) ((1) (1))))-final},
  \eqref{eq:z-est(((0 1 0 1) (0 1 1 0)) (((0) (1)) (nil nil) ((1) (0))))(((0 1 0 1) (0 1 1 0)) (((1)
    (0)) ((0) (1)) (nil nil)))-final}, \eqref{eq:z-est(((0 1 0 1) (0 0 2 2)) (((1) (2)) ((0) (0))
    (nil nil)))(((0 1 0 1) (0 0 2 2)) (((0) (2)) (nil nil) ((1) (0))))-final}, \eqref{eq:z-est(((0 1
    1 0) (0 0 0 0)) (((0) (0)) ((1) nil) (nil nil))):7}, \eqref{eq:z-est(((0 1 1 0) (0 1 0 1)) (((0)
    (1)) ((1) (0)) (nil nil)))(((0 1 1 0) (0 1 0 1)) (((1) (0)) (nil nil) ((0) (1))))-final},
  \eqref{eq:z-est(((0 1 1 0) (0 1 1 0)) (((0) (0)) ((1) (1)) (nil nil)))(((0 1 1 0) (0 1 1 0)) (((1)
    (1)) (nil nil) ((0) (0))))-final}, \eqref{eq:z-est(((0 1 1 0) (0 0 2 2)) (((0) (2)) ((1) (0))
    (nil nil)))(((0 1 1 0) (0 0 2 2)) (((1) (2)) (nil nil) ((0) (0))))-final}, \eqref{eq:z-est(((0 0
    2 2) (0 1 0 1)) (((2) (1)) ((0) (0)) (nil nil)))(((0 0 2 2) (0 1 0 1)) (((2) (0)) (nil nil) ((0)
    (1))))-final} and~\eqref{eq:z-est(((0 0 2 2) (0 1 1 0)) (((2) (0)) ((0) (1)) (nil nil)))(((0 0 2
    2) (0 1 1 0)) (((2) (1)) (nil nil) ((0) (0))))-final} yields
  \begin{equation}\label{eq:z-est:final}
    \cond_{\theta,\varepsilon} Z^2
    \leq 12 \|T\|^2 \alpha^{1/2}.
  \end{equation}
\end{myproof}


\subsection*{Acknowledgments}\hfill\\
\noindent
It is my pleasure to thank P.F.X.~Müller for many helpful discussions.  Supported by the Austrian
Science Foundation (FWF) Pr.Nr. P28352.


\bibliographystyle{abbrv}
\bibliography{bibliography}

\end{document}